\documentclass[a4paper,11pt]{amsart}
\usepackage[all]{xy}
\usepackage{amssymb}
\usepackage{amsopn}
\usepackage{amsmath}

\newcommand{\ie}{{\em i.e.}\ }
\newcommand{\cf}{{\em cf.}\ }
\newcommand{\eg}{{\em e.g.}\ }
\newcommand{\ko}{\: , \;}

\numberwithin{equation}{subsection}
\newtheorem{theorem}[subsection]{Theorem}
\newtheorem{classification-theorem}[subsection]{Classification Theorem}
\newtheorem{decomposition-theorem}[subsection]{Decomposition Theorem}
\newtheorem{proposition-definition}[subsection]{Proposition-Definition}
\newtheorem{periodicity-conjecture}[subsection]{Periodicity Conjecture}
\newtheorem{lemma}[subsection]{Lemma}
\newtheorem{proposition}[subsection]{Proposition}
\newtheorem{corollary}[subsection]{Corollary}

\newtheorem{example}[subsection]{Example}
\newtheorem{remark}[subsection]{Remark}

\newcommand{\reminder}[1]{}

\newcommand{\opname}[1]{\operatorname{\mathsf{#1}}}

\renewcommand{\mod}{\opname{mod}\nolimits}

\newcommand{\proj}{\opname{proj}\nolimits}
\newcommand{\inj}{\opname{inj}\nolimits}

\newcommand{\ind}{\opname{ind}\nolimits}
\newcommand{\per}{\opname{per}\nolimits}
\newcommand{\add}{\opname{add}\nolimits}

\newcommand{\Gr}{\opname{Gr}\nolimits}

\newcommand{\Za}{\opname{Za}\nolimits}
\newcommand{\hrz}{\opname{hrz}\nolimits}
\newcommand{\vrt}{\opname{vrt}\nolimits}
\newcommand{\Trop}{\opname{Trop}\nolimits}

\renewcommand{\ker}{\opname{ker}\nolimits}
\newcommand{\obj}{\opname{obj}\nolimits}

\newcommand{\Z}{\mathbb{Z}}
\newcommand{\N}{\mathbb{N}}
\newcommand{\Q}{\mathbb{Q}}
\newcommand{\C}{\mathbb{C}}

\newcommand{\iso}{\stackrel{_\sim}{\rightarrow}}
\newcommand{\id}{\mathbf{1}}

\newcommand{\Hom}{\opname{Hom}}

\newcommand{\RHom}{\opname{RHom}}

\newcommand{\Ext}{\opname{Ext}}

\newcommand{\End}{\opname{End}}

\newcommand{\boxten}{\boxtimes}
\newcommand{\ten}{\otimes}
\newcommand{\lten}{\overset{\boldmath{L}}{\ten}}

\newcommand{\Tor}{\opname{Tor}}

\newcommand{\hh}{\opname{HH}}
\newcommand{\ca}{{\mathcal A}}

\newcommand{\cc}{{\mathcal C}}
\newcommand{\cd}{{\mathcal D}}

\newcommand{\cp}{{\mathcal P}}

\newcommand{\ct}{{\mathcal T}}

\newcommand{\eps}{\varepsilon}
\renewcommand{\phi}{\varphi}
\newcommand{\del}{\partial}

\renewcommand{\hat}[1]{\widehat{#1}}

\newcommand{\bt}{\bullet}

\newcommand{\sgn}{\mbox{sgn}}

\renewcommand{\tilde}[1]{\widetilde{#1}}

\newcommand{\arr}[1]{\stackrel{#1}{\rightarrow}}

\begin{document}

\date{December 2009, last modified on \today}

\title[The periodicity conjecture]{The periodicity conjecture \\for pairs of Dynkin diagrams}
\author{Bernhard Keller}
\address{Universit\'e Paris Diderot -- Paris 7\\
	Institut Universitaire de France\\
    UFR de Math\'ematiques\\
   Institut de Math\'ematiques de Jussieu, UMR 7586 du CNRS \\
   Case 7012\\
    B\^{a}timent Chevaleret\\
    75205 Paris Cedex 13\\
    France
}
\email{keller@math.jussieu.fr}

\begin{abstract}
  We prove the periodicity conjecture for pairs of Dynkin diagrams
  using Fomin-Zelevinsky's cluster algebras and their
  (additive) categorification via triangulated categories.
\end{abstract}

\maketitle
\tableofcontents

\section{Introduction}
\label{s:introduction}

\subsection{The periodicity conjecture}
The {\em $Y$-system} associated with a pair $(\Delta, \Delta')$
of Dynkin diagrams is a certain infinite system of algebraic
recurrence equations. The {\em periodicity conjecture} asserts
that all solutions to this system are periodic of period
dividing the double of the sum of the Coxeter numbers of $\Delta$ and $\Delta'$.
The conjecture was first formulated by Al.~B.~Zamolodchikov \cite{Zamolodchikov91}
(for $(\Delta, A_1)$, where $\Delta$ is simply laced) in his
study of the thermodynamic Bethe ansatz. We refer to
\cite{FrenkelSzenes95} for an introduction to the significance
of the conjecture in physics and to \cite{FrenkelSzenes95}
\cite{Chapoton05} \cite{FockGoncharov09} \cite{Nakanishi11c}
for its applications to dilogarithm identities.
Zamolodchikov's original conjecture was subsequently generalized
\begin{itemize}
\item by Ravanini-Valleriani-Tateo to $(\Delta, \Delta')$, where $\Delta$
and $\Delta'$ are simply laced or `tadpoles' \cite[(6.2)]{RavaniniTateoValleriani93};

\item by Kuniba-Nakanishi to $(\Delta, A_n)$, where $\Delta$ is not
necessarily simply laced \cite[(2a)]{KunibaNakanishi92}, see also
Kuniba-Nakanishi-Suzuki \cite[B.6]{KunibaNakanishiSuzuki94}.

\end{itemize}

\noindent The conjecture was proved

\begin{itemize}
\item for $(A_n, A_1)$  by Frenkel-Szenes \cite{FrenkelSzenes95}
(who produced explicit solutions) and independently by Gliozzi-Tateo
\cite{GliozziTateo96} (via volumes of threefolds computed using
triangulations),

\item by Fomin-Zelevinsky \cite{FominZelevinsky03b} for $(\Delta,
A_1)$, where $\Delta$ is not necessarily simply laced (via the
philosophy of cluster algebras and a computer check for the
exceptional types; a uniform proof can now be given
using \cite{YangZelevinsky08}),

\item for $(A_n, A_m)$ by Volkov \cite{Volkov07}, who exhibited explicit
solutions using cross ratios, and by Szenes \cite{Szenes09}, who interpreted
the system as a system of flat connections on a graph; an equivalent statement
was proved by Henriques \cite{Henriques07}.
\end{itemize}
An original approach via representations of quantum affine algebras
is due to Hernandez-Leclerc \cite{HernandezLeclerc10}.
They treat the case $(A_n, A_1)$ and obtain formulas for
solutions in terms of $q$-characters.

In \cite{Keller08c}, we announced a proof of the conjecture in the
general case and gave an outline of the proof. In this article,
we provide the detailed proof. Notice that for the case $(\Delta, A_n)$, where $\Delta$ is
not simply laced, there are two variants of the conjecture: the
first one was stated by Kuniba-Nakanishi \cite{KunibaNakanishi92} and
involves {\em dual Coxeter numbers} (\cf \eg Chapter~6 of \cite{Kac90}).
It is proved in the papers \cite{InoueIyamaKellerKunibaNakanishi10a} and
\cite{InoueIyamaKellerKunibaNakanishi10b}. The second one is due to
Fomin-Zelevinsky \cite{FominZelevinsky03b} and is
proved at the end of this article.
The conjecture has a refined version (half-periodicity) and
an analogue for the so-called $T$-system. Both were proved in
\cite{InoueIyamaKunibaNakanishiSuzuki10} in the simply laced case
by adapting the method of \cite{Keller08c}.

\subsection{On the proof} The proof we propose is based on a reformulation of
the conjecture in terms of Fomin-Zelevinsky's cluster algebras
\cite{FominZelevinsky02} \cite{FominZelevinsky03} \cite{BerensteinFominZelevinsky05}
\cite{FominZelevinsky07} \cite{Zelevinsky07a} \cite{Fomin07}
and on the recent theory linking cluster algebras to representations
of quivers (with relations).

The link between the conjecture and cluster algebras goes back to
Fomin-Zelevinsky's fundamental work \cite{FominZelevinsky03b}. Here and in
\cite{FominZelevinsky07}, they showed how the $Y$-system appearing
in the conjecture is controlled by the evolution of $Y$-seeds in
the bipartite belt of a cluster algebra. Let us point out that recently, other
discrete dynamical systems have been linked to cluster algebras,
in particular the $T$-system \cite{InoueIyamaKunibaNakanishiSuzuki10}
and the $Q$-system \cite{Kedem08} \cite{DiFrancescoKedem09}.

The theory linking cluster algebras to
quiver representations was initiated by Marsh-Reineke-Zelevinsky
\cite{MarshReinekeZelevinsky03} and subsequently developed in
several variants and by many authors, \cf for example the surveys
\cite{BuanMarsh06} \cite{GeissLeclercSchroeer08a} \cite{Keller09b}
\cite{Reiten10} \cite{Ringel07}. It is related to
Kontsevich-Soibelman's interpretation of cluster transformations in
their theory of Donaldson-Thomas invariants
\cite{KontsevichSoibelman08} \cite{KontsevichSoibelman10a} and, in
fact, it provided one of the starting points for Kontsevich-Soibelman's
work.

For our purposes, we use the {\em cluster categories} first introduced in
\cite{BuanMarshReinekeReitenTodorov06} (for general quivers) and
independently in \cite{CalderoChapotonSchiffler06} (for quivers of
type $A_n$). These are certain triangulated categories which are
$2$-Calabi-Yau and contain a distinguished object with remarkable
properties (a cluster tilting object). Originally, the cluster
category $\cc_A$ was defined for algebras $A$ of global dimension at
most one; however, in recent work, Amiot \cite{Amiot09} has extended
the construction to many algebras of global dimension at most $2$.
We show that Amiot's results apply in particular to tensor products
$A=kQ\ten kQ'$ of path algebras of quivers obtained by orienting the
given Dynkin diagrams $\Delta$ and $\Delta'$. Using Palu's
\cite{Palu08a} generalization of the Caldero-Chapoton map
\cite{CalderoChapoton06} we show that the resulting cluster category
$\cc_A$ is indeed an (additive) categorification of the cluster
algebra which controls the $Y$-system associated with
$(\Delta,\Delta')$. In the categorification $\cc_A$, it is easy to
write down an autoequivalence, the {\em Zamolodchikov
transformation}, whose powers provide the solutions to the
$Y$-system. We conclude by showing that the Zamolodchikov
transformation is of order dividing the sum of the Coxeter numbers
of $\Delta$ and $\Delta'$.

This proof is effective in the sense that in prinicple,
it yields explicit (periodic) formulas for the general
solution of the $Y$-system. It is conceptual in the sense
that the validity of the conjecture is obtained from
a categorical periodicity theorem, which in turn follows from
classical results of Gabriel and Happel. For $\Delta'=A_1$,
the proof specializes to a new proof of the case due to
Fomin-Zelevinsky \cite{FominZelevinsky03b}.

\subsection{Contents} In section~\ref{s:the-conjecture}, we state
the conjecture for pairs of Dynkin diagrams $(\Delta, \Delta')$ which may be multiply
laced. We present the plan of the proof in section~\ref{ss:plan-of-the-proof}.
We treat the case of simply laced Dynkin diagrams in sections
\ref{s:reformulation} to \ref{s:categorical-periodicity}.
In section~\ref{s:reformulation}, we recall fundamental constructions
from the theory of cluster algebras: quiver mutation and the
mutation of $Y$-seeds. Then we introduce three types of products
of quivers (section~\ref{ss:products-of-quivers})
and reformulate the conjecture as a statement
about the periodicity of a sequence of mutations
$\mu_\boxtimes$ applied to the initial $Y$-seed associated
with the triangle product $Q\boxtimes Q'$ of two alternating
quivers with underlying graphs $\Delta$ and $\Delta'$ (section~\ref{ss:reformulation}).
Section~\ref{s:more-cluster-combinatorics} is devoted to
tropical $Y$-variables and $F$-polynomials: we recall how they
are constructed and how they can be used to express the
non tropical $Y$-variables (section~\ref{ss:trop-Y-var-F-pol});
we also introduce $g$-vectors and cluster variables
(section~\ref{ss:cluster-variables-g-vectors}), which are
useful in the application of our categorical model to the so-called
$T$-systems.

In section~\ref{s:Calabi-Yau-triangulated-categories}, we recall the
notions of $2$-Calabi-Yau triangulated category $\cc$ and of a cluster tilting
object $T$ in such a category. There is a canonical quiver associated with
the datum of $\cc$ and $T$, namely the {\em endoquiver}  $R$ of $T$ (=quiver
of the endomorphism algebra of $T$ in $\cc$). The pair $(\cc,T)$ is called
a {\em $2$-Calabi-Yau realization} of the quiver $R$. One expects the
cluster combinatorics associated with the quiver $R$ to be encoded in the category
$\cc$. We therefore construct a $2$-Calabi-Yau realization $(\cc,T)$ for
the triangle product $R=Q\boxtimes Q'$ (section~\ref{ss:2-CY-realizations-of-triangle-products}).
In section~\ref{ss:description-via-quiver-with-potential}, we describe
the category $\cc$ using a quiver with potential. This will later enable
us to determine the endoquivers of new cluster tilting objects in $\cc$.

In section~\ref{s:cluster-combinatorics-from-triangulated-categories}, we
show how to recover cluster combinatorial data associated with a quiver
$R$ from a $2$-Calabi-Yau realization $(\cc,T)$ of $R$. The most important
step is to lift the mutation operation to the mutation of cluster tilting
objects in such a way that the endoquivers transform as predicted by
Fomin-Zelevinsky's mutation rule (section~\ref{ss:decategorification:quiver-mutation}).
In fact, this is not always possible because $2$-cycles may appear in the
quivers of the mutated cluster tilting objects. However, we show in
sections~\ref{ss:decategorification:quiver-mutation},
\ref{ss:decategorification:g-vectors-tropical-Y-variables}
and \ref{ss:decategorification:cluster-variables-F-polynomials}, that
if no $2$-cycles appear, then all cluster combinatorial data we need
can be recovered from the category: quivers, tropical $Y$-variables and
$F$-polynomials (as well as $g$-vectors and cluster variables). For
a given $2$-Calabi-Yau realization $(\cc,T)$, the appearance of $2$-cycles
in the endoquivers of a sequence of mutations of $T$ has to be excluded
by an explicit computation. In section~\ref{ss:constrained-quivers},
we perform this computation for the $2$-Calabi-Yau realizations
associated with a class of quivers with potential (the $(Q,Q')$-constrained
quivers with potential) which includes our $2$-Calabi-Yau realization $\cc$ of
the  triangle product $Q\boxtimes Q'$
and its canonical sequence of mutations $\mu_\boxtimes$. We then reduce
the conjecture to the study of an auto-equivalence, the {\em Zamolodchikov transformation},
of the category $\cc$ (sections~\ref{ss:Zamolodchikov-transformation}
and \ref{ss:Zamolodchikov-composition-of-mutations}). We conclude by
showing that the order of the Zamolodchikov transformation is finite and
divides the sum of the Coxeter numbers (section~\ref{ss:order-of-the-Zamolodchikov-transformation}).
In section~\ref{s:non-simply-laced-case}, we reduce the non simply laced
case of the conjecture (in the form due to Fomin-Zelevinsky) to the
simply laced case using the classical folding technique.

In the final section~\ref{s:effectiveness}, we show that our proof is
effective in the sense that, in principle, it allows us to write down explicit (periodic)
formulas for the general solution of the $Y$-system in terms
of homological invariants and Euler characteristics of quiver Grassmannians.

\subsection*{Acknowledgments}
The author is indebted to Bernard Leclerc, who rekindled his
interest in the conjecture during a conversation at the MSRI
program on combinatorial representation theory in May 2008.
He thanks Carles Casacuberta, Andr\'e Joyal, Joachim Kock
and Amnon Neeman for an invitation to the Centre de
Recerca Matem\`atica, Barcelona, where the details of the
proof were worked out and \cite{Keller08c} was written down
in June 2008. He is grateful Sarah Scherotzke, Lingyan Guo
and Alfredo N\'ajera Ch\'avez
for pointing out misprints in an earlier version of this
paper as well as to two anonymous referees for their careful
reading of the manuscript and their helpful advice.

\section{The conjecture}
\label{s:the-conjecture}

\subsection{Statement}
\label{ss:statement} Let $\Delta$ and $\Delta'$ be two Dynkin
diagrams with vertex sets $I$ and $I'$. Let $A$ and $A'$ be the
incidence matrices of $\Delta$ and $\Delta'$, \ie if $C$ is the
Cartan matrix of $\Delta$ and $J$ the identity matrix of the same
format, then $A=2 J -C$. Let $h$ and $h'$ be the Coxeter numbers of
$\Delta$ and $\Delta'$.

The {\em $Y$-system of algebraic equations} associated with the pair
of Dynkin diagrams $(\Delta,\Delta')$ is a system of countably many
recurrence relations in the variables $Y_{i,i',t}$, where $(i,i')$
is a vertex of $\Delta\times\Delta'$ and $t$ an integer. The system
reads as follows:
\begin{equation} \label{eq:Y-system}
Y_{i,i',t-1} Y_{i,i',t+1} = \frac{\prod_{j\in I}
(1+Y_{j,i',t})^{a_{ij}}}{\prod_{j'\in I'}
(1+Y_{i,j',t}^{-1})^{a'_{i'j'}}}.
\end{equation}

\begin{periodicity-conjecture} \label{conj:periodicity} All solutions to this
system are periodic in $t$ of period dividing $2(h+h')$.
\end{periodicity-conjecture}

\noindent
We refer to the introduction for a sketch of the
history of the conjecture.

\begin{theorem}
The periodicity conjecture~\ref{conj:periodicity} is true.
\end{theorem}

Let us give an algebraic reformulation of the conjecture:
Let $K$ be the fraction field of the ring of integer polynomials
in the variables $Y_{i i'}$, where
$i$ runs through the set of vertices $I$ of $\Delta$ and $i'$
through the set of vertices $I'$ of $\Delta'$. Since $\Delta$ is a
tree, the set $I$ is the disjoint union of two subsets $I_+$ and
$I_-$ such that there are no edges between any two vertices of $I_+$
and no edges between any two vertices of $I_-$. Analogously, $I'$ is
the disjoint union of two sets of vertices $I'_+$ and $I'_-$. For a
vertex $(i,i')$ of the product $I\times I'$, define $\eps(i,i')$ to
be $1$ if $(i,i')$ lies in $I_+\times I'_+ \cup I_-\times I'_-$ and
$-1$ otherwise. For $\eps=\pm 1$, define an automorphism $\tau_\eps$
of $K$ by
\begin{equation} \label{eq:automorphism-tau-eps}
\tau_\eps(Y_{i i'}) = \left\{\begin{array}{ll}
Y_{i i'}\prod_{j} (1 + Y_{ji'})^{a_{ij}}
\prod_{j'} (1+ Y_{ij'}^{-1})^{-a'_{i'j'}} & \mbox{ if } \eps(i,i')=\eps ; \\
Y_{ii'}^{-1} & \mbox{ if } \eps(i,i')=-\eps.
\end{array} \right.
\end{equation}
Finally, define an automorphism $\phi$ of $K$
by
\begin{equation} \label{eq:automorphism-phi}
\phi = \tau_- \tau_+.
\end{equation}
Then, as in \cite{FominZelevinsky03b}, we have the following lemma:

\begin{lemma} \label{lemma:periodicity-from-finite-order}
The periodicity conjecture holds iff
the order of the automorphism $\phi$ is finite and divides $h+h'$.
\end{lemma}
\begin{proof} We adapt the proof given in \cite{FominZelevinsky03b}
for the case where $\Delta'=A_1$.
First we notice that the equation~\ref{eq:Y-system}
only involves variables $Y_{i,i',t}$ with a fixed `parity'
$\eps(i,i')(-1)^t$. Therefore, the $Y$-system decomposes into two
independent systems, an even one and an odd one, and it suffices to
show periodicity for one of them, say the even one. Thus, without
loss of generality, we may modify the odd system, and for the
purposes of the proof, we choose to put
\begin{equation} \label{eq:normalize-Y}
Y_{i,i',t}=Y_{i,i',t-1}^{-1} \mbox{ whenever } \eps(i,i')(-1)^t =-1.
\end{equation}
We combine~\ref{eq:Y-system} and \ref{eq:normalize-Y} into
\begin{equation} \label{eq:normalized-Y-system}
Y_{i,i',t+1}=\left\{ \begin{array}{ll} Y_{i,i',t}
\prod_{j\in I} (1+Y_{j,i',t})^{a_{ij}}\prod_{j'\in I'}
(1+Y_{i,j',t}^{-1})^{-a'_{i'j'}} &
\mbox{ if }\eps(i,i')(-1)^{t+1}=1 ; \\
Y_{i,i',t}^{-1} & \mbox{ if } \eps(i,i')(-1)^{t+1}=-1.
\end{array} \right.
\end{equation}
Thus, if we put $Y_{i,i',t}=Y_{ii'}$, then we have
\[
Y_{i,i',t+1}=\tau_{(-1)^{t+1}}(Y_{ii'}) \ko
\]
for all $i\in I$ and $i'\in I$,
as we see by comparing \ref{eq:normalized-Y-system} with
\ref{eq:automorphism-tau-eps}.
Now we set $Y_{i,i',0}=Y_{ii'}$ for $i\in I$, $i'\in I'$.
By induction on $t$, we obtain, for all $t\geq 0$ and all
$i\in I$ and $i'\in I'$,
\[
Y_{i,i', t} = (\tau_-\tau_+ \ldots \tau_{\pm})(Y_{ii'}) \ko
\]
where the number of factors $\tau_+$ and $\tau_-$ equals $t$.
In particular,
we obtain $Y_{i,i',2t}=\phi^t(Y_{i,i',0})$ for all $t\geq 0$,
$i\in I$ and $i'\in I'$, which clearly implies the assertion.
\end{proof}

\subsection{Plan of the proof} \label{ss:plan-of-the-proof}
We refer to the respective sections in the body of the paper for
detailed explanations of the notions appearing in the following plan.

Let $\Delta$ and $\Delta'$ be two Dynkin diagrams. In section~\ref{s:non-simply-laced-case},
we use the standard folding technique to reduce the conjecture to the case
where $\Delta$ and $\Delta'$ are simply laced, which we assume from
now on. We choose alternating quivers
$Q$ and $Q'$ (\cf section~\ref{ss:quiver-mutation}) whose underlying
graphs are $\Delta$ and $\Delta'$. We define the square
product $Q\square Q'$ and the triangle product $Q\boxtimes Q'$
as certain quivers whose vertex set is the product of the
vertex sets of $Q$ and $Q'$ (\cf section~\ref{ss:products-of-quivers}).
We associate canonical sequences of mutations $\mu_\square$ and $\mu_\boxtimes$
to these products. These yield restricted $Y$-patterns
$\mathbf{y}_\square$ and $\mathbf{y}_\boxtimes$, \cf section~\ref{ss:reformulation}.
Let $\phi$ be as defined in equation~\ref{eq:automorphism-phi}.

\smallskip\noindent
{\em Step 1.} We have $\phi^{h+h'}=\id$ iff the restricted $Y$-pattern $\mathbf{y}_\square$
is periodic of period dividing $h+h'$ iff this holds for the restricted
$Y$-pattern $\mathbf{y}_\boxtimes$.
\smallskip

This step is proved in section~\ref{ss:reformulation} by adapting the
methods of section~8 of \cite{FominZelevinsky07}. Notice however that
in the case $\Delta'=A_1$ considered there, the systems
$\mathbf{y}_\square$ and $\mathbf{y}_\boxtimes$ are indistinguishable.

\smallskip\noindent
{\em Step 2.} The restricted $Y$-pattern $\mathbf{y}_\boxtimes$ is periodic
of period dividing $h+h'$ iff such a periodicity holds for the sequences
of the tropical $Y$-variables and of $F$-polynomials
(\cf section~\ref{ss:trop-Y-var-F-pol}) associated with
the sequence of mutations $\mu_\boxtimes^p$, $p\in\Z$.
\smallskip

This follows from Proposition~3.12 of \cite{FominZelevinsky07}, \cf
sections~\ref{ss:trop-Y-var-F-pol} and \ref{ss:periodicity-from-that-of-tropical-Y-variables}.
We refer to sections~\ref{ss:KrullSchmidtCategories} and \ref{ss:CalabiYauCategories}
for the terminology used in the following step.

\smallskip\noindent
{\em Step 3.} There is  a triangulated $2$-Calabi-Yau category $\cc$
with a cluster-tilting object $T$ whose endoquiver (=quiver of its
endomorphism algebra) is $Q\boxtimes Q'$.
\smallskip

We construct the category $\cc$ as the (generalized) cluster
category in the sense of Amiot~\cite{Amiot09} associated with
the tensor product $kQ\ten kQ'$ of the path algebras
of the quivers $Q$ and $Q'$, \cf section~\ref{ss:2-CY-realizations-of-triangle-products}.
Thanks to Iyama-Yoshino's results \cite{IyamaYoshino08}, there
is a well-defined mutation operation for cluster tilting
objects in arbitrary $2$-Calabi-Yau categories (\cf section~\ref{ss:decategorification:quiver-mutation}).

\smallskip\noindent
{\em Step 4.} When we apply powers of the sequence of mutations $\mu_\boxtimes$
to the cluster tilting object $T$, no loops or $2$-cycles appear in the endoquivers
of the mutated cluster tilting objects.
\smallskip

We prove this by an explicit computation: First we show that $\cc$ is
equivalent to the cluster category associated \cite{Amiot09} with a Jacobi-finite
quiver with potential of the form $(Q\boxtimes Q', W)$ and that under
this equivalence, the object $T$ corresponds to the canonical cluster tilting
object (section~\ref{ss:description-via-quiver-with-potential}). Then we show that
when we perform the sequence of mutations $\mu_\boxtimes$ on the
quiver with potential $(Q\boxtimes Q',W)$ following \cite{DerksenWeymanZelevinsky08}
(\cf section~\ref{ss:reminder-quivers-with-potential}) no loops or $2$-cycles
appear in the mutated quivers with potential. Another proof of this step,
based on \cite{Amiot09}, was given in Proposition~4.35
of \cite{InoueIyamaKunibaNakanishiSuzuki10}.

\smallskip\noindent
{\em Step 5.} If there is an isomorphism $\mu_\boxtimes^{h+h'}(T) \cong T$, then
periodicity holds for the tropical $Y$-variables and the $F$-polynomials
associated with the sequence $\mu_\boxtimes^p$, $p\in\Z$.
\smallskip

This is `decategorification'. Thanks to steps~3 and 4, it follows essentially
from Palu's multiplication formula \cite{Palu08a} for the generalized Caldero-Chapoton map
\cite{CalderoChapoton06}, \cf sections~\ref{ss:decategorification:g-vectors-tropical-Y-variables}
and \ref{ss:decategorification:cluster-variables-F-polynomials}. Another
proof of this step is given in Proposition~4.35 of \cite{InoueIyamaKunibaNakanishiSuzuki10}.

\smallskip\noindent
{\em Step 6.} We have $\mu_{\boxtimes}(T)=\Za(T)$, where $\Za:\cc\to\cc$ is
the Zamolodchikov transformation.
\smallskip

The Zamolodchikov transformation, defined in section~\ref{ss:Zamolodchikov-transformation},
should be viewed as a categorification of the automorphism $\phi$ of equation~\ref{eq:automorphism-phi}.
If $\Delta'$ equals $A_1$, it coincides with the (inverse) Auslander-Reiten translation
(and also with the loop functor of the category $\cc$).

\smallskip\noindent
{\em Step 7.} We have $\Za^{h+h'}=\id$ and therefore $\mu_\boxtimes^{h+h'}(T)\cong T$.
\smallskip

This follows easily from the $2$-Calabi-Yau property of the category $\cc$
and its construction from the tensor product of the path algebras of $Q$ and $Q'$,
\cf section~\ref{ss:order-of-the-Zamolodchikov-transformation}. Thanks to
steps~5 and 2, this concludes the proof, \cf section~\ref{ss:conclusion}.

\section{Reformulation of the conjecture in terms of cluster combinatorics}
\label{s:reformulation}

\subsection{Quiver mutation} \label{ss:quiver-mutation}
A {\em quiver} $Q$ is an oriented graph
given by a set of vertices $Q_0$, a set of arrows $Q_1$ and
two maps $s: Q_1 \to Q_0$ and $t: Q_1 \to Q_0$ taking an arrow
to its source respectively its target. A quiver $Q$ is {\em finite} if
the sets $Q_0$ and $Q_1$ are finite. A vertex $i$ of a quiver
is a {\em source} (respectively, a {\em sink}) if there are
no arrows $\alpha$ with target $i$ (respectively, with source $i$).
A quiver is {\em alternating} if each of its vertices is
a source or a sink. A {\em Dynkin quiver} is a quiver
whose underlying graph is a Dynkin diagram of type
$A_n$, $n\geq 1$, $D_n$, $n\geq 4$, $E_6$, $E_7$ or $E_8$.

Let $Q$ be a quiver.
A {\em loop} of $Q$ is an arrow $\alpha$ whose source and target
coincide. A {\em $2$-cycle} of $Q$ is a pair of distinct
arrows $\beta$ and $\gamma$ such that $s(\beta)=t(\gamma)$
and $t(\beta)=s(\gamma)$.

Let $Q$ be a finite quiver without loops or $2$-cycles. Let
$k$ be a vertex of $Q$. Following Fomin-Zelevinsky \cite{FominZelevinsky02},
we define the {\em mutated quiver} $\mu_k(Q)$: It has
the same set of vertices as $Q$; its set of arrows is obtained
from that of $Q$ as follows:
\begin{itemize}
\item[1)] for each subquiver $\xymatrix{i \ar[r] & k \ar[r] & j}$,
add a new arrow $\xymatrix{i \ar[r] & j}$;
\item[2)] reverse all arrows with source or target $k$;
\item[3)] remove the arrows in a maximal set of pairwise
disjoint $2$-cycles.
\end{itemize}
It is not hard to check that we have $\mu_k^2(Q)=Q$ for each vertex $k$.
For example, the following two quivers are obtained from each other by
mutating at the black vertex
\begin{equation} \label{eq:product-a2-by-a2}
\xymatrix{ \bullet \ar[r] & \circ \ar[ld] \\
\circ \ar[u] \ar[r] & \circ \ar[u] }
\quad\quad\quad
\xymatrix{ \bullet \ar[d] & \circ \ar[l] \\
\circ \ar[r] & \circ \ar[u]} \quad \raisebox{-0.5cm}{.}
\end{equation}
From now and until the end of this section, we assume
that the set of vertices of $Q$ is the set of integers
$1$, \ldots, $n$ for some $n\geq 1$.

There is a skew-symmetric integer matrix $B$ associated
with $Q$ defined such that $b_{ij}$ is the difference of
the number of arrows from $i$ to $j$ minus the number of
arrows from $j$ to $i$ in $Q$. Clearly, the map taking
$Q$ to $B$ establishes a bijection between the quivers
without loops or $2$-cycles with vertex set $\{1, \ldots, n\}$
(modulo isomorphisms which are the identity on the vertices)
and the skew-symmetric integer $n\times n$-matrices.
Via this bijection, the above operation of quiver mutation
corresponds to matrix mutation as originally defined
by Fomin-Zelevinsky \cite{FominZelevinsky02}.
According to \cite{FominZelevinsky07},
the matrix $B'$ corresponding to the mutated quiver is given by
\begin{equation} \label{eq:matrix-mutation}
b'_{ij} =\left\{ \begin{array}{ll} -b_{ij} & \mbox{if $i=k$ or $j=k$} \\
b_{ij}+\sgn(b_{ik}) \max(0, b_{ik}b_{kj}) & \mbox{otherwise.}\end{array} \right.
\end{equation}

Let $\mathbb{T}_n$ be the {\em $n$-regular tree:} Its edges
are labeled by the integers $1$, \ldots, $n$ such that
the $n$ edges emanating from each vertex carry different
labels. Let $t_0$ be a vertex of $\mathbb{T}_n$.
To each vertex $t$ of $\mathbb{T}_n$ we associate
a quiver $Q(t)$
such that at $t=t_0$, we have $Q(t)=Q$
and whenever $t$ is linked to $t'$ by an edge labeled $i$, we
have $Q(t')=\mu_i Q(t)$. The family of quivers $Q(t)$, where
$t$ runs through the vertices of $\mathbb{T}_n$ is the {\em quiver pattern}
associated with $Q$.

\subsection{$Y$-seeds} \label{ss:Y-seeds} We follow \cite{FominZelevinsky07}.
Let $n\geq 1$ be an integer. A {\em $Y$-seed} is a pair $(Q,Y)$ formed by a finite
quiver $Q$ without loops or $2$-cycles with vertex set
$\{1, \ldots, n\}$ and by a free generating set $Y=\{Y_1, \ldots, Y_n\}$
of the field $\Q(y_1, \ldots, y_n)$ generated over $\Q$
by indeterminates $y_1, \ldots, y_n$. If $(Q,Y)$ is a $Y$-seed
and $k$ a vertex of $Q$, the {\em mutated $Y$-seed} $\mu_k(Q,Y)$
is the $Y$-seed $(Q',Y')$ where $Q'=\mu_k(Q)$ and, for $1\leq j\leq n$, we have
\[
Y'_j = \left\{ \begin{array}{ll} Y_k^{-1} & \mbox{if $j=k$;} \\
              Y_j (1+Y_k^{-1})^{-m} & \mbox{if there are $m\geq 0$ arrows $k\to j$} \\
              Y_j (1+Y_k)^m & \mbox{if there are $m\geq 0$ arrows $j\to k$}
              \end{array} \right.
\]
One checks that $\mu_k^2(Q,Y)=(Q,Y)$.
For example, the following $Y$-seeds are related by a mutation at
the vertex $1$
\[
\xymatrix{ y_1 \ar[r] & y_2 \ar[ld] \\ y_3 \ar[u] \ar[r] & y_4 \ar[u]} \quad\quad\quad
\xymatrix{ 1/y_1 \ar[d] & y_2/(1+y_1^{-1}) \ar[l] \\ y_3(1+y_1) \ar[r] & y_4\ar[u]} \ko
\]
where we write the variable $Y_i$ in place of the vertex $i$.

Let $Q$ be a finite quiver without loops or $2$-cycles with
vertex set $\{1, \ldots, n\}$. The {\em initial $Y$-seed} associated
with $Q$ is $(Q, \{y_1, \ldots, y_n\})$. The {\em $Y$-pattern} associated
with $Q$ is the family of $Y$-seeds $(Q(t), Y(t))$ indexed by
the vertices $t$ of the $n$-regular tree $\mathbb{T}_n$
(\cf section~\ref{ss:quiver-mutation}) such
that at the chosen initial vertex $t_0$, the $Y$-seed
$(Q(t_0),Y(t_0))$ is the initial $Y$-seed associated with $Q$
and whenever two vertices $t$ and $t'$ are linked by an edge labeled $k$, we have
$(Q(t'), Y(t'))=\mu_k(Q(t), Y(t))$.

Let $\mathbf{v}$ be a sequence
of vertices $v_1, \ldots, v_N$ of $Q$. We assume that the composed
mutation
\[
\mu_{\mathbf{v}} = \mu_{v_N} \ldots \mu_{v_2} \mu_{v_1}
\]
transforms $Q$ into itself. Then clearly the same holds for the
inverse sequence
\[
\mu_{\mathbf{v}}^{-1} = \mu_{v_1} \mu_{v_2} \ldots \mu_{v_N}.
\]
Now the {\em restricted $Y$-pattern} associated with $Q$ and $\mu_{\mathbf{v}}$
is the sequence of $Y$-seeds obtained from the initial $Y$-seed
$\mathbf{y}_0$ associated with $Q$ by applying all integer powers
of $\mu_{\mathbf{v}}$. Thus this pattern is given by a sequence of seeds
$\mathbf{y}_p$, $p\in \Z$, such that $\mathbf{y}_0$ is the initial
$Y$-seed associated with $Q$ and, for all $p\in \Z$,
$\mathbf{y}_{p+1}$ is obtained from $\mathbf{y}_p$ by the sequence
of mutations $\mu_{\mathbf{v}}$.

\subsection{Products of quivers} \label{ss:products-of-quivers}
Let $Q$ and $Q'$ be two finite
quivers without oriented cycles. We define the {\em tensor product $Q\ten Q'$}
to be the quiver whose set of vertices is the product $Q_0\times Q_0'$ and
where the number of arrows from a vertex $(i,i')$ to a vertex $(j,j')$
\begin{itemize}
\item[a)] is zero if $i\neq j$ and $i'\neq j'$;
\item[b)] equals the number of arrows from $j$ to $j'$ if $i=i'$;
\item[c)] equals the number of arrows from $i$ to $i'$ if $j=j'$.
\end{itemize}
Thus, for each vertex $i'$, the full subquiver of $Q\ten Q'$ formed by the
vertices $(i,i')$, $i\in Q_0$, is isomorphic to $Q$ by an isomorphism
taking $(i,i')$ to $i$ and similarly, for each vertex $i$ of $Q$,
the full subquiver on the vertices $(i,i')$, $i'\in Q'_0$, is
isomorphic to $Q'$ by an isomorphism taking $(i,i')$ to $i'$.
In Lemma~\ref{lemma:quiver-of-tensor-product} below, we will see that
the path algebra of a tensor product of two quivers is the tensor product
of the path algebras.
The tensor product of the quivers
\begin{align*}
\vec{A}_4 &: \xymatrix{ 1 & 2 \ar[l] \ar[r] & 3 & 4 \ar[l] } \ko\\
\vec{D}_5 &: \raisebox{0.75cm}{\xymatrix@R=0.2cm{   &  &  & 4 \ar[dl] \\ 1 & 2 \ar[l] \ar[r] & 3 & \\
& & & 5. \ar[ul] }}
\end{align*}
is depicted in figure~\ref{fig:a4-ten-d5}.

\begin{figure}
\[
\xymatrix@C=0.2cm@R=0.5cm{ & \circ \ar[dd] &  &  & \circ \ar[lll]
\ar[dd] \ar[rrr] & & &
\circ \ar[dd]  & & & \circ \ar[lll] \ar[dd] \\
     \circ  \ar[rd] & &  & \circ \ar[lll]|!{[lld];[llu]}\hole \ar[rrr]|!{[ru];[rd]}\hole \ar[rd] & & &
\circ \ar[rd] & & & \circ \ar[rd] \ar[lll]|!{[lld];[llu]}\hole & &  \\
           & \circ   & & & \circ \ar[lll] \ar[rrr] & & &
\circ  & & & \circ \ar[lll] & \\
           & \circ \ar[u] \ar[d] & & & \circ \ar[u] \ar[d] \ar[lll] \ar[rrr] & & &
\circ \ar[u] \ar[d] & & & \circ \ar[u] \ar[d] \ar[lll] & \\
           & \circ    & & & \circ  \ar[lll] \ar[rrr] &  & &
\circ  & & & \circ \ar[lll] & }
\]
\caption{The quiver $\vec{A}_4\ten\vec{D}_5$}
\label{fig:a4-ten-d5}
\end{figure}
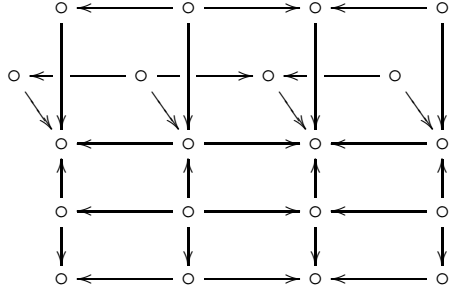

We define the {\em triangle product $Q\boxten Q'$} to be the
quiver obtained from $Q\otimes Q'$ by adding $rr'$ arrows
from $(j,j')$ to $(i,i')$ whenever $Q$ contains $r$ arrows
from $i$ to $j$ and $Q'$ contains $r'$ arrows from
$i'$ to $j'$. For example, the triangle product of the
quivers $\vec{A}_4$ and $\vec{D}_5$
is depicted in figure~\ref{fig:a4-prod-d5}.
We will see a representation-theoretic interpretation
of the triangle product in
Corollary~\ref{cor:cluster-category-Hom-finite} below.

Now assume that $Q$ and $Q'$ are {\em alternating}, \ie each
vertex is a source or a sink.
For example, the above quivers $\vec{A}_4$ and $\vec{D}_5$ are alternating.
We define the {\em square product $Q\square Q'$}
to be the quiver obtained from $Q\ten Q'$ by reversing all arrows in
the full subquivers of the form $\{i\}\times Q'$ and $Q\times \{i'\}$,
where $i$ is a sink of $Q$ and $i'$ a source of $Q'$.
The square product
of the above quivers $\vec{A}_4$ and $\vec{D}_5$ is depicted in
figure~\ref{fig:a4-prod-d5}. The triangle product and the square
product of two copies of the quiver $\vec{A}_2: \circ \to \circ$
are given in \ref{eq:product-a2-by-a2}.

\begin{lemma} \label{lemma:triangle-mutates-to-square} Let $Q$ and $Q'$
be alternating and let $M$ be the set of vertices $(i,i')$ of $Q\square Q'$
such that $i$ is a sink of $Q$ and $i'$ a source of $Q'$. Then there are
no arrows between any two vertices of $M$ and the composition of the
mutations at the vertices of $M$ transforms $Q\square Q'$ into $Q\boxtimes Q'$.
\end{lemma}

We leave the easy proof to the reader. In figure~\ref{fig:a4-prod-d5}, the
vertices belonging to $M$ are marked by $\bullet$.

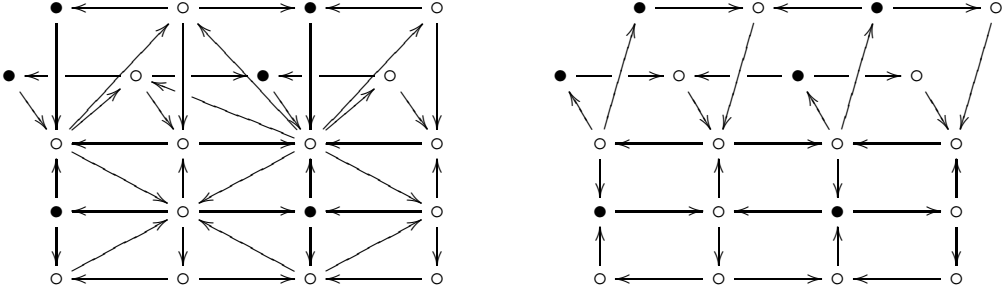
\begin{figure}
\[
\xymatrix@C=0.2cm@R=0.5cm{ & \bullet \ar[dd] &  &  & \circ \ar[lll]
\ar[dd] \ar[rrr] & & &
\bullet \ar[dd]  & & & \circ \ar[lll] \ar[dd] \\
     \bullet  \ar[rd] & &  & \circ \ar[lll]|!{[lld];[llu]}\hole \ar[rrr]|!{[ru];[rd]}\hole \ar[rd] & & &
\bullet \ar[rd] & & & \circ \ar[rd] \ar[lll]|!{[lld];[llu]}\hole & &  \\
           & \circ \ar[rrrd] \ar[urr] \ar[rrruu]  & & & \circ \ar[lll] \ar[rrr] & & &
\circ \ar[llluu] \ar[rru]  \ar[rrruu] \ar[llld] \ar[rrrd] \ar[llllu]|!{[lll];[llluu]}\hole & & & \circ \ar[lll] & \\
           & \bullet \ar[u] \ar[d] & & & \circ \ar[u] \ar[d] \ar[lll] \ar[rrr] & & &
\bullet \ar[u] \ar[d] & & & \circ \ar[u] \ar[d] \ar[lll] & \\
           & \circ \ar[rrru]   & & & \circ  \ar[lll] \ar[rrr] &  & &
\circ \ar[lllu] \ar[rrru] & & & \circ \ar[lll] & }
\quad\quad
\xymatrix@C=0.1cm@R=0.5cm{ & & \bt \ar[rrr] &  &  & \circ \ar[ldd]
& & &
\bt \ar[lll] \ar[rrr] & & & \circ \ar[ldd] \\
     \bt \ar[rrr]|!{[dr];[urr]}\hole & &  & \circ \ar[rd] & & &
\bt \ar[lll]|!{[lu];[dll]}\hole \ar[rrr]|!{[rd];[rru]}\hole & & & \circ \ar[rd] & &  \\
           & \circ \ar[lu] \ar[ruu] \ar[d] & & & \circ \ar[lll] \ar[rrr] & & &
\circ \ar[ruu] \ar[lu] \ar[d] & & & \circ \ar[lll] & \\
           & \bt \ar[rrr] & & & \circ \ar[d] \ar[u] & & &
\bt \ar[lll] \ar[rrr] & & & \circ \ar[u] \ar[d] & \\
           & \circ \ar[u] & & & \circ \ar[lll] \ar[rrr] &  & &
\circ \ar[u] & & & \circ \ar[lll] & }
\]
\caption{The quivers $\vec{A}_4\boxten\vec{D}_5$ and $\vec{A}_4 \square \vec{D}_5$}
\label{fig:a4-prod-d5}
\end{figure}

\subsection{Reformulation of the conjecture}
\label{ss:reformulation} Let $\Delta$ and $\Delta'$ be simply
laced Dynkin diagrams. We choose alternating quivers $Q$ and $Q'$ whose
underlying graphs are $\Delta$ and $\Delta'$. If $i$ is a vertex of
$Q$ or $Q'$, we put $\eps(i)=1$ if
$i$ is a source and $\eps(i)=-1$ if $i$ is a sink. For example, we
can consider the quivers $\vec{A}_4$ and $\vec{D}_5$ of section~\ref{ss:products-of-quivers}.
For two elements $\sigma$, $\sigma'$ of $\{+,-\}$, we define the
following composed mutation of $Q\square Q'$:
\[
\mu_{\sigma, \sigma'} = \prod_{ \eps(i)=\sigma, \eps(i')=\sigma'}
\mu_{(i,i')}.
\]
Notice that there are no arrows between any two vertices of the
index set so that the order in the product does not matter. Then it
is easy to check that $\mu_{+,+} \mu_{-,-}$ and $\mu_{-,+}
\mu_{+,-}$ both transform $Q\square Q'$ into $(Q\square Q')^{op}$
and vice versa. Thus the composed sequence of mutations
\[
\mu_{\square}=\mu_{-,-} \mu_{+,+} \mu_{-,+} \mu_{+,-}
\]
transforms $Q\square Q'$ into itself. We define the {\em $Y$-system
$\mathbf{y}_\square$} associated with $Q\square Q'$ to be the
restricted $Y$-pattern (\cf section~\ref{ss:Y-seeds}) associated with $Q\square Q'$ and
$\mu_\square$. The following lemma is inspired from section~8 of \cite{FominZelevinsky07}.

\begin{lemma} The periodicity conjecture holds for $\Delta$ and $\Delta'$
if and only if the $Y$-system $\mathbf{y}_\square$ is periodic of period dividing
$h+h'$.
\end{lemma}

\begin{proof} Let $\Sigma_0$ be the initial $Y$-seed associated with
$Q\square Q'$ and, for $t\geq 1$, define the $Y$-seed $\Sigma_t$ by
\[
\Sigma_t = \left\{ \begin{array}{ll} \mu_{+,-}\mu_{-,+} (\Sigma_{t-1}) &  \mbox{if $t$ is odd;} \\
\mu_{+,+}\mu_{-,-} (\Sigma_{t-1}) & \mbox{if $t$ is even.}
\end{array} \right.
\]
Let $(Y_{i,i',t})$ be the generating set of $\Q(y_1, \ldots, y_n)$ given by $\Sigma_t$.
Then using equation~\ref{eq:normalized-Y-system} we see
that the $Y_{i,i',t}$ are precisely those defined in the proof
of Lemma~\ref{lemma:periodicity-from-finite-order}. Therefore, the periodicity
conjecture translates into the fact that $\Sigma_{2(h+h')}=\Sigma_0$ and
this yields the assertion.
\end{proof}

By lemma~\ref{lemma:triangle-mutates-to-square}, we have
\[
\mu_{-+}(Q\square Q') = Q\boxten Q'.
\]
Therefore, the periodicity of the
restricted $Y$-system $\mathbf{y}_\square$ associated with $Q\square Q'$ and
$\mu_\square$ is equivalent to that of the restricted $Y$-system $\mathbf{y}_\boxten$
associated with $Q\boxten Q'$ and
\begin{equation} \label{eq:inverse-Zamolodchikov-mutation}
\mu_\boxten=\mu_{-,+}  \mu_{+,+}\mu_{-,-}\mu_{+,-} .
\end{equation}
So we finally obtain the following lemma.

\begin{lemma} \label{lemma:reduction-to-y-boxten}
The periodicity conjecture holds for $\Delta$ and $\Delta'$
if and only if the $Y$-system $\mathbf{y}_\boxten$ is periodic of period dividing
$h+h'$.
\end{lemma}

\section{More cluster combinatorics}
\label{s:more-cluster-combinatorics}

Let $n\geq 1$ be an integer and $Q$ a finite quiver without
loops or $2$-cycles with vertex set $\{1, \ldots, n\}$.
Let $\mathbb{T}_n$ be the $n$-regular tree and $t_0$
a distinguished vertex of $\mathbb{T}_n$. Let $(Q(t))$
be the quiver pattern associated with $Q$ as in
section~\ref{ss:quiver-mutation}.

\subsection{Tropical $Y$-variables and $F$-polynomials} \label{ss:trop-Y-var-F-pol}
The {\em tropical semifield} $\Trop(y_1, \ldots, y_n)$
generated by the indeterminates $y_1, \ldots, y_n$ is the free
multiplicative group generated by the $y_i$ endowed with
the auxiliary addition defined by
\[
\prod_i y_i^{a_i} \oplus \prod_i y_i^{b_i} = \prod_i y_i^{\min(a_i, b_i)}.
\]

The {\em tropical $Y$-variables
$\eta_i(t)$}, $1\leq i\leq n$, are elements of $\Trop(y_1, \ldots, y_n)$
associated to the vertices $t$ of $\mathbb{T}_n$. They are
defined recursively as follows: we put
\[
\eta_i(t_0)=y_i \ko 1\leq i\leq n \ko
\]
and if $t$ is linked to $t'$ by an edge labeled $k$, we put
\begin{equation} \label{eq:mut-tropical-y-variable}
\eta_j(t')= \left\{ \begin{array}{ll} \eta_j(t)^{-1} & \mbox{if $j=k$; } \\
\eta_j(t) (1\oplus \eta_k(t))^m & \mbox{if there are $m\geq 0$ arrows $j \to k$ in $Q(t)$;} \\
\eta_j(t)/(1\oplus \eta_k(t)^{-1})^m & \mbox{if there are $m\geq 0$ arrows $k\to j$ in $Q(t)$}.
\end{array}\right.
\end{equation}

The {\em $F$-polynomials $F_i(t)$}, $1\leq i\leq n$,
are elements of the polynomial ring
$\Z[y_1, \ldots, y_n]$ associated to the vertices $t$ of $\mathbb{T}_n$.
According to Proposition~5.1 of \cite{FominZelevinsky07},
they can be defined recursively as follows: we put
\[
F_i(t_0)=1 \ko 1\leq i\leq n \ko
\]
and if $t$ is linked to $t'$ by an edge labeled $k$,
then $F_i(t')=F_i(t)$ for all $i\neq k$ and $F_k(t')$ is defined
by the {\em exchange relation}
\begin{equation} \label{eq:exchange-relation-F-polynomials}
F_k(t) F_k(t')= \prod_{c_{ik}>0} y_i^{c_{ik}}\prod_{\stackrel{\mbox{\tiny arrows}}{k\to j}} F_j(t) +
\prod_{c_{ik}<0} y_i^{-c_{ik}}\prod_{\stackrel{\mbox{\tiny arrows}}{i\to k}} F_i(t)\ko
\end{equation}
where the products are taken over the arrows in the quiver $Q(t)$ and the $c_{ik}$
are defined by
\[
\eta_k(t) = \prod_{i=1}^n y_i^{c_{ik}}.
\]

Notice that the group ring of the free multiplicative abelian group
underlying the tropical semifield $\Trop(y_1, \ldots, y_n)$ is
canonically isomorphic to the ring of Laurent polynomials
\[
\Z[y_1^{\pm 1}, \ldots, y_n^{\pm 1}].
\]
Thanks to this identification, it makes sense to multiply elements
of the tropical semifield with elements of the field $\Q(y_1, \ldots, y_n)$.
This is the multiplication that we use in the following proposition.
\begin{proposition}[Proposition 3.12 of \cite{FominZelevinsky07}]
\label{prop:Y-variables}
At each
vertex $t$ of $\mathbb{T}_n$, and for each $1\leq j\leq n$,
we have
\begin{equation} \label{eq:non-trop-y-var}
Y_j(t)= \eta_j(t) \prod_{i=1}^n F_i(t)^{b_{ij}(t)}
\end{equation}
where $(b_{ij}(t))$ is the skew-symmetric integer
matrix associated with $Q(t)$.
\end{proposition}

\subsection{Consequence for the periodicity conjecture}
\label{ss:periodicity-from-that-of-tropical-Y-variables}
Now let $Q$ and $Q'$ be as in section~\ref{ss:reformulation}.
Let $t_0$ be the initial vertex of the $N$-regular tree $\mathbb{T}_{N}$
whose edges are labeled by the vertices of $Q\boxten Q'$.
Let $t_i$, $i\in \Z$, be the sequence of vertices of
$\mathbb{T}_N$ which are visited when performing the
mutations in the integer powers of the composition
\[
\mu_\boxten=\mu_{+,-} \mu_{-,-} \mu_{+,+} \mu_{-,+}
\]
(for each factor $\mu_{\sigma,\sigma'}$, we choose some order
of the mutations). Notice that $\mu_{\boxten}$ contains
exactly one mutation at each of the $N$ vertices of $Q\boxten Q'$.
We already know from section~\ref{ss:products-of-quivers} that
$(Q\boxten Q')(t_{pN})=Q\boxten Q'$ for all $p\in \Z$.
Thus, by the above proposition, the periodicity
conjecture holds for $Q$ and $Q'$ if, for each vertex $j$ of
$Q\boxten Q'$, the sequences $\eta_j(t_{pN})$ and $F_j(t_{pN})$
are periodic in $p$ of period dividing $h+h'$.

\subsection{Cluster variables, $g$-vectors}
\label{ss:cluster-variables-g-vectors}
We will not need cluster variables for the proof of the conjecture.
We nevertheless define them since they are useful
in other applications of the categorical model that we will construct,
notably in the study \cite{InoueIyamaKunibaNakanishiSuzuki10} of
$T$-systems.  We will use the categorical lift of the $g$-vectors to express
the tropical $Y$-variables (cf.~Corollary~\ref{cor:decat-tropical-Y-var} and
Corollary~\ref{cor:decat-tropical-Y-variables}). In this way, the $g$-vectors do
play a role in our proof.

The {\em cluster variables $X_i(t)$, $1\leq i\leq n$}, associated
to the vertices $t$ of $\mathbb{T}_n$ lie in the field $\Q(x_1, \ldots, x_n)$
generated by $n$ indeterminates $x_1$, \ldots, $x_n$; they are
defined recursively by
\begin{equation}
X_i(t_0) = x_i \ko 1 \leq i\leq n \ko
\end{equation}
and if $t$ is linked to $t'$ by an edge labeled $k$, then
$X_i(t')=X_i(t)$ for all $i\neq k$ and $X_k(t')$ is determined
by the {\em exchange relation}
\begin{equation} \label{eq:exchange-relation}
X_k(t) X_k(t')=\prod_{\stackrel{\mbox{\tiny arrows}}{i\to k}}
X_i(t) + \prod_{\stackrel{\mbox{\tiny arrows}}{k\to j}} X_j(t) \ko
\end{equation}
where the products are taken over the arrows of $Q(t)$ with source, respectively,
target $k$. By a fundamental theorem
of Fomin-Zelevinsky \cite{FominZelevinsky02}, all cluster
variables are Laurent polynomials with integer coefficients
in the initial cluster variables $x_1$, \ldots, $x_n$.

The {\em $g$-vectors} $g_i^{t_0}(t)$, $1\leq i\leq n$, are
vectors in $\Z^n$ associated with the vertices $t$ of $\mathbb{T}_n$.
In the language of Fock-Goncharov \cite{FockGoncharov09}, they can be interpreted as
distinguished points of the tropical $\mathcal{X}$-variety associated with $Q$.
To define them, we first define, for all vertices $t_1$ and $t_2$
of $\mathbb{T}_n$, a piecewise linear bijection
\[
G(t_2, t_1): \Z^n \to \Z^n
\]
by induction on the distance between $t_1$ and $t_2$ in the tree $\mathbb{T}_n$
as follows: 1) We put $G(t_1, t_1)$ equal to the identity. 2) If
$t_1$ and $t_1'$ are linked by an edge labeled $k$, we put
\[
G(t_1', t_1)(v) = \left\{
\begin{array}{ll} \phi_+(v) & \mbox{if } v=\sum x_i e_i \mbox{ with } x_k\geq 0 \ko\\
\phi_-(v) & \mbox{if } v=\sum x_i e_i \mbox{ with } x_k\leq 0 \ko
\end{array} \right.
\]
where $\phi_+$ and $\phi_-$ are the linear automorphisms of $\Z^n$ with
$\phi_+(e_i)=e_i=\phi_-(e_i)$ for $i\neq k$ and
\[
\phi_+(e_k)= -e_k +\sum_{i \to k} e_i \mbox{ and }
\phi_-(e_k)= -e_k +\sum_{k\to j}  e_j \ko
\]
where the sums are taken over the arrows $i\to k$
respectively $k\to j$ of the quiver $Q(t_1')$. One checks that
$G(t_1,t_1')$ and $G(t_1', t_1)$ are inverse to each other.
3) If the shortest path linking $t_1$ to $t_2$ is of length
greater than or equal to two, we define
\[
G(t_2, t_1) = G(t_2, t_1') \circ G(t_1', t_1),
\]
where $t_1'$ is the first vertex different from $t_1$
on the shortest path from $t_1$ to $t_2$.
Now, for all vertices $t$ and all $1\leq i \leq n$,  we define
\[
g_i^{t_0}(t) = G(t,t_0)(e_i).
\]
By Theorem~1.7 of \cite{DerksenWeymanZelevinsky10},
this definition agrees with the one given in section~6 of
\cite{FominZelevinsky07}. Notice that, like all the other data
defined before, the $g$-vectors also depend on the initial
vertex $t_0$. For later reference, notice that
the $g$-vectors are characterized by
\begin{equation} \label{eq:g-vector-recursion1}
g^{t_0}_i(t_0)=e_i \ko 1 \leq i\leq n \ko
\end{equation}
and, whenever $t_0$ and $t_1$ are linked by an edge labeled $k$,
we have
\begin{equation} \label{eq:g-vector-recursion2}
g_i^{t_1}(t)= \left\{\begin{array}{ll} \phi_+(v) & \mbox{if } x_k\geq 0 \ko \\
                             \phi_-(v) & \mbox{if } x_k\leq 0 \ko
                             \end{array} \right.
\end{equation}
where $v$ is short for $g_i^{t_0}(t)$, the integer $x_k$ is the coefficient of
$e_k$ in $v$ and the maps $\phi_+$ and $\phi_-$ are the linear automorphisms of
$\Z^n$ associated as above with the vertex $k$ and the quiver $Q(t_0)$.

By Corollary~6.3 of \cite{FominZelevinsky07}, the cluster variables
can be expressed in terms of the quiver pattern, the $g$-vectors
and the $F$-polynomials
as follows: For each vertex $t$ of $\mathbb{T}_n$ and each integer
$1\leq i\leq n$, we have
\[
X_i(t) =  F_i(t)(\hat{y}_1,\ldots, \hat{y}_n) \prod_{j=1}^n x_j^{g_j} \ko
\]
where the $g_j$ are the components of $g_i(t)$ and, if $(b_{ij})$ is the
skew-symmetric integer matrix corresponding to $Q(t)$, the elements $\hat{y}_j$
are given by
\[
\hat{y}_j = \prod_{i=1}^n x_i^{b_{ij}}.
\]

\section{Calabi-Yau triangulated categories}
\label{s:Calabi-Yau-triangulated-categories}

\subsection{Krull-Schmidt categories} \label{ss:KrullSchmidtCategories}
We briefly recall basic notions from the representation theory of finite-dimensional
associative algebras. More details can be found
in the books \cite{AssemSimsonSkowronski06} \cite{AuslanderReitenSmaloe95} \cite{GabrielRoiter92} \cite{Ringel84}.
An introduction with motivating examples from cluster theory is given in sections~5 and 6
of \cite{Keller10b}. By a {\em module}, we will mean a right module.

Recall that an {\em additive category} is a category where 1) each finite family of objects
admits a direct sum and 2) the morphism sets are endowed with structures of abelian
groups such that the composition is bilinear. For example, the category of free
modules over a ring is additive, and so are the category of projective modules
and that of all modules.
An additive category {\em has split idempotents}
if each idempotent endomorphism $e$ of an object $X$  gives rise to a direct sum decomposition
$Y\oplus Z\iso X$ such that $Y$ is a kernel for $e$. The category of free modules
over a ring usually does not have split idempotents but the category of projective
modules does. An object $X$ in an additive category is {\em indecomposable}
if in each direct sum decomposition $X\iso Y\oplus Z$, the object $Y$ or
the object $Z$ is a zero object. In the category of modules
over a ring, the simple modules are indecomposable but usually there are
many other indecomposable objects. For example, let $k$ be a field, let $n\geq 1$
be an integer and let $T_n(k)$ be the associative algebra of  {\em lower triangular
$n\times n$--matrices over $k$}. Let $e_{ij}$ denote the matrix
whose $(i,j)$-coefficient equals $1$ and whose other coefficients vanish.
Let $A=T_2(k)$. Then
the modules $P_1 = e_{11}A = [k\ 0]$ and $P_2=e_{22}A = [k\ k]$ are
indecomposable (and projective). The module $P_1$ is also simple but
the module $P_2$ is not since it contains $P_1$ as a proper submodule.
A {\em Krull-Schmidt category} is an
additive category where the endomorphism rings of indecomposable objects are
local and each object decomposes into a finite direct sum of indecomposable
objects (which are then unique up to isomorphism and permutation). One
can show that each Krull-Schmidt category has split idempotents.
We {\em write $\mathsf{indec}(\cc)$} for the set of isomorphism classes
of indecomposable objects of a Krull-Schmidt category $\cc$. For
example, if $k$ is a field, the category of finitely generated
modules over the polynomial ring $k[x]$ is not a Krull-Schmidt
category since the free module $k[x]$ is indecomposable but
its endomorphism algebra, which is $k[x]$, is not local. On the
other hand, the category of coherent sheaves over a projective
variety over $k$ is a Krull-Schmidt category. If $A$ is a finite-dimensional
associative algebra over $k$, then the category $\mod(A)$ of $A$-modules whose
underlying $k$-vector spaces are finite-dimensional is a Krull-Schmidt
category. So is its subcategory $\proj(A)$ of finitely generated
projective modules. In
rare cases, one can explicitly enumerate the isomorphism classes
of indecomposable objects of $\mod(A)$ and $\proj(A)$. For example,
if $A$ is the algebra $T_2(k)$, then, up to
isomorphism, the indecomposables of the category $\mod(A)$ are
$P_1=e_{11}A$, $P_2=e_{22}A$ and $S_2=P_2/P_1$. More generally,
the indecomposable modules over the algebra $T_n(k)$ are,
up to isomorphism, the quotients $P_j/P_i$, where
$1\leq i<j\leq n$ and $P_i=e_{ii} T_n(k)$. We refer to the
books  books \cite{AssemSimsonSkowronski06} 
\cite{AuslanderReitenSmaloe95} \cite{GabrielRoiter92} 
\cite{Ringel84} and to section~5 of \cite{Keller10b} for more examples.

Let $\cc$ be a Krull-Schmidt category.
An object $X$ of $\cc$ is {\em basic} if every indecomposable of $\cc$
occurs with multiplicity $\leq 1$ in $X$. For example, the category of finitely
generated projective modules over $T_2(k)$ contains, up to isomorphism,
exactly four basic objects: $0$, $P_1$, $P_2$ and $P_1\oplus P_2$.
If an object $X$ is basic, it is
determined, up to isomorphism, by the full additive {\em subcategory $\add(X)$}
whose objects are the direct factors of finite direct sums of copies of $X$.
The map $X \mapsto \add(X)$
yields a bijection between the isomorphism classes of basic objects and
the full additive subcategories of $\cc$ which are stable under
taking direct factors and only contain finitely many indecomposables
up to ismorphism.

From now on, let $k$ be an algebraically closed field. A {\em $k$-category} is
a category whose morphism sets are endowed with structures of
$k$-vector spaces such that the composition maps are bilinear. For example,
each full subcategory of the category of modules over a $k$-algebra
is naturally a $k$-category. A $k$-category is {\em $\Hom$-finite} if all of its morphism spaces
are finite-dimensional. For example, the category of finitely generated
projective modules over a $k$-algebra $A$ is $\Hom$-finite if and
only if the algebra $A$ is finite-dimensional over $k$.
A {\em $k$-linear category} is a $k$-category
which is additive. For example, a subcategory of the category of
modules over a $k$-algebra is $k$-linear if and only if it
is stable under forming finite direct sums (in particular
it then contains the zero module, which is the sum over the empty family of modules).
Let $\cc$ be a $k$-linear $\Hom$-finite
category with split idempotents. One can show that $\cc$ is then a Krull-Schmidt category.
Let $\ct$ be an additive subcategory of $\cc$ stable under taking direct
factors. The {\em quiver $Q=Q(\ct)$ of $\ct$}
is defined as follows: The vertices of $Q$ are the isomorphism classes of indecomposable
objects of $\ct$ and the number of arrows from the isoclass of
$T_1$ to that of $T_2$ equals the dimension of the
{\em space of irreducible morphisms}
\[
\mathsf{irr}(T_1, T_2) = \mathsf{rad}(T_1, T_2)/\mathsf{rad}^2(T_1,T_2) \ko
\]
where $\mathsf{rad}$ denotes the {\em radical of $\ct$}, \ie the ideal such
that $\mathsf{rad}(T_1,T_2)$ is formed by all non isomorphisms from $T_1$ to $T_2$.
We refer to the next section for examples.
The {\em quiver of a finite-dimensional algebra $A$} is the quiver of the category
of finitely generated projective $A$-modules. By lemma~\ref{lemma:functor-to-modules}
below, the computation of the quiver of a category can often be reduced to the computation
of the quiver of a finite-dimensional algebra $A$.

\subsection{The quiver of a finite-dimensional algebra, path algebras, representations}
\label{ss:path-algebras}
In the case of the category of finitely generated projective modules
over a finite-dimensio\-nal algebra $A$, one can describe
the radical and its square more explicitly: A morphism $f: P \to Q$ between finitely
generated projective modules lies in the radical, respectively its
square, iff it factors as $f=\pi g \,\iota$, where the morphism $\iota: P \to A^p$
is the inclusion of a direct summand of a free module, the morphism
$\pi: A^q\to Q$ is the projection onto a direct summand of a free module
and the morphism $g:A^p \to A^q$ has all its matrix coefficients in the Jacobson
radical of the algebra $A$, respectively its square. For example, let $A$ be
the algebra $T_n(k)$ (section~\ref{ss:KrullSchmidtCategories}). 
Then the Jacobson radical of $A$ is formed by the strictly lower triangular 
matrices. For $1\leq i\leq n$, put $P_i=e_{ii}A$.
Then, if $i+1\leq j\leq n$, all morphisms $P_i \to P_j$ lie in the radical.
On the other hand, if $j=i+1$, then the space of irreducible morphisms from $P_i$ to $P_j$
is one-dimensional. Thus, the quiver of the algebra $T_n(k)$ is the chain
\[
\xymatrix{[P_1] \ar[r] & [P_2] \ar[r] & \ldots \ar[r] & [P_n] }\ko
\]
where $[P_i]$ denotes the isomorphism class of $P_i$ (in the sequel,
we will usually omit the brackets). One can
generalize this example as follows: Let $Q$ be any finite quiver.
A {\em path of length $l$ in $Q$} is a formal
composition $p=(z|\alpha_l| \ldots|\alpha_2 |\alpha_1|x)$ of $l\geq 0$
arrows forming a diagram
\[
\xymatrix{ x \ar[r]^{\alpha_1} & y \ar[r]^-{\alpha_2} \ar[r] & \ldots \ar[r]^{\alpha_l} & z}.
\]
The vertex $x$ is the {\em source} and the vertex
$z$ the {\em target} of the path $p$.
For each vertex $x$ of $Q$ we have the {\em lazy path}
$e_x=(x|x)$ of length $0$. Two paths $p$ and $q$ are {\em composable}
if the target of $q$ equals the source of $p$; in this case, their
{\em composition} is obtained by concatenating $p$ with $q$.
The {\em path algebra $kQ$} is the vector
space whose basis is formed by all paths and where the product
of two paths is their composition if they are composable and
zero otherwise. Notice that $kQ$ is finite-dimensional if
and only if $Q$ does not have oriented cycles. 
For example, if $Q$ is the quiver
\[
\xymatrix{ 1 \ar[r] & 2 \ar[r] & \ldots \ar[r] & n} \ko
\]
then the path algebra $kQ$ is isomorphic to the algebra $T_n(k)$.
For an arbitrary quiver $Q$ without oriented cycles, the Jacobson
radical of $kQ$ is the two-sided ideal generated by all arrows.
Using this it is not hard to show that the quiver of the category
$\proj kQ$ is isomorphic to $Q$, where the isomorphism sends a
vertex $x$ of $Q$ to the isomorphism class of the module $e_x kQ$.

For later use, let us record the classical equivalence
(\cf \cite{AssemSimsonSkowronski06} \cite{AuslanderReitenSmaloe95} \cite{GabrielRoiter92} \cite{Ringel84})
between the category of modules over the path algebra $kQ$ and the
category of representations of the opposite quiver $Q^{op}$:
It sends a module $M$ over the path algebra to the representation
$V$ of $Q^{op}$ whose value at the vertex $i$ is $V_i=M e_i$ and
which maps an arrow $\alpha: i \to j$ to the linear
map $V_\alpha: V_j \to V_i$ given by the right multiplication with $\alpha$.
For example, the modules $P_i$, $1 \leq i\leq 3$, over the
path algebra of $Q: 1\to 2 \to 3$ correspond
to the representations
\[
\xymatrix{k & 0 \ar[l] & 0 \ar[l]} \ko
\xymatrix{k & k \ar[l]_-{\id} & 0 \ar[l]} \ko
\xymatrix{k & k \ar[l]_-{\id} & k \ar[l]_-{\id}}.
\]

\subsection{From objects to modules}
Let us fix $\cc$, a $k$-linear $\Hom$-finite category with split idempotents.
Let $T$ be a basic object of $\cc$ and $B$ its endomorphism algebra. Let
$\mod(B)$ denote the category of $k$-finite-dimensional right $B$-modules.
The following lemma is well-known and easy to prove. We denote by
$D=\Hom_k(?,k)$ the duality over the ground field.

\begin{lemma} \label{lemma:functor-to-modules}
The functor
\[
\cc(T,?): \cc \to \mod(B)
\]
induces an equivalence from $\add(T)$ to the full subcategory $\proj B$
of finitely generated projective $B$-modules. Moreover, for each object
$U$ of $\add(T)$ and each object $X$ of $\cc$, the canonical map
\[
\cc(U,X) \to \Hom_B(\cc(T,U),\cc(T,X))
\]
is bijective. Dually, the functor
\[
D\cc(?,T): \cc \to \mod(B)
\]
induces an equivalence from $\add(T)$ to the full subcategory $\inj B$
of finitely generated injective $B$-modules. Moreover, for each object
$U$ of $\add(T)$ and each object $X$ of $\cc$, the canonical map
\[
\cc(X,U) \to \Hom_B(D\cc(X,T), D\cc(U,T))
\]
is bijective.
\end{lemma}

In particular, the lemma shows that the isomorphism classes of the
indecomposable projective $B$-modules are represented by the $\cc(T,T_i)$, $1\leq i\leq n$,
where the $T_i$ are the indecomposable pairwise non isomorphic direct factors
of $T$. Thus, the quiver of the endomorphism
algebra $\End_\cc(T)$ is canonically isomorphic to that of the category
$\add(T)$: The isomorphism sends the indecomposable factor $T_i$ to the
indecomposable projective $B$-module $\cc(T,T_i)$, $1\leq i\leq n$. We
sometimes refer to the quiver of $\End_\cc(T)$ as the {\em endoquiver of $T$}.

\subsection{$2$-Calabi-Yau triangulated categories} \label{ss:CalabiYauCategories}
Let $k$ be an algebraically closed field. Let $\cc$ be a $k$-linear
triangulated category with suspension functor $\Sigma$,
\cf \cite{Verdier96}. We refer
to sections~\ref{sss:cluster-categories} and \ref{sss:cluster-category-A3} below for examples. We assume
that
\begin{itemize}
\item[(C1)] $\cc$ is {\em $\Hom$-finite} and has split idempotents.
\end{itemize}
Thus, the category $\cc$  is a Krull-Schmidt category. For objects
$X$, $Y$ of $\cc$ and an integer $i$, we define
\[
\Ext^i(X,Y)=\cc(X,\Sigma^i Y).
\]
An object $X$ of $\cc$ is {\em rigid} if $\Ext^1(X,X)=0$.

Let $d$ be an integer. The category $\cc$ is {\em
$d$-Calabi-Yau} if there exist bifunctorial isomorphisms
\[
D\cc(X,Y) \iso \cc(Y,\Sigma^d X) \ko X,Y \in \cc \ko
\]
where $D=\Hom_k(?,k)$ is the duality over the ground
field (this definition suffices for our purposes; a refined definition
is given in \cite{VandenBergh08}, \cf~also \cite{Keller08d}).
For example, the derived category of the category of coherent
sheaves on a smooth projective variety of dimension $d$ over
an algebraically closed field is $d$-Calabi-Yau if the canonical
bundle is trivial.  Let us assume that
\begin{itemize}
\item[(C2)] $\cc$ is $2$-Calabi-Yau.
\end{itemize}
In the sequel, by a {\em $2$-Calabi-Yau category}, we will mean
a $k$-linear triangulated category satisfying (C1) and (C2).

A {\em cluster tilting object} is a basic object $T$ of $\cc$ such that
$T$ is rigid and each object $X$ satisfying
$\Ext^1(T,X)=0$ belongs to $\add(T)$. If $T$ is a cluster tilting
object, we write $Q_T$ for its endoquiver.

\subsection{$2$-Calabi-Yau realizations} \label{ss:2-CY-realizations}
Let $k$ be an algebraically closed field and
$Q$ a finite quiver. A {\em $2$-Calabi-Yau realization of $Q$}
is a $2$-Calabi-Yau category $\cc$
which admits a cluster tilting object $T$ whose endoquiver
$Q_T$ is isomorphic to $Q$. We will see below that in this case,
under suitable additional assumptions, the cluster combinatorics
associated with $Q$ have a categorical lift in $\cc$. It is
therefore important to construct $2$-Calabi-Yau realizations.

\subsubsection{Cluster categories from quivers} \label{sss:cluster-categories}
Assume that $Q$ does not have oriented cycles. In this case,
the {\em cluster category $\cc_Q$} provides a $2$-Calabi-Yau
realization of $Q$. Let us recall its construction
(\cf section~\ref{sss:cluster-category-A3} below for an example): Let $A=kQ$
denote the path algebra of $Q$ (cf. section~\ref{ss:KrullSchmidtCategories})
and $\mod(A)$ the category
of $k$-finite-dimensional $A$-modules. Let $\cd^b(A)$
denote the bounded derived category of $\mod(A)$, \cf \cite{Verdier96} \cite{Happel88}
\cite{Keller07a}. Thus, the
objects of $\cd^b(A)$ are the bounded complexes
\[
\xymatrix{ \ldots \ar[r] & M^p \ar[r] & M^{p+1} \ar[r] & \ldots }
\]
of $k$-finite-dimensional
$A$-modules and its morphisms are obtained from the morphisms
of complexes by formally inverting all quasi-isomorphisms; the
triangles are `induced' by short exact sequences of complexes.
The category $\cd^b(A)$ is
a $\Hom$-finite triangulated Krull-Schmidt category which admits a
{\em Serre functor}, \ie an auto-equivalence $S:\cd^b(A)\to \cd^b(A)$
such that there are bifunctorial isomorphisms
\[
D\Hom(X,Y) \iso \Hom(Y, SX) \ko X,Y\in \cd^b(A) \ko
\]
where $D=\Hom_k(?,k)$ is the duality over the ground field,
\cf \cite{Happel88}.
In fact, the Serre functor $S$ is given by the derived
tensor product with the bimodule $DA$ which is $k$-dual
to the bimodule $A$. The suspension functor $\Sigma$ of
$\cd^b(A)$ is induced by the functor shifting a complex
one degree to the left, \ie $(\Sigma M)^{p} = M^{p+1}$,
and changing the sign of its differential. The {\em Auslander-Reiten
translation} is the auto-equivalence $\tau$ of $\cd^b(A)$ which is defined by
\[
\tau \Sigma = S.
\]
Notice that $\Sigma \tau = \tau \Sigma$ since $\tau$ is a triangle functor.
If we send a module $M$ to the complex
\[
\xymatrix{ \ldots \ar[r] & 0 \ar[r] & M \ar[r] & 0 \ar[r] & \ldots}
\]
concentrated in degree $0$, we obtain a fully faithful
embedding $\mod(A) \to \cd^b(A)$. In this way, we identify
modules with complexes. It follows from the fact that
the algebra $A$ is of global dimension at most one
that each object of $D^b(A)$ is
isomorphic to a direct sum of objects of the form
$\Sigma^p M$, where $p\in\Z$ and $M$ is a module. In
particular, the indecomposable objects of $\cd^b(A)$
are isomorphic to shifted indecomposable modules.

The cluster category was introduced in \cite{BuanMarshReinekeReitenTodorov06}
and, independently in the case of quivers whose underlying
graph is a Dynkin diagram of type $A$, in \cite{CalderoChapotonSchiffler06}.
It is the orbit category
\[
\cc_{A} = \cd^b(A)/(S^{-1} \Sigma^2)^{\Z} = \cd^b(A)/(\tau^{-1} \Sigma)^{\Z}.
\]
Thus, its objects are those of $\cd^b(A)$ and its morphisms
are defined by
\[
\Hom_{\cc_{A}}(X,Y)=\bigoplus_{p\in \Z} \Hom_{\cd^b(A)}(X, (\Sigma^2 S^{-1})^p Y) \ko
X,Y\in\cc_{A}.
\]
As shown in \cite{Keller05}, the category $\cc_A$ is canonically
triangulated and $2$-Calabi-Yau. Moreover, as shown in \cite{BuanMarshReinekeReitenTodorov06},
it satisfies (C1) and (C2) and
the image in $\cc_{A}$ of the free $A$-module of rank one is a
cluster tilting object $T$ whose endoquiver is
canonically isomorphic to $Q$.

\subsubsection{The example $A_3$} \label{sss:cluster-category-A3}
As an example, let us consider the following quiver $Q$
\[
\xymatrix{1 \ar[r] & 2 \ar[r] & 3}\ko
\]
whose underlying graph is a Dynkin diagram of type $A_3$. As recalled
above, the indecomposable objects of $\cd^b(A)$ are isomorphic
to shifted indecomposable modules. Hence, in our example, using
the notations of section~\ref{ss:KrullSchmidtCategories}, the vertices of the
quiver of the category $\cd^b(A)$ correspond to the complexes $\Sigma^p (P_i/P_j)$,
where $p\in \Z$ and $1 \leq i<j\leq 3$. Its arrows have been
computed by Happel \cite{Happel87} \cite{Happel88}. It turns out that the
quiver of $\cd^b(A)$ is isomorphic to the {\em repetition}
of $Q$, \cf \cite{Happel87} \cite{Keller07a}. In our example, this is the infinite band:
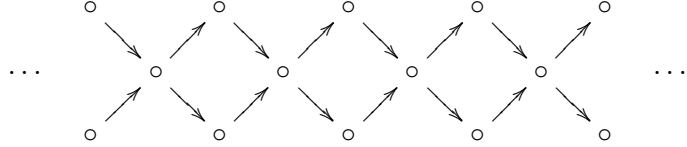
\[
\begin{xy} 0;<0.7pt,0pt>:<0pt,-0.7pt>::
(-35,35) *+{\ldots},
(0,69) *+{\circ} ="0",
(0,0) *+{\circ} ="1",
(35,35) *+{\circ} ="2",
(69,69) *+{\circ} ="3",
(69,0) *+{\circ} ="4",
(103,35) *+{\circ} ="5",
(138,69) *+{\circ} ="6",
(138,0) *+{\circ} ="7",
(172,35) *+{\circ} ="8",
(207,69) *+{\circ} ="9",
(207,0) *+{\circ} ="10",
(241,35) *+{\circ} ="11",
(275,69) *+{\circ} ="12",
(275,0) *+{\circ} ="13",
(310,35) *+{\ldots} ="14",
"0", {\ar"2"},
"1", {\ar"2"},
"2", {\ar"3"},
"2", {\ar"4"},
"3", {\ar"5"},
"4", {\ar"5"},
"5", {\ar"6"},
"5", {\ar"7"},
"6", {\ar"8"},
"7", {\ar"8"},
"8", {\ar"9"},
"8", {\ar"10"},
"9", {\ar"11"},
"10", {\ar"11"},
"11", {\ar"12"},
"11", {\ar"13"},
\end{xy}
\]
The correspondence between its vertices and the
indecomposable objects of $\cd^b(A)$ is indicated by
the following diagram:
\begin{equation} \label{eq:der-cat-A3}
\begin{xy} 0;<1.05pt,0pt>:<0pt,-0.7pt>::
(0,69) *+{\tau P_1} ="0",
(0,0) *+{\Sigma^{-1}(P_2/P_1)} ="1",
(35,35) *+{\tau P_2} ="2",
(69,69) *+{P_1} ="3",
(69,0) *+{\tau P_3} ="4",
(103,35) *+{P_2} ="5",
(138,69) *+{P_2/P_1} ="6",
(138,0) *+{P_3} ="7",
(172,35) *+{P_3/P_1} ="8",
(207,69) *+{P_3/P_2} ="9",
(207,0) *+{\Sigma P_1} ="10",
(241,35) *+{\Sigma P_2} ="11",
(275,69) *+{\Sigma P_3} ="12",
(275,0) *+{\Sigma (P_2/P_1)} ="13",
(310,35) *+{\Sigma (P_3/P_1).} ="14",
"0", {\ar"2"},
"1", {\ar"2"},
"2", {\ar"3"},
"2", {\ar"4"},
"3", {\ar"5"},
"4", {\ar"5"},
"5", {\ar"6"},
"5", {\ar"7"},
"6", {\ar"8"},
"7", {\ar"8"},
"8", {\ar"9"},
"8", {\ar"10"},
"9", {\ar"11"},
"10", {\ar"11"},
"11", {\ar"12"},
"11", {\ar"13"},
"12", {\ar"14"},
"13", {\ar"14"},
\end{xy}
\end{equation}
The objects $\tau P_i$ are the following complexes:
$\tau P_3 = \Sigma^{-1}(P_3/P_2)$,
$\tau P_2 = \Sigma^{-1}(P_3/P_1)$ and
$\tau P_1=\Sigma^{-1}P_3$.
The fully faithful embedding $\mod(A) \to \cd^b(A)$ identifies
the quiver of $\mod(A)$ with the triangle formed by the
objects $P_i/P_j$. Notice that the `meshes' in this triangle
correspond to short exact sequences and thus yield triangles
in the derived category. More generally, each mesh of the whole diagram
comes from a triangle in the
derived category (in fact, it comes from a so-called
Auslander-Reiten triangle, \cf \cite{Happel87}).
The suspension functor $\Sigma$ induces the glide reflection
obtained by reflecting at the horizontal symmetry axis and
translating by two units to the right. The Auslander-Reiten
translation $\tau$ induces the shift by one unit to the left.
The effect of the Serre functor $S=\tau \Sigma$ is the glide
reflection whose translation is the shift by one unit
to the right. The auto-equivalence
\[
S^{-1} \Sigma^2 = \tau^{-1} \Sigma \ko
\]
which appears in the definition of the cluster category,
is the glide reflection taking the $\tau P_i$ to the
$\Sigma P_i$. As shown in \cite{BuanMarshReinekeReitenTodorov06}, we
obtain the quiver of the cluster category $\cc_A$ as the quotient of the
quiver of $\cd^b(A)$ by the action of the automorphism induced by
$S^{-1} \Sigma^2 = \tau^{-1} \Sigma$. Thus, in our example, this
quiver is obtained by cutting out the fattened triangle bordered by the
$\tau P_i$ and the $\Sigma P_i$ and then identifying each vertex
$\tau P_i$ with the corresponding $\Sigma P_i$. The quiver
of $\cc_A$ has nine vertices
so that the cluster category $\cc_A$ has nine indecomposables
(up to isomorphism). The direct sum $T$ of the images of the
$P_i$ in the cluster category is the canonical cluster-tilting
object of $\cc_A$.

\subsubsection{Cluster categories from algebras of global dimension $2$}
Since $k$ is algebraically closed, every finite-dimensional $k$-algebra
of global dimension at most one is Morita-equivalent to $kQ$ for some quiver
$Q$ without oriented cycles. Thus, the cluster category $\cc_A$ is
in fact defined for every finite-dimensional algebra $A$ of global
dimension at most one. If $A$ is an algebra of finite but
arbitrary global dimension, the derived category $\cd^b(A)$
still has a Serre functor (given by the derived tensor
product with $DA$) but the orbit category is no longer
triangulated in general. In recent work \cite{Amiot09}, Amiot has
extended the construction of the (triangulated) cluster category to certain
algebras $A$ of global dimension at most~$2$: Let $A$ be such an
algebra and let $\Pi_3(A)$ be its $3$-Calabi-Yau completion in
the sense of \cite{Keller11b}. By definition, $\Pi_3(A)$
is the tensor algebra
\[
T_A(X) = A \oplus X \oplus (X\ten_A X) \oplus \ldots
\]
of a cofibrant replacement of the complex of bimodules
$X=\Sigma^2\RHom_{A^e}(A,A^e)$, where $A^e=A^{op}\ten_k A$.
Its homology is given by
\[
H^n(\Pi_3(A))= \bigoplus_{p\geq 0} H^n((\Sigma^2 S^{-1})^p A).
\]
In particular, it vanishes in degrees $n\leq 0$.
Then Amiot defines the cluster category $\cc_A$ as the triangle quotient
\[
\per(\Pi_3(A))/\cd_{fd}(\Pi_3(A)) \ko
\]
where, for a dg algebra $B$, the {\em perfect derived category $\per(B)$}
is the thick subcategory of the derived category $\cd(B)$ generated
by the free module $B$ and the {\em finite-dimensional derived category
$\cd_{fd}(B)$} is the full subcategory of $\cd(B)$ whose objects
are the dg modules $M$ whose homology $H^*(M)$ is of finite total dimension.
Amiot shows in \cite{Amiot09} that if $A$ is of global dimension at
most one, this definition agrees with the definition given above.
If $A$ is of global dimension at most $2$, then in general, the
category $\cc_A$ is not $\Hom$-finite. However, we have the following
theorem.

\begin{theorem}[Amiot  \protect{\cite{Amiot09}}] \label{thm:Hom-finite-cluster-category}
Suppose that the functor
\[
\Tor_2^A(?,DA): \mod A \to \mod A
\]
is nilpotent (\ie it vanishes when raised to a sufficiently high power).
\begin{itemize}
\item[a)] The category $\cc_A$ is $\Hom$-finite and $2$-Calabi-Yau.
\item[b)] The image $T$ of $A$ in $\cc_A$ is a cluster tilting
object.
\item[c)] The endomorphism algebra of $T$ in $\cc_A$
is isomorphic to $H^0(\Pi_3(A))$ and the quiver of the endomorphism algebra is obtained
from that of $A$ by adding, for each pair of vertices
$(i,j)$, a number of arrows equal to
\[
\dim \Tor_2^A(S_j, S_i^{op})
\]
from $i$ to $j$, where $S_j$ is the simple right module
associated with $j$ and $S_i^{op}$ the simple left module
associated with $i$.
\end{itemize}
\end{theorem}
As an example, let us consider the tensor product $A$ of two copies
of the path algebra $kQ'$ of the quiver $1 \to 2$ (equivalently, the
tensor product of two copies of the algebra $T_2(k)$ of lower triangular
$2\times 2$-matrices). This algebra is isomorphic to the quotient
of the path algebra of the square
\[
\xymatrix{ (1,2) \ar[r]^\alpha & (2,2) \\
(1,1) \ar[u]^\beta \ar[r]_\delta & (2,1) \ar[u]_\gamma}
\]
by the two-sided ideal generated by the relator $\alpha \beta-\gamma\delta$.
Using this description or lemma~\ref{lemma:quiver-of-tensor-product} below,
one sees that the quiver of $A$ is isomorphic to the tensor product
$Q'\ten Q'$. By a direct computation or using lemma~\ref{lemma:nilpotence}
below, one sees that the algebra $A$ satisfies the assumption of the
theorem. So we obtain that the cluster category $\cc_A$ is $\Hom$-finite
and $2$-Calabi-Yau, that the image $T$ of $A$ in $\cc_A$ is a cluster-tilting
object and that the quiver of its endomorphism algebra is isomorphic to
\[
\xymatrix{ (1,2) \ar[r]^\alpha & (2,2) \ar[dl]^\rho\\
(1,1) \ar[u]^\beta \ar[r]_\delta & (2,1), \ar[u]_\gamma}
\]
that is, to the triangle product $Q'\boxten Q'$.
By proposition~\ref{prop:Calab-Yau-completion-via-quiver-with-potential} below,
the endomorphism algebra itself is isomorphic to the quotient of the path algebra
of this quiver by the two-sided ideal generated by the relators
\[
\alpha\beta - \gamma\delta \ko \beta \rho\ko \rho \alpha \ko \delta\rho\ko \rho \gamma.
\]
We have thus obtained a $2$-Calabi-Yau realization of the triangle
product $Q'\boxten Q'$ for the quiver $Q': 1 \to 2$ and a precise description
of the endomorphism algebra of a cluster-tilting object. In the following
sections, our aim is to generalize this example to triangle products
$Q' \boxten Q''$ of two arbitrary alternating Dynkin quivers.

\subsection{$2$-Calabi-Yau realizations of triangle products}
\label{ss:2-CY-realizations-of-triangle-products}
Let $Q$ and $Q'$ be two finite quivers without oriented cycles.
Let $k$ be a field.
The path algebras $kQ$ and $kQ'$ are then finite-dimensional
algebras of global dimension at most one and the tensor
product $kQ\ten_k kQ'$ is a finite-dimensional algebra
of global dimension at most $2$.

\begin{lemma} \label{lemma:quiver-of-tensor-product}
The quiver of the finite-dimensional $k$-algebra $kQ\ten kQ'$
is isomorphic to $Q\ten Q'$.
\end{lemma}

\begin{proof} Recall that if $B$ is a finite-dimensional
algebra whose simple modules are one-dimensional, the vertices
of the quiver of $B$ are in bijection with the isomorphism classes
of the (simple) right $B$-modules and the number of arrows
from the vertex of $S$ to that of $S'$ equals the dimension
of the first extension group of $S'$ by $S$. Equivalently,
it equals the multiplicity of the projective cover $P_S$
of $S$ in the first term (not the zeroth term!)
of a minimal projective
presentation of $S'$. Now the simple modules of $kQ\ten kQ'$
are the tensor products $S_i\ten S_{i'}$, where $S_i$
is the simple right $kQ'$-module associated with a vertex
$i$ of $Q$ and similarly for $S_{i'}$. We obtain a minimal
projective presentation of $S_i\ten S_{i'}$ by tensoring
a minimal projective presentation of $S_i$ with a minimal
projective presentation of $S_{i'}$. Now the minimal
projective presentation of $S_i$ is of the form
\[
\xymatrix{ \bigoplus P_{s(\alpha)}  \ar[r] & P_i}
\]
where the sum ranges over all arrows $\alpha$ with target $i$,
the vertex $s(\alpha)$ is the source of the arrow $\alpha$ and
$P_{s(\alpha)}$ the projective cover of the simple $S_{s(\alpha)}$. Thus, the
minimal projective presentation of $S_i\ten S_{i'}$ has
its first term isomorphic to the direct sum of modules
$P_{s(\alpha)} \ten P_{i'}$ and $P_i \ten P_{s(\alpha')}$,
where $\alpha$ runs through the arrows of $Q$ with target $i$
and $\alpha'$ through the arrows of $Q'$ with target $i'$.
Clearly, this implies the assertion.
\end{proof}

\begin{lemma} \label{lemma:nilpotence}
Suppose $Q$ and $Q'$ are Dynkin quivers. Let $A=kQ\ten_k kQ'$.
Then the functor
\[
\Tor_2^A(?,DA): \mod A \to \mod A
\]
is nilpotent.
\end{lemma}

\begin{proof} Let $S=?\lten_A DA$ be the Serre functor
of $\cd^b(A)$ and let $\cd^b_{\geq 0}(A)$ be the right aisle
of the canonical $t$-structure on $\cd^b(A)$. For $p\in\Z$,
put $\cd^b_{\geq p}(A)=\Sigma^{-p} \cd^b_{\geq 0}(A)$. We
also use analogous notations for $\cd^b(kQ)$ and $\cd^b(kQ')$.
By Proposition~4.9 of \cite{Amiot09}, to show that
$\Tor_2^A(?,DA)$ is nilpotent, it suffices to show
that
\[
(\Sigma^{-2} S)^p(\cd^b_{\geq 0}(A)) \subset \cd^b_{\geq 1}(A)
\]
for sufficiently large integers $p$. Now if $L$
belongs to $\cd^b(kQ)$ and $M$ to $\cd^b(kQ')$, then
we have a canonical isomorphism
\[
(\Sigma^{-2} S)(L\ten M) = (\Sigma^{-1} S L) \ten (\Sigma^{-1} S M) = (\tau L)\ten (\tau M) \ko
\]
where $\tau=\Sigma^{-1} S$ is the Auslander-Reiten translation
of the derived category of $Q$ respectively $Q'$. Since
$Q$ and $Q'$ are Dynkin quivers, it follows from
Happel's description of the derived category of
a Dynkin quiver \cite{Happel87}, that there are
integers $N$ and $N'$ such that
\[
\tau^N (\cd^b_{\geq 0}(kQ)) \subset \cd^b_{\geq 1}(kQ) \quad \mbox{ and }\quad
\tau^{N'} (\cd^b_{\geq 0}(kQ')) \subset \cd^b_{\geq 1}(kQ').
\]
Then we have
\[
(\Sigma^{-2} S)^{N+N'}(L\ten M) = (\tau^{N+N'} L)\ten (\tau^{N+N'} M) \in \cd_{\geq 2}^b(A)
\]
for all $L\in \cd^b_{\geq 0}(kQ)$ and $M\in\cd^b_{\geq 0}(kQ')$.
Since $\cd^b_{\geq 0}(A)$ is the closure under $\Sigma^{-1}$, extensions
and passage to direct factors of such objects $L\ten M$, the claim
follows.
\end{proof}

\begin{corollary} \label{cor:cluster-category-Hom-finite}
Suppose that  $Q$ and $Q'$ are Dynkin quivers.
Let $A=kQ\ten_k kQ'$. Then the cluster category
$\cc_A$ is $\Hom$-finite and $2$-Calabi-Yau.
Moreover, the image $T$ of $A$ in $\cc_A$ is a cluster tilting
object. The endoquiver of $T$ is
canonically isomorphic to the triangle product $Q\boxten Q'$.
\end{corollary}

\begin{proof} Except for the last claim, this follows directly from
Theorem~\ref{thm:Hom-finite-cluster-category}
and Lemma~\ref{lemma:nilpotence}. To determine the
endoquiver, by Theorem~\ref{thm:Hom-finite-cluster-category}
and Lemma~\ref{lemma:quiver-of-tensor-product}, it suffices to
compute the second torsion groups between the simple $A$-modules.
By the proof of Lemma~\ref{lemma:quiver-of-tensor-product},
these are tensor products of simple modules. We compute
\[
\Tor_2^A(S_j \ten_k S_{j'}, S_i^{op}\ten S_{i'}^{op}) =
\Tor_1^{kQ}(S_j, S_{i}^{op})\ten \Tor_1^{kQ'}(S_{j'}, S_{i'}^{op})
= e_j (kQ_1) e_i \ten e_{j'} (kQ'_1) e_{i'} \ko
\]
where we write $kQ_1$ for the vector space generated by
the arrows of $Q$ and consider it as a bimodule over
the semi-simple subalgebra $\prod_{i\in Q_0} ke_i$ of $kQ$.
This shows the claim.
\end{proof}

\subsection{Reminder on quivers with potential}
\label{ss:reminder-quivers-with-potential} We recall results
from Derksen-Weyman-Zelevinsky's fundamental article
\cite{DerksenWeymanZelevinsky08}. Let $Q$ be a finite
quiver and $k$ a field. Let $\hat{kQ}$ be the {\em completed path algebra}, \ie the
completion of the path algebra at the ideal generated by the arrows
of $Q$. Thus, $\hat{kQ}$ is a topological algebra and the paths of
$Q$ form a topologial basis so that the underlying vector space of
$\hat{kQ}$ is
\[
\prod_{p \mbox{ \tiny path}} kp.
\]
The {\em continuous zeroth Hochschild homology} of $\hat{kQ}$ is the vector
space $\hh_0$ obtained as the quotient of $\hat{kQ}$ by the closure of
the subspace generated by all commutators. It admits a topological
basis formed by the {\em cycles} of $Q$, \ie the orbits
of paths $p=(i|\alpha_m| \ldots |\alpha_1| i)$ of any length $m\geq 0$ with
identical source and target under the action of the cyclic group of order $m$.
In particular, the space $\hh_0$ is a product of copies of $k$ indexed
by the vertices if $Q$ does not have oriented cycles.
For each arrow $a$ of $Q$, the {\em cyclic derivative with respect to $a$}
is the unique linear map
\[
\del_a : \hh_0 \to \hat{kQ}
\]
which takes the class of a path $p$ to the sum
\[
\sum_{p=uav} vu
\]
taken over all decompositions of $p$ as a concatenation
of paths $u$, $a$, $v$, where $u$ and $v$ are of length $\geq 0$.
A {\em potential} on $Q$ is an element $W$ of $\hh_0$ whose expansion
in the basis of cycles does not involve cycles of length $\leq 1$.
A potential is {\em reduced} if it does not involve cycles of length
$\leq 2$. The {\em Jacobian algebra} $\cp(Q,W)$ associated to a quiver
$Q$ with potential $W$ is the quotient of the completed path algebra
by the closure of the $2$-sided ideal generated by the cyclic derivatives
of the elements of $W$. If the potential $W$ is reduced and the Jacobian
algebra $\cp(Q,W)$ is finite-dimensional, its quiver is isomorphic to $Q$.
Two quivers with potential $(Q,W)$ and $(Q',W')$
are {\em right equivalent} if $Q_0=Q_0'$ and there exists a $k$-algebra
isomorphism $\phi: \hat{kQ} \to \hat{kQ'}$ such that $\phi$ induces
the identity on the subalgebra $\prod_{Q_0} k$ and takes $W$ to $W'$.
One of the main theorems of \cite{DerksenWeymanZelevinsky08} is the
existence, for each quiver with potential $(Q,W)$, of a reduced
quiver with potential $(Q_{red}, W_{red})$, unique up to right
equivalence, such that $(Q,W)$ is right equivalent to the sum
of $(Q_{red}, W_{red})$ with a trivial quiver with potential
(\cf \cite{DerksenWeymanZelevinsky08} for the definition). In particular,
the Jacobian algebras of $(Q,W)$ and $(Q_{red}, W_{red})$ are
isomorphic. The quiver with potential $(Q_{red}, W_{red})$ is
the {\em reduced part} of $(Q,W)$.

Let $(Q,W)$ be a quiver with potential such that $Q$ does not
have loops. Let $i$ be a vertex of $Q$
not lying on a $2$-cycle. The {\em mutation} $\mu_i(Q,W)$ is
defined as the reduced part of the quiver with potential
$\tilde{\mu}_i(Q,W)=(Q',W')$, which is defined as follows:
\begin{itemize}
\item[a)]
\begin{itemize}
\item[(i)] To obtain $Q'$ from $Q$, add a new arrow $[\alpha\beta]$
for each pair of arrows $\alpha:i\to j$ and $\beta: l \to i$ of $Q$ and
\item[(ii)] replace each arrow $\gamma$ with source or target $i$ by
a new arrow $\gamma^*$ with $s(\gamma^*)=t(\gamma)$ and $t(\gamma^*)=s(\gamma)$.
\end{itemize}
\item[b)] Put $W'=[W]+\Delta$, where
\begin{itemize}
\item[(i)] $[W]$ is obtained from $W$ by
replacing, in a representative of $W$ without cycles passing through
$i$, each occurrence of $\alpha\beta$ by $[\alpha\beta]$, for each pair
of arrows $\alpha:i\to j$ and $\beta: l \to i$ of $Q$;
\item[(ii)] $\Delta$ is the sum of the cycles $[\alpha\beta]\beta^* \alpha^*$
taken over all pairs of arrows $\alpha:i\to j$ and $\beta: l \to i$ of $Q$.
\end{itemize}
\end{itemize}
Then $i$ is not contained in a $2$-cycle of $\mu_i(Q,W)$ and
$\mu_i(\mu_i(Q,W))$ is shown in \cite{DerksenWeymanZelevinsky08}
to be right equivalent to $(Q,W)$.
Note that if neither $Q$ nor the quiver $Q'$ in $(Q',W')=\mu_i(Q,W)$
have loops or $2$-cycles, then $Q$ and $Q'$ are linked
by the quiver mutation rule.

Let $(Q,W)$ be a quiver with potential. The{\em Ginzburg dg algebra}
$\Gamma=\Gamma(Q,W)$ associated with $(Q,W)$ is due to V.~Ginzburg
\cite{Ginzburg06}. We use the variant defined in \cite{KellerYang11}.
The Ginzburg algebra is a topological differential graded
algebra whose homology vanishes in
(cohomological) degrees $>0$ and whose zeroth homology
is isomorphic to the Jacobian algebra $\cp(Q,W)$. We refer
to \cite{KellerYang11} for the definitions of the
derived category $\cd\Gamma$, the perfect derived category
$\per(\Gamma)$ and the finite-dimensional derived category
$\cd_{fd}(\Gamma)$. All three are triangulated categories and
$\cd_{fd}(\Gamma)$ is a thick subcategory of $\per(\Gamma)$,
which is a thick subcategory of $\cd(\Gamma)$. The objects of $\cd(\Gamma)$
are all differential graded right $\Gamma$-modules.
The {\em cluster category} $\cc_{\Gamma}$ is defined as the
triangle quotient of $\per(\Gamma)$ by $\cd_{fd}(\Gamma)$.

Let us assume that the Jacobian algebra of $(Q,W)$ is finite-dimensional.
Then the cluster category $\cc_\Gamma$ is a  ($\Hom$-finite!)
$2$-Calabi-Yau category and the image $T$ in $\cc_\Gamma$ of
the free right $\Gamma$-module of rank one is a cluster tilting
object by \cite{Amiot09}. Moreover, as shown in \cite{Amiot09},
the endomorphism algebra of $T$ in $\cc_\Gamma$ is isomorphic
to the Jacobian algebra and thus, if $(Q,W)$ is reduced,
the quiver of the endomorphism algebra of $T$ is isomorphic
to $Q$.

Let $(Q,W)$ be a quiver with potential and $i$ a vertex
of $Q$ not lying on a $2$-cycle. Let $\Gamma=\Gamma(Q,W)$
and $\Gamma'=\Gamma(\mu_i(Q,W))$. For each vertex $j$ of
$Q$, let $e_j$ be the associated idempotent and $P_j=e_j \Gamma$
the right ideal generated by $e_j$. We use analogous notations
for $\Gamma'$. It is shown in \cite{KellerYang11} that there
is a triangle equivalence
\[
F: \cd(\Gamma') \iso \cd(\Gamma)
\]
which induces triangle equivalences in the perfect and the finite-dimensional
derived categories and sends the objects $P'_j$, $j\neq i$, to the $P_j$
and the object $P_i$ to the cone over the morphism of dg modules
\[
P_i \to \bigoplus P_j \ko
\]
where the sum is taken over all arrows in $Q$ with source $i$ and the
components of the morphism are the left multiplications with the
corresponding arrows. If the Jacobian algebra of $(Q,W)$ is finite-dimensional,
the equivalence $F$ induces a triangle equivalence
\[
\cc_{\Gamma'} \iso \cc_\Gamma
\]
which sends $T'=\Gamma'$ to $\mu_i(T)$, by the exchange triangles
of Lemma~\ref{lemma:exchange-triangles}. In particular, the endomorphism
algebra of $\mu_i(T)$ is isomorphic to the Jacobian algebra of $\mu_i(Q,W)$
(this also follows from Theorem~5.1 in \cite{BuanIyamaReitenSmith11}) and
its quiver is the quiver $Q'$ appearing in $(Q',W')=\mu_i(Q,W)$.

\subsection{Description of $\cc_A$ by a quiver with potential}
\label{ss:description-via-quiver-with-potential}
Let $k$ be a field and $Q$ and $Q'$ finite quivers without
oriented cycles. Let $A=kQ\ten_k kQ'$. The cluster category
$\cc_A$ defined in section~\ref{ss:2-CY-realizations} is
associated with the $3$-Calabi-Yau completion $\Pi_3(A)$.
We can describe this dg algebra conveniently using
quivers with potentials.

For arrows $\alpha: i\to j$ of $Q$ and $\alpha':i'\to j'$
of $Q'$, we use notations as in the following full
subquiver of $Q\boxten Q'$:
\[
\xymatrix{ (i,j') \ar[r]^{(\alpha, j')} & (j,j') \ar[ld]|-{\rho(\alpha,\alpha')} \\
(i,i') \ar[u]^{(i,\alpha')} \ar[r]_{(\alpha,i')} & (j,i') \ar[u]_{(j,\alpha')}
}
\]

\begin{proposition} \label{prop:Calab-Yau-completion-via-quiver-with-potential}
The $3$-Calabi-Yau completion $\Pi_3(A)$
is quasi-isomorphic to the Ginz\-burg algebra $\Gamma(Q\boxten Q', W)$
associated with the quiver $Q\boxten Q'$ and the potential
\[
W= \sum \rho(\alpha,\alpha')\circ((j,\alpha')\circ(\alpha,i')  - (\alpha,j')\circ (i,\alpha')) \ko
\]
where the sum ranges over all pairs of arrows $\alpha$ of $Q$ and $\alpha'$ of $Q'$.
\end{proposition}

\begin{proof}A closer examination of the proof of Corollary~\ref{cor:cluster-category-Hom-finite}
shows that the minimal relations for the algebra $A$ are precisely
the commutativity relations $(\alpha,i')\circ (j,\alpha') - (\alpha,j')\circ (i,\alpha')$
associated with pairs of arrows $\alpha$ of $Q$ and $\alpha'$ of $Q'$. Now Theorem~6.9 of
\cite{Keller11b} shows that $\Pi_3(A)$ is quasi-isomorphic to the non completed
Ginzburg algebra $\Gamma'(Q\boxten Q', W)$. Since the homology in each degree
of this dg algebra is finite-dimensional, the canonical morphism
\[
\Gamma'(Q\boxten Q', W) \to \Gamma(Q\boxten Q', W)
\]
is a quasi-isomorphism.
\end{proof}

\section{Cluster combinatorics from Calabi-Yau triangulated categories}
\label{s:cluster-combinatorics-from-triangulated-categories}
Here we describe how to associate cluster combinatorial data with
objects in $2$-Calabi-Yau categories with a cluster tilting object.
We start with the categorical lift of the most basic operation: quiver mutation.

\subsection{Decategorification: quiver mutation} \label{ss:decategorification:quiver-mutation}
Let $k$ be an algebraically closed field and $\cc$ a $2$-Calabi-Yau
category with cluster tilting object $T$. Let $T_1$ be an indecomposable direct
factor of $T$.

\begin{theorem}[Iyama-Yoshino \protect{\cite{IyamaYoshino08}}]
\label{thm:2-CY-mutation} Up to isomorphism, there is a unique
indecomposable object $T_1^*$ not isomorphic to $T_1$ such that the
object $\mu_1(T)$ obtained from $T$ by replacing the indecomposable
summand $T_1$ with $T_1^*$ is cluster tilting.
\end{theorem}

We call $\mu_1(T)$ the {\em mutation} of $T$ at $T_1$. If $T_1$, \ldots, $T_n$
are the pairwise non isomorphic indecomposable direct summands of
$T$ and $\mathbb{T}_n$ is the $n$-regular tree with distinguished
vertex $t_0$, we define cluster tilting objects $T(t)$ for each
vertex $t$ of $\mathbb{T}_n$ in such a way that $T(t_0)=T$ and,
whenever $t$ and $t'$ are linked by an edge labeled $k$, we
have $T(t')=\mu_k(T(t))$. Notice that the construction simultaneously
yields a natural numbering of the indecomposable summands
$T_i(t')$ of $T(t')$ in such a way that $T_i(t)\iso T_i(t')$
for all $i\neq k$.

\begin{lemma} \label{lemma:exchange-triangles}
Suppose that endoquiver
of $T$ does not have a loop at the vertex corresponding to $T_1$.
Then the space $\Ext^1(T_1, T_1^*)$ is one-dimensional and we have
non split triangles
\begin{equation} \label{eq:exchange-triangles}
\xymatrix{
T_1^* \ar[r]^i & E \ar[r]^p & T_1 \ar[r]^{\eps} & \Sigma T_1^*} \mbox{ and }
\xymatrix{
T_1 \ar[r]^{i'} & E' \ar[r]^{p'} & T_1^* \ar[r]^{\eps'} & \Sigma T_1} \ko
\end{equation}
where
\[
E= \bigoplus_{\stackrel{\mbox{\tiny arrows}}{i\to 1}} T_i
\mbox{ and }
E'=\bigoplus_{\stackrel{\mbox{\tiny arrows}}{1 \to j}} T_j
\]
and the components of the morphisms $i$, $i'$, $p$ and $p'$
represent the corresponding arrows of the endoquivers
of $T$ respectively $\mu_1(T)$.
\end{lemma}

The lemma is well-known to the experts. We include a proof for the
convenience of the reader.

\begin{proof} Let $T'$ be the direct sum of the indecomposable direct
factors of $T$ not isomorphic to $T_1$. By the construction of $T_1^*$
in \cite{IyamaYoshino08}, we have a triangle
\[
\xymatrix{
T_1^* \ar[r]^i & E \ar[r]^p & T_1 \ar[r]^{\eps} & \Sigma T_1^*} \ko
\]
where $p$ is a minimal right $\add(T')$-approximation of $T_1$. Since $T$
is rigid, we obtain an exact sequence
\begin{equation} \label{eq:proj-presentation}
\cc(T,E) \to \cc(T,T_1) \to \cc(T, \Sigma T_1^*) \to 0.
\end{equation}
Since the endoquiver of $T$ does not have a loop at $T_1$, each non isomorphism
from $T_1$ to itself factors through $E$. Since $k$ is algebraically
closed, it follows that the space $\cc(T, \Sigma T_1^*)$ is one-dimensional
and isomorphic to the simple quotient of the indecomposable projective $\cc(T,T)$-module
$\cc(T,T_1)$. Moreover, the sequence~\ref{eq:proj-presentation} is a minimal
projective presentation of this simple quotient. This yields the description
of $E$. By applying the same argument to $\cc^{op}$, we obtain the description
of $E'$.
\end{proof}

\begin{theorem}[Buan-Iyama-Reiten-Scott \protect{\cite{BuanIyamaReitenScott09}}]
\label{thm:quiver-2-CY-mutation}
Suppose that the endoquivers
$Q$ and $Q'$ of $T$ and $T'=\mu_1(T)$
do not have loops nor $2$-cycles. Then $Q'$ is the mutation of
$Q$ at the vertex $1$.
\end{theorem}

We define a cluster tilting object $T'$ to be {\em reachable
from $T$} if there is a path
\[
\xymatrix{t_0 \ar@{-}[r] & t_1 \ar@{-}[r]& \ldots &  t_{N} \ar@{-}[l] }
\]
in $\mathbb{T}_n$
such that $T(t_N)=T'$ and the quiver of $\End(T(t_i))$ does not have loops
nor $2$-cycles for all $0\leq i\leq N$. It follows from the theorem
above that in this case, for each $0\leq i\leq N$, the endoquiver $Q_{T(t_i)}$
of $T(t_i)$ is obtained from the endoquiver of $T=T(t_0)$ by
the corresponding sequence of mutations, \ie we have
$Q_{T(t_i)}= Q(t_i)$ for all $1\leq i\leq N$.
We define a rigid indecomposable object of $\cc$ to be {\em reachable from $T$}
if it is a direct summand of a reachable cluster tilting object.

\begin{example} \label{ex:linear-A3-mutation}
Consider the cluster category  $\cc$ of the quiver $1 \to 2\to 3$
from example~\ref{sss:cluster-category-A3} with its cluster tilting object
$T$ which is the sum of $P_1$, $P_2$ and $P_3$. If we mutate $T$ at its
summand $T_1=P_1$, we find $T_1^*=P_2/P_1$ and
the exchange triangles~\ref{eq:exchange-triangles} are
\[
\xymatrix{P_2/P_1 \ar[r] & 0 \ar[r] & P_1 \ar[r] & \Sigma(P_2/P_1)}
\quad\mbox{and}\quad
\xymatrix{P_1 \ar[r] & P_2 \ar[r] & P_2/P_1 \ar[r] & \Sigma P_1}.
\]
Notice that the third morphism of the first triangle is the
composition of the isomorphism $P_1 \iso \tau(P_2/P_1)$, already
present in the derived category, with the isomorphism $\tau(P_2/P_1) \iso \Sigma(P_2/P_1)$
obtained thanks to the passage to the cluster category. We obtain the new
quiver
\[
\xymatrix@R=0.4cm@C=0.4cm{  & P_3 \\ P_2 \ar[ur] \ar[dr] & \\  & P_2/P_1 }
\]
which is indeed isomorphic to the mutation of $1\to 2 \to 3$ at the
vertex $1$. If we mutate $T$ at its summand $P_2$, we obtain $P_2^* = P_3/P_1$
and the exchange triangles
\[
\xymatrix{P_3/P_2 \ar[r] & P_1 \ar[r] & P_2 \ar[r] & \Sigma(P_3/P_2)}
\quad\mbox{and}\quad
\xymatrix{P_2 \ar[r] & P_3\ar[r] & P_3/P_2 \ar[r] & \Sigma P_2.}
\]
The new quiver is
\[
\xymatrix@R=0.4cm@C=0.4cm{ & P_3 \ar[rd] &  \\
P_1 \ar[ru] & & P_3/P_2 \ar[ll]}
\]
and it is indeed isomorphic to the mutation of $1\to 2 \to 3$
at the vertex $2$.
\end{example}

\begin{example} \label{ex:alternating-A3-mutation} Let us consider
the following alternating Dynkin quiver of type $A_3$:
\[
Q: \xymatrix{2 & 1 \ar[l] \ar[r] & 3}.
\]
If we identify this quiver with the triangle product $Q\boxten Q'$,
where $Q'$ consists of a single vertex (considered as a source) without
arrows, then the mutation sequence $\mu_\boxten$ 
defined in formula~\ref{eq:inverse-Zamolodchikov-mutation}
simplifies to
\[
\mu_\boxten = \mu_- \mu_+ \ko
\]
where $\mu_+$ is the mutation at the source $1$ and $\mu_-$ the
sequence of mutations $\mu_2 \mu_3$ at the sinks $2$ and $3$.
Let us lift the composition $\mu_\boxten$ to the categorical level
and check that the action of $\mu_\boxten^{h+h'} = \mu_\boxten^{h+2}$
is indeed the identity,
where $h=4$ and $h'=2$ are the Coxeter numbers of $A_3$ and $A_1$.
Let $A$ be the path algebra (\cf section~\ref{ss:path-algebras})
and for each vertex $i$, let $P_i = e_i A$ and
$I_i = \Hom(A e_i, k)$ (these are the indecomposable projective,
respectively injective, $A$-modules, up to isomorphism). In analogy
with example~\ref{sss:cluster-category-A3}, we draw a piece of the quiver of
the derived category $\cd^b(A)$.
\begin{equation} \label{eq:der-cat-A3-alt}
\begin{xy} 0;<1.05pt,0pt>:<0pt,-0.7pt>::
(0,69) *+{\Sigma^{-1}P_3} ="0",
(0,0) *+{\Sigma^{-1}P_2} ="1",
(35,35) *+{\tau P_1} ="2",
(69,69) *+{\tau P_2} ="3",
(69,0) *+{\tau P_3} ="4",
(103,35) *+{P_1} ="5",
(138,69) *+{P_2} ="6",
(138,0) *+{P_3} ="7",
(172,35) *+{I_1} ="8",
(207,69) *+{I_3} ="9",
(207,0) *+{I_2} ="10",
(241,35) *+{\Sigma P_1} ="11",
(275,69) *+{\Sigma P_3} ="12",
(275,0) *+{\Sigma P_2} ="13",
(310,35) *+{\Sigma I_1.} ="14",
"0", {\ar"2"},
"1", {\ar"2"},
"2", {\ar"3"},
"2", {\ar"4"},
"3", {\ar"5"},
"4", {\ar"5"},
"5", {\ar"6"},
"5", {\ar"7"},
"6", {\ar"8"},
"7", {\ar"8"},
"8", {\ar"9"},
"8", {\ar"10"},
"9", {\ar"11"},
"10", {\ar"11"},
"11", {\ar"12"},
"11", {\ar"13"},
"12", {\ar"14"},
"13", {\ar"14"},
\end{xy}
\end{equation}
Let us point out that the derived category in this example
is equivalent to that in example~\ref{sss:cluster-category-A3} by
the derived functor of the Bernstein-Gelfand-Ponomarev
reflection functor associated with the vertex $1$,
\cf \cite{Happel87} \cite{Keller07a}. This explains the
isomorphism between their quivers, which of course respects
the actions of the automorphisms $\Sigma$, $S$ and $\tau$.
If we mutate the initial cluster-tilting object $T=P_1\oplus P_2 \oplus P_3$
at the summand $P_1$, we obtain the exchange triangles
\[
\xymatrix{I_1 \ar[r] & 0 \ar[r] & P_1 \ar[r] & \Sigma I_1}
\quad\mbox{and}\quad
\xymatrix{P_1 \ar[r] & P_2\oplus P_3 \ar[r] & I_1 \ar[r] & \Sigma P_1}.
\]
and the new quiver
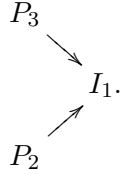
\[
\xymatrix@R=0.4cm@C=0.4cm{P_3 \ar[rd] & \\ & I_1 .\\ P_2 \ar[ru]  }
\]
If we now successively mutate at the summands associated with the
vertices $2$ and $3$ of the original quiver, we successively obtain
the quivers
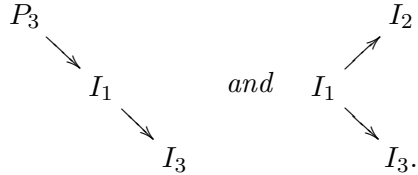
\[
\xymatrix@R=0.4cm@C=0.4cm{P_3 \ar[rd] & & \\ & I_1 \ar[rd] & \\  & & I_3}
\quad\raisebox{-0.9cm}{and}\quad
\xymatrix@R=0.4cm@C=0.4cm{ & I_2 \\ I_1 \ar[ru] \ar[rd] & \\ & I_3.}
\]
What we see in this example is that applying $\mu_\boxten$ to the
initial cluster-tilting object is tantamount to applying the
autoequivalence $\tau^{-1}$. Therefore, if we raise $\mu_\boxten$
to the power $h+h'=4+2$, the resulting sequence of mutations acts on
the initial cluster tilting object like the auto-equivalence $\tau^{-h}\tau^{-2}$. Now the
functor $\tau^{-h} : \cd^b(A) \to \cd^b(A)$ is isomorphic to
$\Sigma^2$ by the classical Theorem~\ref{thm:Gabriel-Happel} below.
On the other hand, the functor induced by $\tau^{-1}$ in the cluster category is
isomorphic to that induced by $\Sigma^{-1}$, by the definition of the
cluster category. So the effect of $\tau^{-h}\tau^{-2}$ in the
cluster category is that of $\Sigma^2 \Sigma^{-2} = \id$.
This is a simple case of `categorical periodicity', proved in
general in section~\ref{eq:der-cat-A3} below.
\end{example}

\subsection{Decategorification: $g$-vectors and tropical $Y$-variables}
\label{ss:decategorification:g-vectors-tropical-Y-variables}
As in section~\ref{ss:decategorification:quiver-mutation}, we assume
that $k$ is an algebrically closed field and $\cc$ a $2$-Calabi-Yau
category with a cluster tilting object $T$.

Let $\ct=\add(T)$ be the full subcategory whose objects are
all direct factors of finite direct sums of copies
of $T$. Let $K_0(\ct)$ be the Grothendieck group
of the additive category $\ct$. Thus, the group
$K_0(\ct)$ is free abelian on the isomorphism
classes of the indecomposable summands of $T$.

\begin{lemma}[\protect{\cite{KellerReiten07}}]
\label{lemma:Keller-Reiten}
For each object $L$ of $\cc$, there is
a triangle
\[
T_1 \to T_0 \to L \to \Sigma T_1
\]
such that $T_0$ and $T_1$ belong to $\ct$. The difference
\[
[T_0] - [T_1]
\]
considered as an element of $K_0(\ct)$ does not depend
on the choice of this triangle.
\end{lemma}
In the situation of the lemma,
we define the {\em index $\ind_T(L)$ of $L$} as
the element $[T_0]-[T_1]$ of $K_0(\ct)$.

\begin{theorem}[\protect{\cite{DehyKeller08}}] \label{thm:Dehy-Keller}
\begin{itemize}
\item[a)] Two rigid objects are isomorphic iff their indices
are equal.
\item[b)] The indices of the indecomposable summands of a cluster tilting
object form a basis of $K_0(\ct)$. In particular, all cluster tilting
objects have the same number of pairwise non isomorphic indecomposable
summands.
\item[c)] In the situation of Lemma~\ref{lemma:exchange-triangles}, if
$T'=\mu_1(T)$ and $L$ is an object of $\cc$, we have
\begin{equation} \label{eq:transformation-of-indices}
\ind_{T'}(L)= \left\{ \begin{array}{ll} \phi_+(\ind_T(L)) & \mbox{if } [\ind_T(L):T_1]\geq 0 \\
\phi_-(\ind_T(L)) & \mbox{if } [\ind_T(L):T_1] \leq 0.
\end{array} \right.
\end{equation}
where $\phi_\pm$ are the linear automorphisms of $K_0(\ct)$ which
fix all classes of indecomposable factors of $T$ not isomorphic to
$T_1$ and send the class of $T_1$ to
\[
\phi_+([T_1]) = -[T_1]+[E] \mbox{ respectively }
\phi_-([T_1]) = -[T_1]+[E'].
\]
\end{itemize}
\end{theorem}

\begin{corollary}[\protect{\cite{FuKeller10}}] \label{cor:decat-tropical-Y-var}
Let $T(t)$ be a cluster tilting object reachable from $T=T(t_0)$.
For each $1\leq j\leq n$, we have
\begin{equation}\label{eq:index-vs-g-vector}
\ind_T(T_j(t))=\sum_{i=1}^n g_{ij}^{t_0}(t) [T_i].
\end{equation}
\end{corollary}
\begin{proof} This follows by induction from Theorem~\ref{thm:Dehy-Keller}
and from the recursive characterization of the $g$-vectors in
equations~\ref{eq:g-vector-recursion1} and \ref{eq:g-vector-recursion2}.
\end{proof}

Notice that a cluster tilting object $T'$ of $\cc$
is also a cluster tilting object
of the opposite category $\cc^{op}$ so that each object
$L$ in $\obj(\cc)=\obj(\cc^{op})$ also has a well-defined
index in $\cc^{op}$ with respect to $T'$; we denote it by
$\ind^{op}_{T'}(L)$.
If we identify the Grothendieck groups of $\add T'$ and $(\add(T'))^{op}$,
this index identifies with $-\ind_{T'}(\Sigma L)$.

\begin{corollary} \label{cor:decat-tropical-Y-variables}
Let $T(t)$ be a cluster tilting object reachable from $T=T(t_0)$.
We have
\begin{equation} \label{eq:opposite-index-vs-tropical-y-variable}
\ind^{op}_{T(t)}(T_j)=\sum_{i=1}^n c_{ji} [T_i(t)] \ko
\end{equation}
where the $c_{ij}$ are given by the tropical $Y$-variable
\begin{equation} \label{eq:opposite-index-vs-tropical-y-variable2}
\eta_i(t)=\prod_{j=1}^n y_j^{c_{ij}} \ko 1\leq i\leq n.
\end{equation}
\end{corollary}

\begin{proof} This follows from the recursive definition \ref{eq:mut-tropical-y-variable}
of the tropical $Y$-variables and Theorem~\ref{thm:Dehy-Keller}.
\end{proof}

\begin{example} \label{ex:trop-y-var}
We continue example~\ref{ex:alternating-A3-mutation}. Let
$T=T(t_0)$ be the direct sum of $P_1$, $P_2$ and $P_3$ in the cluster
category $\cc_A$. Let $T'=\mu_\boxten(T)$. In example~\ref{ex:alternating-A3-mutation},
we have obtained that $T'_1 = I_1$, $T'_2=I_3$ (sic!) and $T'_3=I_2$.
Let us compute the indices of the $T_j=P_j$ with respect to
$T'$ in the opposite of the cluster category. For this, we
have to produce `co-resolutions' of the $T_j$ by objects
belonging to $\add(T')$. Now we have the exact sequences
\begin{align*}
 & \xymatrix{0 \ar[r] & P_1 \ar[r] & I_1 \ar[r] & I_2\oplus I_3 \ar[r] & 0} \ko \\
 & \xymatrix{0 \ar[r] & P_2 \ar[r] & I_1 \ar[r] & I_3 \ar[r] & 0} \ko\\
 & \xymatrix{0 \ar[r] & P_3 \ar[r] & I_1 \ar[r] & I_2 \ar[r] & 0} .
\end{align*}
To obtain these, it is best to identify modules
over the path algebra with representations of the quiver
\[
Q^{op} : \xymatrix{ 2 \ar[r] & 1 & 3 \ar[l]} \ko
\]
\cf section~\ref{ss:path-algebras}. Specifically, these
representations are:
\begin{align*}
P_1 &: \xymatrix{0 \ar[r] & k & 0 \ar[l] } \ko
P_2 : \xymatrix{k \ar[r] & k & 0 \ar[l] } \ko
P_3 : \xymatrix{0 \ar[r] & k & k \ar[l] } \\
I_1 &: \xymatrix{k \ar[r] & k & k \ar[l] } \ko
I_2 : \xymatrix{k \ar[r] & 0 & 0 \ar[l] } \ko
I_3 : \xymatrix{0 \ar[r] & 0 & k \ar[l] }
\end{align*}
So we obtain that the matrix whose coefficients are
the integers $c_{ij}$ defined in the corollary is
\[
\left[ \begin{array}{ccc} 1 & 1 & 1 \\ -1 & -1 & 0 \\ -1 & 0 & -1 \end{array} \right].
\]
Notice again that $T'_2=I_3$ and $T'_3=I_2$!
An easy computation shows that the tropical $Y$-variables associated
with the vertex $t=\mu_\boxten(t_0)$ of the regular tree $\mathbb{T}_3$
are indeed the Laurent monomials
\[
y_1 y_2 y_3 \ko y_1^{-1} y_2^{-1} \ko y_1^{-1} y_3^{-1}
\]
whose exponents appear in the rows of this matrix.
\end{example}

\subsection{Decategorification: cluster variables and $F$-polynomials}
\label{ss:decategorification:cluster-variables-F-polynomials}
Let $k$ be the field of complex numbers. Let $T_1, \ldots, T_n$
be the pairwise non isomorphic indecomposable
direct summands of $T$ and $B$ its endomorphism algebra. Let
$P_i=\Hom(T,T_i)$ be the indecomposable projective right $B$-module
associated with $T_i$, $1\leq i\leq n$. Let $S_i$ be the
simple quotient of $P_i$. For a right $B$-module
$M$, the {\em dimension vector} is the $n$-tuple formed by
the $\dim \Hom_B(P_i, M)$, $1\leq i\leq n$.

For two finite-dimensional
right $B$-modules $L$ and $M$ put
\[
\langle L, M \rangle_a = \dim \Hom(L,M) -\dim \Ext^1(L,M)
                        -\dim \Hom(M,L) +\dim \Ext^1(M,L).
\]
This is the antisymmetrization of a truncated Euler form. A priori
it is defined on the split Grothendieck group of the category
$\mod B$ (\ie the quotient of the free abelian group on the
isomorphism classes divided by the subgroup generated by
all relations obtained from direct sums in $\mod B$).

\begin{proposition}[Palu \protect{\cite{Palu08a}}] The form $\langle, \rangle_a$ descends
to an antisymmetric form on $K_0(\mod B)$. Its matrix in the basis
of the simples is the antisymmetric matrix associated with
the quiver of $B$ (loops and $2$-cycles do not contribute
to this matrix).
\end{proposition}

 In notational accordance with equation~\ref{eq:index-vs-g-vector},
for $L\in \cc$, we define the integer $g_i(L)$ to be the
multiplicity of $[T_i]$ in the index $\ind(L)$, $1\leq i\leq n$.
We define the element $X_L$ of the field $\Q(x_1, \ldots, x_n)$ by
\begin{equation} \label{eq:cluster-char}
X_L= \prod_{i=1}^n x_i^{g_i(L)} \sum_e \chi(\Gr_e(\Ext^1(T,L))) \prod_{i=1}^n x_i^{\langle S_i, e\rangle_a} \ko
\end{equation}
where the sum ranges over all
$n$-tuples $e\in \N^n$, the {\em quiver Grassmannian} $\Gr_e(\Ext^1(T,L))$ is the variety
of all $B$-submodules of the $B$-module $\Ext^1(T,L)$ whose dimension
vector is $e$ and $\chi$ denotes the Euler characteristic of singular cohomology
with coefficients in $\C$. Notice that we have
$X_{T_i}=x_i$, $1\leq i\leq n$. The expression~\ref{eq:cluster-char}
is a vastly generalized form of Caldero-Chapoton's formula \cite{CalderoChapoton06}.
We define the {\em $F$-polynomial} associated with $L$ as the
integer polynomial in the indeterminates $y_1$, \ldots, $y_n$ given by
\[
F_L=\sum_e \chi(\Gr_e(\Ext^1(T,L))) \prod_{i=1}^n y_i^{e_i}
\]

Now let $Q$ be the endoquiver of $T$
in $\cc$. We assume that $Q$ does not have loops or $2$-cycles.
Let $\ca_Q$ be the associated cluster algebra.

\begin{theorem}[Palu \protect{\cite{Palu08a}}] \label{thm:Palu-formula}
If $L$ and $M$ are objects of $\cc$ such that
$\Ext^1(L,M)$ is one-dimensional and
\[
\xymatrix{L \ar[r]^i & E \ar[r] & M \ar[r] & \Sigma L} \mbox{ and }
\xymatrix{M \ar[r]^{i'} & E' \ar[r] & L \ar[r] & \Sigma M}
\]
are `the' two non split triangles, then we have
\begin{align}
\label{eq:Palu-formula-cluster-variables}
X_L X_M &= X_E + X_{E'} \\
\label{eq:Palu-formula-F-polynomials}
F_L F_M &= F_E \prod_{i=1}^n y_i^{d_i} + F_{E'} \prod_{i=1}^n y_i^{d'_i} \ko
\end{align}
where
\[
d_i = \dim \ker(\cc(T_i,\Sigma L) \arr{i_*} \cc(T_i, \Sigma E)) \mbox{ and }
d'_i = \dim \ker(\cc(T_i, \Sigma M) \arr{i'_*} \cc(T_i, \Sigma E')).
\]
\end{theorem}

\begin{proof} For $L\in\cc$, let $X_L^{\mbox{\tiny Palu}}$ be the polynomial defined
by Palu in \cite{Palu08a}. We then have $X_L=X^{\mbox{\tiny Palu}}_{\Sigma L}$, as follows from
the formula at the end of section~2 in \cite{Palu08a}. By Theorem~4 of
\cite{Palu08a}, the map $L \mapsto X_L^{\mbox{\tiny Palu}}$ satisfies the
formula~\ref{eq:Palu-formula-cluster-variables}. Hence so does the map
$L \mapsto X_L$. The formula~\ref{eq:Palu-formula-F-polynomials} is
implicit in section~5.1 of \cite{Palu08a} and in particular in
formula~(2) of the proof of Lemma~16 of \cite{Palu08a}.
\end{proof}

\begin{corollary} \label{cor:decat-F-polynomials}
If the cluster tilting object $T(t)$ associated with
a vertex $t$ of $\mathbb{T}_n$ is reachable from $T(t_0)$, we have
\[
X_{T_i(t)}= X_i(t) \mbox{ and } F_{T_i(t)} = F_i(t)
\]
for all $1\leq i\leq n$.
\end{corollary}

\begin{proof}
Indeed, the first formula follows by induction from Theorem~\ref{thm:Palu-formula} and
the description of the exchange triangles given in Lemma~\ref{lemma:exchange-triangles}.
The second formula similarly follows from Theorem~\ref{thm:Palu-formula} once we
show that the integers $d_i$ and $d_i'$ coincide with the corresponding exponents
in the tropical $Y$-variable. This results from the categorical
interpretation of the tropical $Y$-variables in
Corollary~\ref{cor:decat-tropical-Y-variables} and
the following lemma.
\end{proof}

\begin{lemma} Let $L$ be a rigid object of $\cc$ and let $1\leq k\leq n$.
Let
\[
\xymatrix{T_k^* \ar[r] & E \ar[r] & T_k \ar[r] & \Sigma T_k^*} \mbox{ and }
\xymatrix{T_k \ar[r] & E' \ar[r] & T_k^* \ar[r] & \Sigma T_k}
\]
be `the' exchange triangles. Let $m$ be the multiplicity of $[T_k]$ in $\ind^{op}_T(L)$.
Then we have
\begin{equation} \label{eq:kernel-dimensions}
m= \left\{ \begin{array}{ll} \dim\ker(\cc(L, \Sigma T_k^*) \to \cc(L, \Sigma E)) & \mbox{if } m\geq 0; \\
\dim\ker(\cc(L, \Sigma T_k) \to \cc(L, \Sigma E')) & \mbox{if } m\leq 0.
\end{array} \right.
\end{equation}
Moreover, at least one of the two integers on the right hand side vanishes.
\end{lemma}

\begin{proof} A triangle is {\em contractible} if it is a direct sum
of triangles one of whose terms is zero. A triangle is {\em minimal}
if it does not contain a contractible triangle as a direct factor.
Every triangle is the sum of a contractible and a minimal triangle.
In particular, we can choose the triangle
\[
L \to T^0_L \to T^1_L \to \Sigma L \ko
\]
of Lemma~\ref{lemma:Keller-Reiten},
where $T^0_L$ and $T^1_L$ lie in $\add(T)$, to be minimal. Let
$m_0$ and $m_1$ be the multiplicities of $T_k$ in these two
objects. We will show that $m_0$ and $m_1$ agree with the
two dimensions on the right hand side of equation~\ref{eq:kernel-dimensions}.
By definition, we have $m=m_0-m_1$.
Now by Proposition~2.1 of \cite{DehyKeller08}, the indecomposable
$T_k$ cannot occur in both $T^0_L$ and $T^1_L$ and so $m_0=0$ or $m_1=0$.
Clearly, this will imply both assertions. It remains to prove that
$m_0$ and $m_1$ equal the two dimensions. Let us first reinterpret $m_0$ and
$m_1$. Let $B$ be the endomorphism
algebra of $T$ and $S_k$ the simple top of the indecomposable projective
$B$-module $\cc(T,T_k)$. Then we see from Lemma~\ref{lemma:functor-to-modules}
combined with the $2$-Calabi-Yau property and the rigidity of $T$
that the multiplicity $m_0$ of $T_k$ in $T^0_L$ is the multiplicity of $S_k$
in the socle of the $B$-module $\cc(T,\Sigma^2 L)$ and the multiplicity $m_1$
of $T_k$ in $T^1_L$ is the multiplicity of $S_k$ in the top of
the $B$-module $\cc(T,\Sigma L)$. We have to show that these multiplicities
agree with the respective dimensions on the right hand side of
equation~\ref{eq:kernel-dimensions}. We first consider the kernel
of
\[
\cc(\Sigma L, \Sigma^2 T_k) \to \cc(\Sigma L, \Sigma^2 E').
\]
By Lemma~\ref{lemma:functor-to-modules}, it is isomorphic to the kernel
of
\[
\Hom_B(\cc(T,\Sigma L), \cc(T, \Sigma^2 T_k)) \to \Hom_B(\cc(T,\Sigma L), \cc(T, \Sigma^2 E'))
\]
and thus to the value of $\Hom_B(\cc(T,\Sigma L),?)$ on the kernel of
\[
\cc(T,\Sigma^2 T_k) \to \cc(T,\Sigma^2 E').
\]
Because of the triangle
\[
\Sigma E' \to \Sigma T^*_k \to \Sigma^2 T_k \to \Sigma^2 E'
\]
and the fact that $\Ext^1(T_k, T_k^*)$ is one-dimensional, we have an exact
sequence
\[
0 \to S_k \to \cc(T,\Sigma^2 T_k) \to \cc(T, \Sigma^2 E')
\]
and so the kernel of $\cc(L, \Sigma T_k) \to \cc(L, \Sigma E')$ is isomorphic
to the space
\[
\Hom_B(\cc(T, \Sigma L), S_k)\ko
\]
whose dimension clearly equals
the multiplicity of $S_k$ in the head of $\cc(T, \Sigma L)$. This is
what we wanted to show. Now we consider the kernel of
\[
\cc(L, \Sigma T_k^*) \to \cc(L, \Sigma E).
\]
By the $2$-Calabi-Yau property, it is isomorphic to the dual of the cokernel of
\[
\cc(E, \Sigma L) \to \cc(T_k^*, \Sigma L).
\]
By the triangle
\[
\Sigma^{-1} E \to \Sigma^{-1} T_k \to T_k^* \to E\ko
\]
this cokernel is isomorphic to the kernel of
\[
\cc(T_k, \Sigma^2 L) \to \cc(E, \Sigma^2 L).
\]
By Lemma~\ref{lemma:functor-to-modules}, this kernel is isomorphic
to the kernel of
\[
\Hom_B(\cc(T, T_k), \cc(T, \Sigma^2 L)) \to \Hom_B(\cc(T,E),\cc(T,\Sigma^2 L)).
\]
Because of the triangle
\[
E \to T_k \to \Sigma T_k^* \to \Sigma E \ko
\]
the rigidity of $T$ and the fact that $\Ext^1(T_k, T_k^*)$
is one-dimensional, we have an exact sequence
\[
\cc(T,E) \to \cc(T, T_k) \to S_k \to 0.
\]
So the above kernel is isomorphic to the space
\[
\Hom_B(S_k, \cc(T, \Sigma^2 L)) \ko
\]
whose dimension clearly equals the multiplicity of $S_k$
in the socle of $\cc(T, \Sigma^2 L)$. This is what we had to
show.
\end{proof}

\begin{example} \label{ex:F-polynomials} We continue
example~\ref{ex:trop-y-var}. We keep the initial cluster tilting
object $T=P_1\oplus P_2 \oplus P_3$ and wish to compute the
$F$-polynomials associated with the direct factors of
$T'=\mu_\boxten(T)= I_1 \oplus I_3 \oplus I_2$. For this,
we first compute
\[
\Ext^1_{\cc_A}(T, T'_i) = \Ext^1_{\cc_A}(T, \tau^{-1} T_i) =
\Hom_{\cc_A}(T, \Sigma \tau^{-1} T_i) = \Hom_{\cc_A}(T, T_i).
\]
By lemma~\ref{lemma:functor-to-modules}, the space $\Hom_{\cc_A}(T,T_i)$,
considered as a module over $\End_{\cc_A}(T) = A$,
is isomorphic to $P_i$. Hence the $F$-polynomials we are looking
for are the generating functions of the Euler characteristics
of the quiver Grassmannians of the modules $P_i$, or equivalently
of the associated representations of $Q^{op} : \xymatrix{2 \ar[r] & 1 & 3 \ar[l]}$.
Now $P_1$ is simple and so has exactly two submodules,
namely $0$ and $P_1$, which leads to
\[
F_{T'_1} = 1 + y_1.
\]
The module $P_2$ has exactly three submodules, namely $0$, $P_1$ and $P_2$,
which leads to
\[
F_{T'_3} = 1 + y_1 + y_1 y_2.
\]
Similarly, we find
\[
F_{T'_3} = 1+y_1 + y_1 y_3.
\]
A simple computation shows that these are indeed the $F$-polynomials
associated with the vertex $\mu_\boxten(t_0)$ of $\mathbb{T}_3$.
Thanks to Fomin-Zelevinsky's formula~\ref{eq:non-trop-y-var} expressing
the non tropical $Y$-variables at a vertex $t$ of $\mathbb{T}_n$
in terms of the tropical $Y$-variables, the $F$-polynomials and the quiver
at $t$, we obtain a categorical expression for the (non tropical) $Y$-variables
associated with $\mu_\boxten(t_0)$ by combining the results of
this example with those of examples~\ref{ex:alternating-A3-mutation} 
and \ref{ex:trop-y-var}.
\end{example}

\subsection{Consequence for the conjecture} \label{ss:reduction-to-no-loops-and-T-periodic}
Let $\Delta$ and $\Delta'$
be simply laced Dynkin diagrams with Coxeter numbers $h$ and $h'$.
Let $Q$ and $Q'$ be quivers with
underlying graphs $\Delta$ and $\Delta'$.
According to Lemma~\ref{lemma:reduction-to-y-boxten},
in order to show the periodicity conjecture for $(\Delta,\Delta')$,
it suffices to show that the restricted $Y$-pattern $\mathbf{y}_\boxten$
associated with $Q\boxten Q'$ and the sequence of mutations
\begin{equation} \label{eq:boxten-mutation-sequence}
\mu_\boxten=\mu_{+,-} \mu_{-,-} \mu_{+,+} \mu_{-,+}.
\end{equation}
is periodic of period dividing $h+h'$.
By section~\ref{ss:periodicity-from-that-of-tropical-Y-variables},
for this it suffices to show that, for each vertex $v$ of
$Q\boxten Q'$, the sequences $\eta_v(t_{pN})$ and $F_v(t_{pN})$
are periodic in $p$ of period dividing $h+h'$, where
the vertices $t_i$ of $\mathbb{T}_N$ are those visited
when going through an integer power of $\mu_\boxten$.

Now let $A$ be the algebra $kQ\ten kQ'$ and $\cc_A$ its associated
cluster category with canonical cluster tilting object $T$ as
recalled in section~\ref{ss:2-CY-realizations-of-triangle-products}.
By Corollary~\ref{cor:cluster-category-Hom-finite}, the endoquiver of $T$ is isomorphic to $Q\boxten Q'$
so that the category $\cc_A$ yields a $2$-Calabi-Yau realization
of $Q\boxten Q'$. In particular, it makes sense to consider
the sequence of cluster tilting objects
\[
\xymatrix{T=T(t_0) \ar@{-}[r] & T(t_1) \ar@{-}[r] & \ldots}
\]
associated with the vertices $t_i$.
If we can show that the endoquivers
of all the objects $T(t_i)$ do not have loops nor $2$-cycles,
it will follow from Corollaries~\ref{cor:decat-tropical-Y-variables}
and \ref{cor:decat-F-polynomials} that, for all vertices $v$
of $Q\boxten Q'$ and all $i$, we have
\[
\eta_v(t_i)=\prod y_j^{\ind^{op}_{T(t_i)}(T_v)} \quad\mbox{ and }\quad
F_v(t_i)= F_{T_v(t_i)}.
\]
To conclude that the sequences $\eta_v(t_{pN})$ and $F_v(t_{pN})$ are
periodic of period dividing $h+h'$, it will then suffice to show that the
sequence $T(t_i)$ is periodic of period dividing $(h+h')N$.
In section~\ref{s:mutations-of-products} below, we will show
that indeed, the endoquivers of
the $T(t_i)$ do not have loops or $2$-cycles and we will
describe the objects $T(t_{pN})=\mu_{\boxten}^p(T)$ using
the Zamolodchikov transformation of $\cc_A$. In
section~\ref{s:categorical-periodicity} below,
we will show that the sequence of the $T(t_i)$ is
periodic of period dividing $N(h+h')$ or, in other
words, that $\mu_\boxten^{h+h'}(T)\cong T$.

\section{Mutations of products}
\label{s:mutations-of-products}

\subsection{Constrained quivers with potential} \label{ss:constrained-quivers}
We want to study the effect of the mutations $\mu_\boxtimes$ of
section~\ref{ss:reformulation} on the cluster tilting object
of Corollary~\ref{cor:cluster-category-Hom-finite}. We will use the
description  of this category by a quiver with potential constructed
in Proposition~\ref{prop:Calab-Yau-completion-via-quiver-with-potential}.
For this we introduce a class of quivers with potential containing
the ones from Proposition~\ref{prop:Calab-Yau-completion-via-quiver-with-potential}.

Let $Q$ and $Q'$ be finite quivers without oriented cycles. To simplify
the notations, let us suppose that between any two vertices of $Q$ and $Q'$,
there is at most one arrow. Let
$Q_0$ and $Q'_0$ denote their sets of vertices. Let $R$ be a quiver
whose vertex set is the product $Q_0\times Q'_0$. An arrow
$\alpha: (i,i') \to (j,j')$ of $R$ is {\em horizontal} (respectively,
{\em vertical}) if $i'=j'$ (respectively, if $i=j$). It is {\em diagonal}
if it is neither horizontal nor vertical. The {\em non diagonal subquiver
of $R$} is the subquiver formed by all vertices and by all the non diagonal
arrows of $R$. The quiver $R$ is {\em $(Q,Q')$-constrained} if
\begin{itemize} \item[a)] its
non diagonal subquiver has the same underlying graph as $Q\ten Q'$
(as defined in section~\ref{ss:products-of-quivers}) and
\item[b)] for any pair of arrows $i\to j$ of $Q$ and $i'\to j'$ of $Q'$, the
full subquiver of $R$ with vertex set $\{i,i'\}\times \{j,j'\}$ is isomorphic
to
\begin{equation} \label{eq:standard-squares}
\xymatrix{ \circ \ar[r]^{\alpha'} & \circ \ar[ld]^\rho \\
\circ \ar[u]^\beta \ar[r]_\alpha & \circ \ar[u]_{\beta'}
}
\quad\mbox{ or }\quad
\xymatrix{ \circ \ar[r]^{\alpha'} & \circ \ar[d]^{\beta'} \\
\circ \ar[u]^\beta & \circ \ar[l]^{\alpha}
}
\end{equation}
in such a way that $\alpha$ and $\alpha'$ correspond to horizontal
arrows, $\beta$ and $\beta'$ to vertical arrows and $\rho$ to a
diagonal arrow.
\end{itemize}
We sometimes call the subquivers appearing in b) the {\em squares} of $R$.
For example the quivers $Q\square Q'$ and $Q\boxtimes Q'$ are
$(Q,Q')$-constrained and have the minimal, respectively maximal,
number of diagonal arrows. Notice that a $(Q,Q')$-constrained
quiver does not have loops nor $2$-cycles.

A quiver with potential $(R,W)$ is {\em $(Q,Q')$-constrained} if
$R$ is $(Q,Q')$-constrained and the potential $W$ is the sum
of non zero scalar multiples of all the cycles $\alpha'\beta\rho$ and $\alpha\rho\beta'$
appearing in a square as in the left diagram in (\ref{eq:standard-squares}) as well as
all the cycles $\alpha \beta' \alpha' \beta$ appearing in a square as
in the right diagram in (\ref{eq:standard-squares}). Recall that
changing the starting point in a cycle does not change
the superpotential.
For example, the quivers
with potential obtained from Proposition~\ref{prop:Calab-Yau-completion-via-quiver-with-potential}
are $(Q,Q')$-constrained.

Let $R$ be a $(Q,Q')$-constrained quiver. For a vertex $(i,i')$
of $R$, the {\em horizontal slice through $(i,i')$} is the full subquiver
$\hrz(R,i')$ formed by the vertices $(j,i')$, $j\in Q_0$, of $R$;
the {\em vertical slice $\vrt(R,i)$ through $(i,i')$} is defined
analogously. A vertex $(i,i')$ of $R$ is
a {\em source-sink} if it is a source in its horizontal slice
and a sink in its vertical slice and is not the source or
the target of any diagonal arrow. Analogously,
one defines {\em sink-sources}, \ldots\ .
Notice that two source-sinks are never linked by an arrow and
that if $(i,i')$ is a source-sink, each arrow with source
(respectively target) $(i,i')$ lies in the horizontal (respectively
the vertical) slice passing through $(i,i')$.

\begin{lemma} \label{lemma:mutation-Q-Qprime-constrained}
Let $R$ be a $(Q,Q')$-constrained quiver and
$(i,i')$ a source-sink of $R$.
\begin{itemize}
\item[a)] The mutated quiver $\mu_{(i,i')}(R)$
is still $(Q,Q')$-constrained. The horizontal slice of $\mu_{(i,i')}(R)$
passing through $(i,i')$ is the mutation $\mu_{(i,i')}(\hrz(R,i'))$ and
the vertical slice the mutation $\mu_{(i,i')}(\vrt(R,i))$.
\item[b)] If $(R,W)$ is a $(Q,Q')$-constrained quiver with
potential, then $\mu_{(i,i')}(R,W)$ is still $(Q,Q')$-constrained.
In particular, it does not have loops or $2$-cycles.
\end{itemize}
\end{lemma}

\begin{proof} a) The mutation at $(i,i')$ reverses the arrows
passing through $(i,i')$ and does not create any diagonal
arrows incident with $(i,i')$. Thus, $(i,i')$ is a sink-source
in the mutated quiver and the second assertion is clear.
If $(i,i')$ belongs to a square (where it has to be
the vertex $\bt$)
\begin{equation}
\xymatrix{ \bt \ar[r]^{\alpha'} & \circ \ar[ld]^\rho \\
\circ \ar[u]^\beta \ar[r]_\alpha & \circ \ar[u]_{\beta'}
}
\quad\mbox{ or }\quad
\xymatrix{ \bt \ar[r]^{\alpha'} & \circ \ar[d]^{\beta'} \\
\circ \ar[u]^\beta & \circ \ar[l]^{\alpha}
}
\end{equation}
as in \ref{eq:standard-squares} then it transforms it into the
other type of square; it leaves all squares not containing $(i,i')$
unchanged. This shows the first assertion.

b) In the non reduced mutated quiver $\tilde{\mu}_{(i,i')}(R,W)$
the squares of the first type containing $(i,i')$ are modified as follows
\begin{equation}
\xymatrix@R=1.5cm@C=1.5cm{ \bt \ar[d]_{\beta^*} & \circ \ar[l]_-{\alpha'^*} \ar@/_/[dl]_\rho \\
\circ \ar[r]_\alpha \ar@/_/[ur]_r & \circ \ar[u]_{\beta'} }
\end{equation}
where $r=[\alpha'\beta]$. Their contribution to the potential
is changed from $c_1 \alpha' \beta \rho + c_2 \beta' \alpha\rho$,
where $c_1$ and $c_2$ are non zero scalars, to
\[
c_1 r\rho + r \beta^* \alpha'^* + c_2 \beta' \alpha \rho
=(c_1 r + c_2 \beta'\alpha)(\rho+c_1^{-1} \beta^* \alpha'^*) - c_2 c_1^{-1} \beta' \alpha\beta^* \alpha'^*.
\]
Notice that $r$ and $\rho$ do not appear in any other terms of
the potential. Therefore, after reduction, the square becomes
\[
\xymatrix{ \bt \ar[r]^{\alpha'} & \circ \ar[d]^{\beta'} \\
\circ \ar[u]^\beta & \circ \ar[l]^{\alpha}
}
\]
and its contribution to the potential becomes $-c_2 c_1^{-1} \beta' \alpha\beta^* \alpha'^*$.
For the squares of the second type containing $(i,i')$, mutation changes
them into
\[
\xymatrix{\bt \ar[d]_{\beta^*} & \circ \ar[l]_{\alpha'^*} \ar[d]^{\beta'} \\
\circ \ar[ru]_r & \circ \ar[l]^\alpha }
\]
where $r=[\alpha'\beta]$,
and their contribution to the potential changes from $c \alpha' \beta' \alpha \beta'$,
where $c$ is a non zero scalar, to
\[
c r \alpha\beta' + r \beta^* \alpha'^* .
\]
So we see that the effect of mutation is to exchange the two types
of squares. The assertion follows.
\end{proof}

\subsection{The Zamolodchikov transformation} \label{ss:Zamolodchikov-transformation}
Let $Q$ be a finite quiver without oriented cycles. We order the vertices of $Q$
such that $i\leq j$ iff there is a path (of length $\geq 0$) leading
from $i$ to $j$. A {\em source sequence of $Q$} is an
enumeration $i_1$, $i_2$, \ldots, $i_n$ of the vertices of $Q$
which is non-decreasing with respect to $\leq$,
i.~e. if $1\leq s\leq t\leq n$, then we have $i_s\leq i_t$
for the order on the vertices that we have just defined.

Let $kQ$ be the path algebra of $Q$ and $\mod kQ$ the category
of finite-dimensional right $kQ$-modules. Let $\cd^b(kQ)$ its
bounded derived category and denote by $\tau=\Sigma^{-1}S$
its Auslander-Reiten functor. As explained in \cite{Happel87},
for each vertex $i$ of $Q$, we have a canonical Auslander-Reiten triangle
\begin{equation} \label{eq:AR-triangles}
\xymatrix{P_i \ar[r] &
(\bigoplus_{i\to j} P_j) \oplus (\bigoplus_{j\to i} \tau^{-1} P_j) \ar[r] &
\tau^{-1} P_i \ar[r] &\Sigma P_i}
\end{equation}

Now let $\cc_Q$ be the cluster category of $Q$, \cf section~\ref{ss:2-CY-realizations}.
It is a $2$-Calabi-Yau category and the image $T$ of the free
module $kQ$ is a cluster tilting object in $\cc_Q$.

\begin{lemma} \label{lemma:tau-for-kQ} Let $i_1, \ldots, i_n$
be a source sequence of $Q$. For each $1\leq j\leq n$, the mutated
cluster tilting object
\[
\mu_{i_j} \mu_{i_{j-1}} \ldots \mu_{i_1}(T)
\]
is the direct sum of the objects $\tau^{-1} P_{i_r}$, $1\leq r\leq j$, and
the objects $P_{i_s}$, $j<s\leq n$. In particular, for $j=n$,the
mutated object cluster tilting object
\[
\mu_{i_n} \mu_{i_{n-1}} \ldots \mu_{i_1}(T)
\]
is isomorphic to $\tau^{-1} T$.
\end{lemma}

\begin{proof} This is an easy induction: For $j=0$, there is
nothing to show. For $j>0$, we compare
the above triangle~\ref{eq:AR-triangles} with the
exchange triangles~\ref{eq:exchange-triangles} to see
that $\mu_j$ replaces the summand $P_j$ by $\tau^{-1}P_j$.
\end{proof}

Now let $Q'$ be another finite quiver without oriented cycles.
For simplicity, we assume that in both, $Q$ and $Q'$, there is at
most one arrow between any two given vertices. Let $A$ be
the finite-dimensional algebra $kQ\ten_k kQ'$. We denote
by $\mod(A)$ its category of $k$-finite-dimensional right $A$-modules,
by $\cd^b(A)$ its bounded derived category and by $\cc_A$ the
associated cluster category as recalled in section~\ref{ss:2-CY-realizations-of-triangle-products}.
We define $\tau\ten\id$ to be the autoequivalence of $\cd^b(A)$
given by left derived tensor product with the bimodule
\[
(\Sigma^{-1} D(kQ))\ten_k kQ'
\]
over $A=kQ\ten_k kQ'$ and we define $\tau^{-1}\ten \id$ to be
its quasi-inverse. By Proposition~4.1 of \cite{Keller11b},
this functor induces an autoequivalence $\tilde{\Za}$ of
the perfect derived category $\per(\Pi_3(A))$ which preserves
the finite-dimensional derived category.
The {\em Zamolodchikov transformation}
is the induced autoequivalence $\Za:\cc_A \to \cc_A$ of the cluster
category. By abuse of notation, we still write $\Za=\tau^{-1}\ten \id$.
If the underlying graph of the quiver $Q'$ is the Dynkin diagram $A_1$,
the Zamolodchikov transformation coincides with the inverse Auslander-Reiten
translation $\tau^{-1}$.

\subsection{The Zamolodchikov transformation as a composition of mutations}
\label{ss:Zamolodchikov-composition-of-mutations} As in
section~\ref{ss:Zamolodchikov-transformation}, let $Q$ and
$Q'$ be finite quivers without oriented cycles which, for
simplicity of notation, are assumed to have at most one arrow
between any two vertices.
Let us order the vertices of $Q\boxtimes Q'$ such that
$(i,i')\leq (j,j')$ iff $i\leq j$ and $i'\geq j'$ (sic!).
Let $v_1, \ldots, v_{N}$ be a non decreasing enumeration
of the vertices of $Q\boxtimes Q'$. For each $1\leq j\leq N$,
we put
\[
(R(j), W(j)) = \mu_{v_j} \mu_{v_{j-1}} \ldots \mu_{v_1}(Q\boxtimes Q', W)
\]
where $W$ is the potential constructed in
Proposition~\ref{prop:Calab-Yau-completion-via-quiver-with-potential}.

\begin{lemma} \label{lemma:composed-mutation-of-product}
For each $0\leq j\leq N$, we have
\begin{itemize}
\item[a)] $v_{j+1}$ is a source-sink of $R(j)$;
\item[b)] $(R(j),W(j))$ is $(Q,Q')$-constrained, and so $R(j)$
does not have loops nor $2$-cycles;
\item[c)] for each vertex $i'$ of $Q'$, the
quiver $\hrz(R(j),i')$ is isomorphic to $\mu_{i_s} \mu_{i_{s-1}}\ldots \mu_{i_1}(Q)$,
where $(i_1, i')$, \ldots, $(i_s, i')$ is the subsequence of the
vertices of the form $(x,i')$, $x\in Q_0$, among the sequence
$v_1$, \ldots, $v_j$;
\item[d)] for each $i\in Q_0$, the quiver $\vrt(R(j),i)$ is isomorphic
to $\mu_{i_s} \mu_{i_{s-1}}\ldots \mu_{i_1}(Q')$,
where $(i, i_1)$, \ldots, $(i, i_s)$ is the subsequence of the
vertices of the form $(i,y)$, $y\in Q'_0$, among the sequence
$v_1$, \ldots, $v_j$;
\item[e)] The object
\[
\mu_{v_j} \mu_{v_{j-1}} \ldots \mu_{v_1}(T)
\]
is the direct sum of the objects $(\tau^{-1} P_i)\ten P_{i'}$
where $(i,i')$ is among the $v_s$, $1\leq s\leq j$, and of
the $P_i\ten P_{i'}$, where $(i,i')$ is not among the $v_s$, $1\leq s\leq j$.
\end{itemize}
\end{lemma}

\begin{proof} We prove a)-d) simultaneously by induction on $j$.
For $j=0$, all the assertions
are clear. Assume the statement hold up to $j-1$. Then by a)$_{j-1}$, the
vertex $v_j$ is a source-sink of $R(j-1)$ and $(R(j-1), W(j-1))$ is $(Q,Q')$-constrained
by b)$_{j-1}$. So $(R(j), W(j))$ is still $(Q,Q')$-constrained by
Lemma~\ref{lemma:mutation-Q-Qprime-constrained}. So we have proved
b)$_j$. To prove c)$_j$, let $i'$ be a vertex of $Q'$. If $i'$
is not the second component of $v_j$, the sequence $i_1$, \ldots, $i_s$
remains unchanged and so does the subquiver $\hrz(R(j),i')$ by
Lemma~\ref{lemma:mutation-Q-Qprime-constrained}. If $i'$ is the second
component of $v_j$, then the sequence $i_1$, \ldots, $i_s$ is extended
by adding $i_{s+1}\geq i_s$ and the claim of c)$_j$ still follows from
Lemma~\ref{lemma:mutation-Q-Qprime-constrained}. Similarly, one proves d).
Finally, we have to show a)$_j$. Indeed, the vertex $v_{j+1}$ is of
the form $(i,i')$, and the first components $i_1$, \ldots, $i_s$, $i_{s+1}$
of the $v_1$, \ldots, $v_{j+1}$
which are of the form $(x,i')$, $x\in Q_0$,
form a source sequence of $\hrz(R(0),i')$. So $i$ is a source of
$\hrz(R(j), i')=\mu_{i_s} \ldots \mu_{i_1}(\hrz(R(0),i'))$. Similarly,
one sees that $(i,i')$ is a sink of $\vrt(R(j), i)$.

Now let us prove e) by induction on $j$. For $j=0$, there is nothing
to prove. Assume the assertion holds up to $j-1$. By a)$_{j-1}$, the
vertex $v_j=(i,i')$ is a source-sink of $R(j-1)$. So in particular,
the vertex $v_j$ is a source of $\hrz(R(j-1),i')$.
By the induction hypothesis, the direct summands of
\[
\mu_{v_{j-1}} \ldots \mu_{v_1}(T)
\]
associated with the vertices of $\hrz(R(j-1),i')$ are the
\[
\tau^{-1}P_{i_u} \ten P_{i'}
\]
for $1\leq u<s$ and the $P_j\ten P_{i'}$ for $j$ not among the
$i_u$, $1\leq u<s$. Now
in $\cd^b(kQ)$, we have the sequence
\begin{equation}
\xymatrix{P_i \ar[r] &
(\bigoplus_{i\to j} P_j) \oplus (\bigoplus_{j\to i} \tau^{-1} P_j) \ar[r] &
\tau^{-1} P_i \ar[r] &\Sigma P_i}
\end{equation}
recalled in \ref{eq:AR-triangles}. By tensoring the sequence with
$P_{i'}$ over $k$ and taking the image in $\cc_A$, we get the
exchange triangle, which shows that the mutation at $v_j=(i,i')$
replaces the summand $P_i\ten P_{i'}$ with $(\tau^{-1}P_i)\ten P_{i'}$.
\end{proof}

\begin{corollary} Let $v_1$, \ldots, $v_N$ be a non decreasing enumeration
of the vertices of $Q\boxtimes Q'$ (for the order where $(i,i')\leq (j,j')$
iff there is a path from $i$ to $j$ and a path from $j'$ to $i'$!).
Let $\mu_v$ be the composed mutation
$
\mu_{v_N} \mu_{v_{N-1}} \ldots \mu_{v_1}.
$
For each vertex $(i,i')$ of $Q\boxtimes Q'$,
let $T_{i,i'}$ be the indecomposable summand $P_i\ten P_{i'}$ of
the canonical cluster tilting object $T$ of the cluster category $\cc_A$.
\begin{itemize}
\item[a)] For any vertex $t$ of $\mathbb{T}_n$ visited when performing an
integer power of $\mu_v$, the endoquiver of $T(t)$
does not have loops nor $2$-cycles.
\item[b)] For each $p\geq 0$ and all $(i,i')$, we have an isomorphism in the cluster
category
\[
\Za^p(T_{i,i'}) \cong (\mu_\boxten^p(T))_{i,i'}.
\]
\end{itemize}
\end{corollary}

\begin{proof} a)
This follows from part b) of Lemma~\ref{lemma:composed-mutation-of-product}
and the fact that the endoquiver of a mutation
of a cluster tilting object is the quiver appearing in the mutated
quiver with potential as recalled at the end of
section~\ref{ss:reminder-quivers-with-potential}.

b) We proceed by induction on $p$. For $p=0$, there is nothing
to show. For $p=1$, the assertion follows from part e) of Lemma~\ref{lemma:composed-mutation-of-product}.
Now suppose that $p\geq 1$ and that we have an isomorphism
\[
\Za^p(T_{i,i'}) \cong (\mu_{v}^p(T))_{i,i'}.
\]
Since $\Za$ is an autoequivalence of $\cc_A$, it follows that
\[
\Za^{p+1}(T_{i,i'}) \cong \Za ((\mu_{v}^p(T))_{i,i'}) \cong (\mu_v^p(\Za(T)))_{i,i'}.
\]
Thanks to the case where $p=1$, we conclude that
\[
\Za^{p+1}(T_{i,i'})\cong (\mu_v^p(\mu_v(T)))_{i,i'} \cong (\mu_v^{p+1}(T))_{i,i'}.
\]
\end{proof}

\section{Categorical periodicity}
\label{s:categorical-periodicity}

\subsection{Categorification of the Coxeter element} Let $\Delta$
be a simply laced Dynkin diagram with vertices $1$, \ldots, $n$.
Let $\alpha_i$ be the simple root associated
with the vertex $i$, $1\leq i\leq n$. Let $h$ be the Coxeter
number of $\Delta$.
Let $Q$ be a quiver with underlying
graph $\Delta$. Let $k$ be a field and $\cd^b(kQ)$ the bounded
derived category of finite-dimensional right modules over
the path algebra $kQ$. Then $\cd^b(kQ)$ is a $\Hom$-finite
triangulated category. We write $\Sigma$ for its suspension
functor, $S$ for its Serre functor and $\tau=\Sigma^{-1} S$
for its Auslander-Reiten translation. For each $1\leq i\leq n$, let
$S_i$ be the simple module associated with the vertex $i$.

\begin{theorem}[Gabriel, Happel] \label{thm:Gabriel-Happel}
\begin{itemize}
\item[a)]
There is a canonical isomorphism from the Gro\-then\-dieck group
$K_0(\cd^b(kQ))$ of the triangulated
category $\cd^b(kQ)$ to the root lattice of $\Delta$ which
takes the class of the simple module $[S_i]$ to the
simple root $\alpha_i$.
\item[b)] Under this isomorphism, the (positive and negative)
roots correspond precisely
to the classes in $K_0(\cd^b(kQ))$ of the indecomposable objects of $\cd^b(kQ)$.
\item[c)] The automorphism of the Grothendieck group of the derived category
induced by $\tau^{-1}$ corresponds to the action of a Coxeter element $c$ on the
root lattice. The identity $c^h=\id$ lifts to an isomorphism
\[
\tau^{-h} \iso \Sigma^2
\]
of $k$-linear functors $\cd^b(kQ)\to \cd^b(kQ)$.
\end{itemize}
\end{theorem}

\begin{proof} Parts a) and b) follow immediately from Gabriel's
theorem \cite{Gabriel72} and from Happel's description of the
derived category $\cd^b(kQ)$ in \cite{Happel87}. The isomorphism
of functors in part c) follows
from Happel's description of the category $\cd^b(kQ)$ as the
mesh category of the translation quiver $\Z Q$ and from
Gabriel's description of the Serre functor (alias Nakayama functor)
in Proposition~6.5 of \cite{Gabriel80}. A detailed proof of a more
precise statement is given by Miyachi-Yekutieli in Theorem~4.1 of \cite{MiyachiYekutieli01}.
\end{proof}

\subsection{On the order of the Zamolodchikov transformation}
\label{ss:order-of-the-Zamolodchikov-transformation}
As in section~\ref{ss:Zamolodchikov-transformation}, let $Q$ and $Q'$ be
finite quivers without oriented cycles such that between any two vertices,
there is at most one arrow. Let $\Za=\tau^{-1}\ten \id$ be the
Zamolodchikov transformation of the cluster category $\cc_A$
associated with $A=kQ\ten_k kQ'$, \cf section~\ref{ss:Zamolodchikov-transformation}.

\begin{theorem} \label{thm:order-Zamolodchikov}
Suppose that the graphs underlying $Q$ and $Q'$ are
two simply laced Dynkin diagrams with Coxeter numbers $h$ and $h'$.
Then we have an isomorphism of functors from $\cc_A$ to itself
\begin{equation}
\Za^{h+h'} \iso \id.
\end{equation}
\end{theorem}

\begin{remark} Let $W$ be the Weyl group associated with $Q$
and $w_0$ its longest element. The element $w_0$ takes all
positive roots to negative roots. It is the product (in a
suitable order) of the reflections at all positive roots.
Let $w'_0$ be the longest element of the Weyl group associated
with $Q'$. One can refine the proof below to show that if
both $w_0$ and $w'_0$ act by multiplication with $-1$ on their respective
root lattices, we have an isomorphism
\begin{equation}
\Za^{(h+h')/2} \iso \id.
\end{equation}
\end{remark}

\begin{proof} Let $A_1= kQ$. The Auslander-Reiten translation
$\tau=\Sigma^{-1} S$ of $\cd^b(A_1)$ is given by tensoring with the
bimodule $\Sigma^{-1} DA_1$, where $DA_1$ is the $k$-dual of the bimodule $A_1$.
By Theorem~\ref{thm:Gabriel-Happel}, we have an isomorphism
$\tau^h(A_1)=\Sigma^{-2}(A_1)$ which is compatible with the
left actions of $A_1$ on the two sides. By the main
theorem of \cite{Keller93}, this implies that we have an isomorphism
in the derived category of bimodules $\cd^b(A_1\ten A_1^{op})$
\[
(\Sigma^{-1} DA_1)^{\ten_{A_1} h} \iso \Sigma^{-2} A_1,
\]
where we write $\ten_{A_1}h $ for the derived tensor power.
This yields an isomorphism in the derived category of $A$-bimodules
\[
(\Sigma^{-1} DA_1\ten A_2)^{\ten_{A} h} \iso \Sigma^{-2} A.
\]
Whence an isomorphism of functors $(\tau\ten\id)^h = \Sigma^{-2}$
from $\cc_A$ to itself.
Similarly, we have an isomorphism in the derived category of
$A_2$-bimodules
\[
(\Sigma^{-1}DA_2)^{\ten_{A_2} h'} \iso \Sigma^{-2} A_2\ko
\]
which yields an isomorphism of functors $(\id\ten\tau)^{h'}=\Sigma^{-2}$
in $\cc_A$.
Now we have
\[
(\Sigma^{-1} DA_1\ten A_2)\ten_A (A_1\ten \Sigma^{-1} DA_2) = \Sigma^{-2} (DA_1\ten DA_2) = \Sigma^{-2} DA.
\]
Since the category $\cc_A$ is $2$-Calabi-Yau, the Serre functor of $\cc_A$ is
isomorphic to $\Sigma^2$. Now the Serre functor is induced by the left derived
functor of tensoring with $DA$. So tensoring with $\Sigma^{-2} DA$ induces the
identity in $\cc_A$ and so the derived tensor product with
\[
(\Sigma^{-1} DA_1\ten A_2)\ten_A (A_1\ten \Sigma^{-1} DA_2)
\]
induces the identity in $\cc_A$. In other words, we have
\[
(\tau\ten\id)(\id\ten\tau) = \id
\]
as functors from $\cc_A$ to itself. Finally, we find the following chain
of isomorphisms of functors from $\cc_A$ to itself:
\[
(\tau\ten\id)^{h+h'} = (\tau\ten\id)^h (\tau\ten \id)^{h'} =
(\tau\ten\id)^h (\id\ten\tau)^{-h'} = \Sigma^{-2} \Sigma^{2} = \id.
\]
\end{proof}

\subsection{Conclusion}  \label{ss:conclusion}
Let $\Delta$ and $\Delta'$ be simply laced
Dynkin diagrams with Coxeter numbers $h$ and $h'$. Let $Q$ and $Q'$
be quivers with underlying graphs $\Delta$ and $\Delta'$. Let
$\cc_A$ be the cluster category associated with $A=kQ\ten kQ'$
and $T$ is canonical cluster tilting object.
As announced in section~\ref{ss:reduction-to-no-loops-and-T-periodic},
we have shown that in the sequence of cluster tilting
objects associated with the powers of $\mu_\boxten$, the endoquivers
do not have loops or $2$-cycles and that we have
\[
\mu_\boxtimes^{h+h'}(T)\cong T.
\]
This implies the conjecture, as explained in
section~\ref{ss:reduction-to-no-loops-and-T-periodic}.

\section{The non simply laced case}
\label{s:non-simply-laced-case}

In this section, we reduce the general case of the conjecture
to the one where the two Dynkin diagrams are simply laced.
We use the classical folding technique in the spirit
of section~2.4 of \cite{FominZelevinsky03b}. The material
in sections~\ref{ss:valued-quivers} and \ref{ss:valued-orbit-quivers-mutations}
is adapted from \cite{Dupont08a}.

\subsection{Valued quivers and skew-symmetrizable matrices}
\label{ss:valued-quivers}
A {\em valued quiver} is a quiver $Q$ endowed with a function $v: Q_1 \to \N^2$
such that
\begin{itemize}
\item[a)] there are no loops in $Q$,
\item[b)] there is at most one arrow between any two vertices of $Q$ and
\item[c)] there is a function $d: Q_0 \to \N$ such that $d(i)$ is strictly
positive for all vertices $i$ and, for each arrow $\alpha: i\to j$, we have
\[
v(\alpha)_1 d(i) = d(j)v(\alpha)_2 \ko
\]
where $v(\alpha)=(v(\alpha)_1, v(\alpha)_2)$.
\end{itemize}
For example, we have the valued quiver
\[
\vec{B}_2 : \xymatrix{1 \ar[r]^{(2,1)} & 2 } \ko
\]
where a possible function $d$ is given by $d(1)=1$, $d(2)=2$.
Let $Q$ be a valued quiver with vertex set $I$. We associate
an integer matrix $B=(b_{ij})_{i,j\in I}$ with it as follows
\[
b_{ij} = \left\{ \begin{array}{ll} 0 & \mbox{ if there is no arrow between $i$ and $j$;} \\
v(\alpha)_1 & \mbox{ if there is an arrow } \alpha: i \to j ; \\
-v(\alpha)_2 & \mbox{ if there is an arrow } \alpha: j \to i.
\end{array} \right.
\]
If $D$ is the diagonal $I\times I$-matrix with diagonal entries
$d_{ii}=d(i)$, $i\in I$, then the matrix $DB$ is antisymmetric.
The existence of such a matrix $D$ means that the matrix
$B$ is {\em antisymmetrizable}. It is easy to
check that in this way, we obtain a bijection between the
antisymmetrizable $I\times I$-matrices $B$ and the valued
quivers with vertex set $I$. Using this bijection, we define
the {\em mutation of valued quivers} using Fomin-Zelevinsky's
matrix mutation rule~\ref{eq:matrix-mutation}.

Let $(Q,v)$ be a valued quiver with vertex set $I=Q_0$. Its
{\em associated Cartan matrix} is the Cartan companion
\cite{FominZelevinsky03}
of the antisymmetrizable matrix $B$ associated with $Q$.
Explicitly, it is the the $I\times I$-matrix $C$
whose coefficient $c_{ij}$ vanishes if there are no
arrows between $i$ and $j$, equals $2$ if $i=j$,
equals $-v(\alpha)_1$ if there is an arrow $\alpha: i \to j$
and equals $-v(\alpha)_2$ if there is an arrow $\alpha: j\to i$. Thus, the Cartan
matrix associated with the above valued quiver $\vec{B}_2$
equals
\[
\left[ \begin{array}{cc} 2 & -2 \\ -1 & 2 \end{array} \right].
\]
The {\em valued graph} \cite{DlabRingel74a} underlying the valued quiver $(Q,v)$ is
by definition the set $I$ of vertices of $Q$ together with the
non negative integers $e_{ij}$, $i,j\in I$,  defined by
\[
e_{ij} = \left\{ \begin{array}{ll} -c_{ij} & i\neq j\, , \\ 0 & i=j. \end{array} \right.
\]
One checks easily that the pictorial representation (used in \cite{DlabRingel74a})
of the valued graph $(I, \mathbf{e})$ is obtained from that of $(Q,v)$
by replacing all arrows with unoriented edges. In the sequel, we will
identify Dynkin diagrams with the valued graphs corresponding to their
Cartan matrices.


\subsection{Valued orbit quivers and their mutations}
\label{ss:valued-orbit-quivers-mutations}
Let $\tilde{Q}$ be an (ordinary) quiver with vertex
set $\tilde{I}$ without loops or $2$-cycles. Let $G$ be a finite group of automorphisms
of $\tilde{Q}$. Let $\tilde{B}$ be the antisymmetric matrix associated
with $\tilde{Q}$. The {\em orbit quiver} $\tilde{Q}/G$ is the quiver
with vertex set $I=\tilde{I}/G$ and where there is an arrow
from a vertex $i$ to a vertex $j$ if there is an arrow $\tilde{i} \to \tilde{j}$
for some vertices $\tilde{i}$ in $i$ and $\tilde{j}$ in $j$.
In the following example, due to A.~Zelevinsky, the orbit quiver of
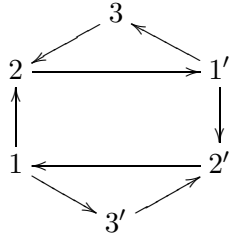
\begin{equation} \label{eq:non-admissible-action}
\xymatrix@R=0.25cm{
   & 3 \ar[dl] &  \\
2 \ar[rr] &  & 1' \ar[ul] \ar[dd] \\
          &  &                  \\
1 \ar[uu] \ar[dr] &  & 2' \ar[ll] \\
   & 3' \ar[ur] &
}
\end{equation}
under the action of $\Z/2\Z$ which exchanges opposite vertices is
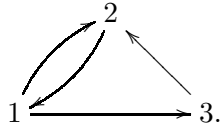
\[
\xymatrix{   & 2 \ar@/^/[dl] &   \\
1 \ar@/^/[ur] \ar[rr] & & 3. \ar[ul] }
\]
Notice the presence of a $2$-cycle in the orbit quiver. The action
of $G$ on $Q$ is {\em admissible} if the orbit quiver does not
have loops or $2$-cycles. In this case, we define the {\em valued orbit quiver}
\[
Q=\tilde{Q}/_v G
\]
to be the valued quiver with vertex set $I=\tilde{I}/G$ and whose
associated antisymmetrizable $I\times I$-matrix is given by
\[
b_{ij} = \sum_{\tilde{i} \in i} \tilde{b}_{\tilde{i}\; \tilde{j}} \ko
\]
where $\tilde{j}$ is any representative of $j$. We obtain
a function $d$ for $Q$ by sending an orbit $i$ to the cardinality
$d(i)$ of the stabilizer in $G$ of any vertex $\tilde{i}$ in $i$.
For example, the above quiver $\vec{B}_2$ is isomorphic to the
orbit quiver of
\[
\xymatrix{1 \ar[r] & 2 & 1' \ar[l]}
\]
under the action of $\Z/2\Z$ which fixes $2$ and exchanges $1$ and $1'$.

Now assume that the action of $G$ on $\tilde{Q}$ is admissible.
Let $k$ be a vertex of $Q$. Notice that between any two
vertices $\tilde{k}$ and $\tilde{k}'$ of the orbit $k$, there
are no arrows (because the orbit quiver has no loops).
Thus, the mutated quiver
\[
\prod_{\tilde{k}} \mu_{\tilde{k}} (\tilde{Q})
\]
is independent of the choice of the order of the vertices
$\tilde{k}$ in the orbit $k$. This quiver inherits a natural
action of $G$. However, this action need not be admissible
any more. For example, the following quiver
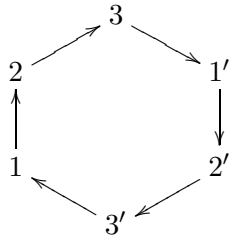
\[
\xymatrix@R=0.25cm{
           & 3 \ar[dr] &  \\
2  \ar[ur] &  & 1' \ar[dd]  \\
           &  &             \\
1 \ar[uu]  &  & 2' \ar[dl] \\
           & 3' \ar[ul] &
}
\]
has an admissible action by $\Z/2\Z$ which exchanges
opposite vertices but if we mutate at the
vertices of the orbit of $3$, we obtain the quiver
of \ref{eq:non-admissible-action} with its non admissible
action.

\begin{lemma} If the action of $G$ on the mutated
quiver
$
\prod_{\tilde{k}} \mu_{\tilde{k}} (\tilde{Q})
$
is admissible, there is a canonical isomorphism between
the valued orbit quiver
\[
(\prod_{\tilde{k}} \mu_{\tilde{k}} (\tilde{Q}))/_v G
\]
and the mutated valued quiver $\mu_k(Q/_v G)$.
\end{lemma}

\subsection{Valued $Y$-seeds} \label{ss:valued-y-seeds}
A {\em (valued) $Y$-seed} is a pair $(Q,Y)$ formed by a
finite valued quiver $Q$ with vertex set $I$ and a free generating
set $Y=\{Y_i \;|\; i\in I\}$ of the semifield $\Q_{sf}(y_i\;|\; i\in I)$
generated over $\Q$ by indeterminates $y_i$, $i\in I$.
If $(Q,Y)$ is a $Y$-seed and $k$ a vertex of $Q$,
the {\em mutated $Y$-seed} $\mu_k(Q,Y)$ is
$(Q',Y')$, where $Q'=\mu_k(Q)$ and, for $j\in I$, we have
\[
Y'_j = \left\{ \begin{array}{ll}  Y_k^{-1} & \mbox{if $j=k$;} \\
Y_i (1+Y_k^{-1})^{-b_{kj}} & \mbox{if $b_{kj}\geq 0$} \\
Y_i (1+Y_k)^{-b_{kj}} & \mbox{if $b_{kj}\leq 0$.}
\end{array} \right.
\]
One checks that $\mu_k^2(Q,Y)=(Q,Y)$. For example, the following
$Y$-seeds are related by a mutation at the vertex $1$:
\[
\xymatrix@R=1cm@C=1cm{ y_1 \ar[r]^{(2,1)} & y_2 \ar[dl]|-{(1,4)}\\
y_3 \ar[u]^{(2,1)} \ar[r]_{(2,1)} & y_4, \ar[u]_{(2,1)}
}
\quad\quad
\xymatrix@R=1cm@C=1cm{ 1/y_1 \ar[d]_{(1,2)} & y_2/(1+y_1^{-1})^2 \ar[l]_-{(1,2)}\\
y_3(1+y_1) \ar[r]_-{(2,1)} & y_4 \ar[u]_{(2,1)}
}\ko
\]
where we write the variable $Y_i$ in place of the vertex $i$.

For a valued quiver $Q$, the {\em initial $Y$-seed}, the {\em $Y$-pattern}
and the {\em restricted $Y$-pattern} associated with a sequence
of vertices of $Q$ are defined as in the simply laced case
in section~\ref{ss:Y-seeds}.

Now let $\tilde{Q}$ be a quiver endowed with an
admissible action of a finite group $G$. Let $Q=\tilde{Q}/_v G$
be the valued orbit quiver. Let $k$ be a vertex of $Q$ and
$\tilde{Q}'$ the mutated quiver
\[
\prod_{\tilde{k}\in k} \mu_{\tilde{k}} \tilde{Q}.
\]
Assume that the action of $G$ on $\tilde{Q}'$ is still admissible.
Let
\[
\pi: \Q_{sf}(y_{\tilde{i}}| \tilde{i}\in\tilde{I}) \to \Q_{sf}(y_i| i\in I)
\]
be the unique morphism of semifields such that $\pi(y_{\tilde{i}})=y_i$ for
all vertices $\tilde{i}$ of $\tilde{Q}$. Let $\tilde{Y}$ be the set of
the $y_{\tilde{i}}$, $\tilde{i}\in\tilde{I}$, and let $Y$ be the set of
the $y_i$, $i\in I$. Define $Y'$ by $\mu_k(Q,Y)=(Q', Y')$ and
$\tilde{Y}'$ such that $(\tilde{Q}', \tilde{Y}')$ is the result of applying
the mutations $\mu_{\tilde{k}}$, $\tilde{k}\in k$, to $(\tilde{Q}, \tilde{Y})$.
\begin{lemma} \label{lemma:mutation-versus-quotients}
We have $\pi(\tilde{Y}'_i)=Y'_i$ for all $i\in I$.
\end{lemma}

The proof is a straightforward computation which we omit.

\subsection{Products of valued quivers} Let $Q$ and $Q'$ be two
valued quivers. The {\em tensor product $Q\ten Q'$ } is defined as
the tensor product of the underlying graphs endowed with the
valuation $v$ such that
\[
v(\alpha,i')=v(\alpha) \mbox{ and } v(i,\alpha')=v(\alpha')
\]
for all arrows $\alpha$ of $Q$ and $\alpha'$ of $Q'$ and
for all vertices $i'$ of $Q'$ and $i$ of $Q$. The
{\em triangle product $Q\boxtimes Q'$} is obtained from
the tensor product by adding an arrow $(j,j') \to (i,i')$
of valuation
\[
(v(\alpha)_2 v(\alpha')_2, v(\alpha)_1 v(\alpha')_1)
\]
for each pair of arrows $\alpha$ of $Q$ and $\alpha'$ of $Q'$.
For example, the triangle product $\vec{B}_2 \boxtimes \vec{B}_2$
is given by
\[
\xymatrix{(1,2) \ar[r]^{(2,1)} & (2,2) \ar[dl]|-{(1,4)} \\
(1,1)\ar[u]^{(2,1)} \ar[r]_{(2,1)} & (1,2) \ar[u]_{(2,1)}
}
\]
A valued quiver is {\em alternating} if its underlying ordinary
quiver is alternating. The {\em opposite} of a valued quiver
$Q$ has the opposite underlying quiver and the valuation defined
by
\[
v^{op}(\alpha^{op}) = (v(\alpha)_2, v(\alpha)_1)
\]
for each arrow $\alpha$ of $Q$. The {\em square product $Q \square Q'$}
of two alternating valued quivers is obtained from $Q\ten Q'$ by
replacing all full valued subquivers $\{i\}\otimes Q'$ and
$Q\otimes \{i'\}$ by their opposites, where $i$ runs through the
sinks of $Q$ and $i'$ through the sources of $Q'$.
Then the quivers $Q \boxten Q'$ and $Q\square Q'$ are
related by the same sequence of mutations as in
Lemma~\ref{lemma:triangle-mutates-to-square}.

\subsection{Restricted $Y$-patterns for pairs of arbitrary Dynkin diagrams}
Let $\Delta$ be a Dynkin diagram and $Q$ an alternating
valued quiver with underlying valued graph $\Delta$, cf. the end of
section~\ref{ss:valued-quivers}.
If $\Delta$ is not simply laced, we write
$Q$ as the valued quotient quiver of a valued quiver $\tilde{Q}$
with a group action by $G$ as in the following list where each
quiver $Q$ is followed by the corresponding quiver $\tilde{Q}$:
\begin{align}
\vec{B}_n &:\xymatrix@C=0.8cm{
1 \ar@{-}[r] & \ldots \ar[r] & (n-2) & (n-1) \ar[l] \ar[r]^-{(2,1)} & n} \\
\vec{A}_{2n-1} &: \xymatrix@R=0.1cm{1 \ar@{-}[r] & \ldots \ar[r] & (n-2)  & (n-1) \ar[l] \ar[rd] &  \\
                        &               &        &                      & n \\
          1'\ar@{-}[r] & \ldots \ar[r] & (n-2)' & (n-1)' \ar[l] \ar[ur] & } \\
\vec{C}_n &: \xymatrix@C=0.8cm{
1 \ar@{-}[r] & \ldots \ar[r] & (n-2) & (n-1) \ar[l] \ar[r]^-{(1,2)} & n} \\
\vec{D}_{n+1} &: \xymatrix@C=0.8cm@R=0.1cm{
             &               &       &                              & n \\
1 \ar@{-}[r] & \ldots \ar[r] & (n-2) & (n-1) \ar[l] \ar[ur] \ar[dr] & \\
             &               &       &                              & n'  }
\end{align}

\begin{align}
\vec{F}_4 &: \xymatrix{
1 \ar[r] & 2 & 3 \ar[l]_-{(2,1)} \ar[r] & 4 } \\
\vec{E}_6 &: \xymatrix@C=0.8cm@R=0.1cm{
1 \ar[r] & 2   &     &         &         \\
               &     & 3 \ar[ul] \ar[dl]\ar[r] & 4 \\
1' \ar[r] & 2' &     &         & } \\
\vec{G}_2 &: \xymatrix{ 1 \ar[r]^{(3,1)} & 2 } \\
\vec{D}_4 &: \xymatrix@R=0.1cm{ 1 \ar[rd] &  \\
           1' \ar[r] & 2 \\
           1'' \ar[ru] & }
\end{align}
If $\Delta$ is simply laced,
we consider $\tilde{Q}=Q$ with the trivial group action.
Notice that the Coxeter numbers of $\Delta$ and the underlying
diagram of $\tilde{Q}$ coincide. They are respectively equal
to $2n$, $2n$, $12$ and $6$. In all cases except $\vec{G}_2$,
the group acting is $\Z/2\Z$, for $\vec{G}_2$, it is $\Z/3\Z$.

Now let $\Delta'$ be another Dynkin diagram and $\tilde{Q}'$ a quiver
with a group action by $G'$ defined similarly. The group $G\times G'$
acts on the each of the products $\tilde{Q} \ten \tilde{Q}'$, $\tilde{Q} \square \tilde{Q}'$
and $\tilde{Q}\boxtimes \tilde{Q}'$ and the quotients are isomorphic
to the respective products of the valued quivers $Q$ and $Q'$.
For example, the triangle product of two copies of $\vec{B}_2$
leads to the following quiver $\tilde{Q}\boxtimes \tilde{Q}'$:
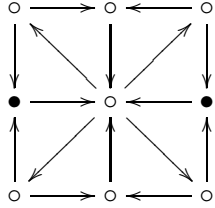
\[
\xymatrix{\circ \ar[r] \ar[d] & \circ \ar[d] & \circ \ar[l] \ar[d] \\
\bt \ar[r] & \circ \ar[ur] \ar[ul] \ar[dl] \ar[dr] & \bt \ar[l] \\
\circ \ar[u] \ar[r] & \circ \ar[u] & \circ \ar[l] \ar[u]
}
\]
The sequences of mutations $\mu_{\square}^{Q\square Q'}$ and $\mu_\boxtimes^{Q\boxtimes Q'}$ defined
in section~\ref{ss:reformulation} make sense for the valued quivers
$Q\square Q'$ and $Q\boxtimes Q'$. As in section~\ref{ss:reformulation},
one checks that the $Y$-system associated with $\Delta$ and $\Delta'$ is
periodic iff the restricted $Y$-pattern associated with $\mu_\boxtimes^{Q\boxtimes Q'}$
is periodic. Now the sequence of mutations $\mu_\boxtimes^{Q\boxtimes Q'}$ lifts
to the sequence of mutations $\mu_{\boxtimes}^{\tilde{Q} \boxtimes \tilde{Q}'}$ associated
with the simply laced quivers $\tilde{Q}$ and $\tilde{Q}'$. None of the mutations
in this latter sequence introduces $2$-cycles in the quotient quiver.
Thus, by lemma~\ref{lemma:mutation-versus-quotients}, the fact that the
restricted $Y$-pattern associated with $\mu_{\boxtimes}^{\tilde{Q} \boxtimes \tilde{Q}'}$ is
periodic with period dividing $h_{\tilde{Q}}+h_{\tilde{Q}'}$ implies that
the restricted $Y$-pattern associated with $\mu_{\boxtimes}^{Q\boxtimes Q'}$ is
periodic with period dividing $h_{\tilde{Q}}+h_{\tilde{Q}'}= h_\Delta+ h_{\Delta'}$.

\section{Effectiveness}
\label{s:effectiveness}

Let $\Delta$ and $\Delta'$ be simply laced Dynkin diagrams and $Q$ and $Q'$
alternating quivers with underlying graphs $\Delta$ and $\Delta'$.
In section~\ref{ss:reformulation}, we have seen that in order to write down
the explicit general solution of the $Y$-system associated
with $(\Delta,\Delta')$, it suffices to
write down the general seeds in the restricted $Y$-patterns
$\mathbf{y}_\square$ or indeed in $\mathbf{y}_\boxtimes$.
In Proposition~\ref{prop:Y-variables}, we have seen that the $Y$-variables
at a vertex of the regular tree are determined by the tropical $Y$-variables
and the $F$-polynomials. Thus, in order to write down the explicit
general solution of the $Y$-system, it suffices to write down
explicit expressions for the $F$-polynomials and the tropical
$Y$-variables at the vertices $\mu_\boxtimes^p(t_0)$
obtained from the initial vertex $t_0$ of the regular tree
by applying the sequence of mutations $\mu_\boxtimes^p$ for
all $p\in Z$. Such explicit expressions can easily be extracted
from the proof we gave, as we show now.

Let $k=\C$ and let $A=kQ$ and $A'=kQ'$ be the path algebras
of $Q$ and $Q'$. For a vertex $i$ of $Q$, we write $P_i=e_i A$
for the corresponding indecomposable projective $A$-module
and similarly $P_{i'}$ for a vertex $i'$ of $Q'$. Let $\cd$
be the derived category of the category $\mod(A\ten A')$
of finite-dimensional right modules over $A\ten A'$.
Consider the finite-dimensional algebra
\[
B= \bigoplus_{r\in\Z} \Hom_\cd(A\ten A', (\tau^{-r} A)\ten \tau^{-r}A') \ko
\]
where $\tau$ is the Auslander-Reiten translation (in the derived
category of right $A$-modules respectively $A'$-modules, \cf the proof
of Theorem~\ref{thm:order-Zamolodchikov}). Notice that since $A$ and $A'$ are
hereditary, it is easy to determine these translations.
By Theorem~\ref{thm:Hom-finite-cluster-category},
the algebra $B$ is isomorphic to $H^0(\Pi_3(A\ten A'))$ which,
by Proposition~\ref{prop:Calab-Yau-completion-via-quiver-with-potential},
is isomorphic to the Jacobian algebra of the quiver with
potential $(Q\boxtimes Q', W)$ for the potential $W$ constructed
in the Proposition. Notice that in the definition of $B$, the
summands indexed by the $r<0$ vanish (because then the
homology in degrees $\leq 0$ of $\tau^{-r}A$ and $\tau^{-r}A'$ vanishes).

Given a vertex $(i,i')$ of $Q\boxtimes Q'$ and an integer $p\in\Z$
we put
\[
M(i,i',p)=\bigoplus_{r\in\Z} \Hom_\cd(A\ten A', (\tau^{-r-p}\Sigma P_i)\ten (\tau^{-r} P_{i'})).
\]
This is a right $B$-module. It follows from Corollary~\ref{cor:decat-F-polynomials}
and section~\ref{ss:reduction-to-no-loops-and-T-periodic} that the $F$-polynomial
associated with $(i,i')$ at the vertex $\mu_\boxtimes^p(t_0)$ can be
expressed as
\[
F_{(i,i')}(\mu_\boxtimes^p(t_0))= \sum_e \chi(Gr_e(M(i,i',p))) y^e \ko
\]
where $e$ runs through the dimension vectors of submodules of $M(i,i',p)$.
Now we define another right $B$-module by
\[
N(i,i',p)=\bigoplus_{r\in\Z} \Hom_{\cd}(A\ten A', (\tau^{-r-p}P_i)\ten(\tau^{-r}P_{i'})).
\]
For vertices $(i,i')$ and $(j,j')$ of $Q\boxtimes Q'$, put
\[
c_{(i,i'),(j,j')}= \dim\Hom(N(i,i',-p),S_{(j,j')}) - \dim\Ext^1(N(i,i',-p),S_{(j,j')}).
\]
Then it follows from Corollary~\ref{cor:decat-tropical-Y-var} and
section~\ref{ss:reduction-to-no-loops-and-T-periodic} that we have
\[
\eta_{(i,i')}(\mu_\boxtimes^p(t_0))= \prod_{(j,j')} y_{(j,j')}^{c_{(i,i'),(j,j')}}.
\]


\begin{thebibliography}{10}

\bibitem{Amiot09}
Claire Amiot, \emph{Cluster categories for algebras of global dimension $2$ and
  quivers with potential}, Annales de l'institut {F}ourier \textbf{59} (2009),
  no.~6, 2525--2590.

\bibitem{AssemSimsonSkowronski06}
Ibrahim Assem, Daniel Simson, and Andrzej Skowro{\'n}ski, \emph{Elements of the
  representation theory of associative algebras. {V}ol. 1}, London Mathematical
  Society Student Texts, vol.~65, Cambridge University Press, Cambridge, 2006,
  Techniques of representation theory.

\bibitem{AuslanderReitenSmaloe95}
M.~Auslander, I.~Reiten, and S.~Smal\o, \emph{Representation theory of {Artin}
  algebras}, Cambridge Studies in Advanced Mathematics, vol.~36, Cambridge
  University Press, 1995 (English).

\bibitem{BuanIyamaReitenScott09}
Aslak Bakke~Buan, Osamu Iyama, Idun Reiten, and Jeanne Scott, \emph{Cluster
  structures for 2-{C}alabi-{Y}au categories and unipotent groups}, Compos.
  Math. \textbf{145} (2009), no.~4, 1035--1079.

\bibitem{BuanIyamaReitenSmith11}
Aslak Bakke~Buan, Osamu Iyama, Idun Reiten, and David Smith, \emph{Mutation of
  cluster-tilting objects and potentials}, Amer. J. Math. \textbf{133} (2011),
  no.~4, 835--887.

\bibitem{BuanMarsh06}
Aslak Bakke~Buan and Robert Marsh, \emph{Cluster-tilting theory}, Trends in
  representation theory of algebras and related topics, Contemp. Math., vol.
  406, Amer. Math. Soc., Providence, RI, 2006, pp.~1--30.

\bibitem{BuanMarshReinekeReitenTodorov06}
Aslak Bakke~Buan, Robert~J. Marsh, Markus Reineke, Idun Reiten, and Gordana
  Todorov, \emph{Tilting theory and cluster combinatorics}, Advances in
  Mathematics \textbf{204 (2)} (2006), 572--618.

\bibitem{BerensteinFominZelevinsky05}
Arkady Berenstein, Sergey Fomin, and Andrei Zelevinsky, \emph{Cluster algebras.
  {III}. {U}pper bounds and double {B}ruhat cells}, Duke Math. J. \textbf{126}
  (2005), no.~1, 1--52.

\bibitem{CalderoChapoton06}
Philippe Caldero and Fr{\'e}d{\'e}ric Chapoton, \emph{Cluster algebras as
  {H}all algebras of quiver representations}, Comment. Math. Helv. \textbf{81}
  (2006), no.~3, 595--616.

\bibitem{CalderoChapotonSchiffler06}
Philippe Caldero, Fr\'ed\'eric Chapoton, and Ralf Schiffler, \emph{Quivers with
  relations arising from clusters (${A}_n$ case)}, Trans. Amer. Math. Soc.
  \textbf{358} (2006), no.~3, 1347--1364.

\bibitem{Chapoton05}
Fr{\'e}d{\'e}ric Chapoton, \emph{Functional identities for the {R}ogers
  dilogarithm associated to cluster {$Y$}-systems}, Bull. London Math. Soc.
  \textbf{37} (2005), no.~5, 755--760.

\bibitem{DehyKeller08}
Raika Dehy and Bernhard Keller, \emph{On the combinatorics of rigid objects in
  $2$-{C}alabi-{Y}au categories}, International Mathematics Research Notices
  \textbf{2008} (2008), rnn029--17.

\bibitem{DerksenWeymanZelevinsky08}
Harm Derksen, Jerzy Weyman, and Andrei Zelevinsky, \emph{Quivers with
  potentials and their representations {I}: {Mutations}}, Selecta Mathematica
  \textbf{14} (2008), 59--119.

\bibitem{DerksenWeymanZelevinsky10}
\bysame, \emph{Quivers with potentials and their representations {II}:
  {Applications to cluster algebras}}, J.~Amer.~Math.~Soc. \textbf{23} (2010),
  749--790.

\bibitem{DiFrancescoKedem09}
Philippe Di~Francesco and Rinat Kedem, \emph{{$Q$}-systems as cluster algebras.
  {II}. {C}artan matrix of finite type and the polynomial property}, Lett.
  Math. Phys. \textbf{89} (2009), no.~3, 183--216.

\bibitem{DlabRingel74a}
Vlastimil Dlab and Claus~Michael Ringel, \emph{Representations of graphs and
  algebras}, Department of Mathematics, Carleton University, Ottawa, Ont.,
  1974, Carleton Mathematical Lecture Notes, No. 8.

\bibitem{Dupont08a}
G.~Dupont, \emph{An approach to non-simply laced cluster algebras}, J. Algebra
  \textbf{320} (2008), no.~4, 1626--1661.

\bibitem{FockGoncharov09}
Vladimir~V. Fock and Alexander~B. Goncharov, \emph{Cluster ensembles,
  quantization and the dilogarithm}, Annales scientifiques de l'ENS \textbf{42}
  (2009), no.~6, 865--930.

\bibitem{Fomin07}
Sergey Fomin, \emph{Cluster algebras portal},
  \verb"www.math.lsa.umich.edu/~fomin/"\verb"cluster.html".

\bibitem{FominZelevinsky02}
Sergey Fomin and Andrei Zelevinsky, \emph{Cluster algebras. {I}.
  {F}oundations}, J. Amer. Math. Soc. \textbf{15} (2002), no.~2, 497--529
  (electronic).

\bibitem{FominZelevinsky03}
\bysame, \emph{Cluster algebras. {II}. {F}inite type classification}, Invent.
  Math. \textbf{154} (2003), no.~1, 63--121.

\bibitem{FominZelevinsky03b}
\bysame, \emph{{$Y$}-systems and generalized associahedra}, Ann. of Math. (2)
  \textbf{158} (2003), no.~3, 977--1018.

\bibitem{FominZelevinsky07}
\bysame, \emph{Cluster algebras {IV}: {C}oefficients}, Compositio Mathematica
  \textbf{143} (2007), 112--164.

\bibitem{FrenkelSzenes95}
Edward Frenkel and Andr{\'a}s Szenes, \emph{Thermodynamic {B}ethe ansatz and
  dilogarithm identities. {I}}, Math. Res. Lett. \textbf{2} (1995), no.~6,
  677--693.

\bibitem{FuKeller10}
Changjian Fu and Bernhard Keller, \emph{On cluster algebras with coefficients
  and $2$-{C}alabi-{Y}au categories}, Trans. Amer. Math. Soc. \textbf{362}
  (2010), 859--895.

\bibitem{Gabriel72}
P.~Gabriel, \emph{Unzerlegbare {Darstellungen} {I}}, Manuscripta Math.
  \textbf{6} (1972), 71--103.

\bibitem{GabrielRoiter92}
P.~Gabriel and A.V. Roiter, \emph{Representations of finite-dimensional
  algebras}, Encyclopaedia Math. Sci., vol.~73, Springer--Verlag, 1992.

\bibitem{Gabriel80}
Peter Gabriel, \emph{Auslander-{R}eiten sequences and representation-finite
  algebras}, Representation theory, I (Proc. Workshop, Carleton Univ., Ottawa,
  Ont., 1979), Springer, Berlin, 1980, pp.~1--71.

\bibitem{GeissLeclercSchroeer08a}
Christof Gei\ss, Bernard Leclerc, and Jan Schr{\"o}er, \emph{Preprojective
  algebras and cluster algebras}, Trends in representation theory of algebras
  and related topics, EMS Ser. Congr. Rep., Eur. Math. Soc., Z\"urich, 2008,
  pp.~253--283.

\bibitem{Ginzburg06}
Victor Ginzburg, \emph{{Calabi-Yau} algebras}, arXiv:math/0612139v3 [math.AG].

\bibitem{GliozziTateo96}
F.~Gliozzi and R.~Tateo, \emph{Thermodynamic {B}ethe ansatz and three-fold
  triangulations}, Internat. J. Modern Phys. A \textbf{11} (1996), no.~22,
  4051--4064.

\bibitem{Happel87}
Dieter Happel, \emph{On the derived category of a finite-dimensional algebra},
  Comment. Math. Helv. \textbf{62} (1987), no.~3, 339--389.

\bibitem{Happel88}
\bysame, \emph{Triangulated categories in the representation theory of
  finite-dimensional algebras}, Cambridge University Press, Cambridge, 1988.

\bibitem{Henriques07}
Andr{\'e} Henriques, \emph{A periodicity theorem for the octahedron
  recurrence}, J. Algebraic Combin. \textbf{26} (2007), no.~1, 1--26.

\bibitem{HernandezLeclerc10}
David Hernandez and Bernard Leclerc, \emph{Cluster algebras and quantum affine
  algebras}, Duke Math. J. \textbf{154} (2010), no.~2, 265--341.

\bibitem{InoueIyamaKellerKunibaNakanishi10a}
Rei Inoue, Osamu Iyama, Bernhard Keller, Atsuo Kuniba, and Tomoki Nakanishi,
  \emph{Periodicities of {$T$} and {$Y$}-systems, dilogarithm identities, and
  cluster algebras {I}: type {$B_r$}}, arXiv:1001.1880 [math.QA], to appear in
  Publ. RIMS.

\bibitem{InoueIyamaKellerKunibaNakanishi10b}
\bysame, \emph{Periodicities of {$T$} and {$Y$}-systems, dilogarithm
  identities, and cluster algebras {II}: types {$C_r$, $F_4$, and $G_2$}},
  arXiv:1001.1881 [math.QA], to appear in Publ. RIMS.

\bibitem{InoueIyamaKunibaNakanishiSuzuki10}
Rei Inoue, Osamu Iyama, Atsuo Kuniba, Tomoki Nakanishi, and Junji Suzuki,
  \emph{Periodicities of {$T$}-systems and {$Y$}-systems}, Nagoya Math. J.
  \textbf{197} (2010), 59--174.

\bibitem{IyamaYoshino08}
Osamu Iyama and Yuji Yoshino, \emph{{Mutations in triangulated categories and
  rigid Cohen-Macaulay modules}}, Inv. Math. \textbf{172} (2008), 117--168.

\bibitem{Kac90}
Victor~G. Kac, \emph{Infinite-dimensional {L}ie algebras}, third ed., Cambridge
  University Press, Cambridge, 1990.

\bibitem{Kedem08}
Rinat Kedem, \emph{{$Q$}-systems as cluster algebras}, J. Phys. A \textbf{41}
  (2008), no.~19, 194011, 14.

\bibitem{Keller09b}
Bernhard Keller, \emph{Alg\`ebres amass\'ees et applications}, S\'eminaire
  Bourbaki, Expos\'e 1014, arXiv:0911.2903 [math.RA].

\bibitem{Keller08c}
\bysame, \emph{Cluster algebras, quiver representations and triangulated
  categories}, arXiv:0807.1960 [math.RT].

\bibitem{Keller93}
\bysame, \emph{A remark on tilting theory and {DG} algebras}, Manuscripta Math.
  \textbf{79} (1993), no.~3-4, 247--252.

\bibitem{Keller05}
\bysame, \emph{{On triangulated orbit categories}}, Doc. Math. \textbf{10}
  (2005), 551--581.

\bibitem{Keller07a}
\bysame, \emph{Derived categories and tilting}, Handbook of Tilting Theory, LMS
  Lecture Note Series, vol. 332, Cambridge Univ. Press, Cambridge, 2007,
  pp.~49--104.

\bibitem{Keller08d}
\bysame, \emph{Triangulated {C}alabi-{Y}au categories}, Trends in
  Representation Theory of Algebras (Zurich) (A.~Skowro\'nski, ed.), European
  Mathematical Society, 2008, pp.~467--489.

\bibitem{Keller10b}
\bysame, \emph{Cluster algebras, quiver representations and triangulated
  categories}, Triangulated categories (Thorsten Holm, Peter J{\o}rgensen, and
  Rapha\"el Rouquier, eds.), London Mathematical Society Lecture Note Series,
  vol. 375, Cambridge University Press, 2010, pp.~76--160.

\bibitem{Keller11b}
\bysame, \emph{Deformed {C}alabi--{Y}au completions}, Journal f{\"u}r die reine
  und angewandte Mathematik (Crelles Journal) \textbf{654} (2011), 125--180,
  with an appendix by Michel~Van den Bergh.

\bibitem{KellerReiten07}
Bernhard Keller and Idun Reiten, \emph{{Cluster-tilted algebras are Gorenstein
  and stably Calabi-Yau}}, Advances in Mathematics \textbf{211} (2007),
  123--151.

\bibitem{KellerYang11}
Bernhard Keller and Dong Yang, \emph{Derived equivalences from mutations of
  quivers with potential}, Advances in Mathematics \textbf{26} (2011),
  2118--2168.

\bibitem{KontsevichSoibelman08}
Maxim Kontsevich and Yan Soibelman, \emph{Stability structures,
  {D}onaldson-{T}homas invariants and cluster transformations}, arXiv:0811.2435
  [math.AG].

\bibitem{KontsevichSoibelman10a}
\bysame, \emph{Motivic {D}onaldson-{T}homas invariants: summary of results},
  Mirror symmetry and tropical geometry, Contemp. Math., vol. 527, Amer. Math.
  Soc., Providence, RI, 2010, pp.~55--89.

\bibitem{KunibaNakanishi92}
A.~Kuniba and T.~Nakanishi, \emph{Spectra in conformal field theories from the
  {R}ogers dilogarithm}, Modern Phys. Lett. A \textbf{7} (1992), no.~37,
  3487--3494.

\bibitem{KunibaNakanishiSuzuki94}
Atsuo Kuniba, Tomoki Nakanishi, and Junji Suzuki, \emph{Functional relations in
  solvable lattice models. {I}. {F}unctional relations and representation
  theory}, Internat. J. Modern Phys. A \textbf{9} (1994), no.~30, 5215--5266.

\bibitem{MarshReinekeZelevinsky03}
Robert Marsh, Markus Reineke, and Andrei Zelevinsky, \emph{Generalized
  associahedra via quiver representations}, Trans. Amer. Math. Soc.
  \textbf{355} (2003), no.~10, 4171--4186 (electronic).

\bibitem{MiyachiYekutieli01}
Jun-ichi Miyachi and Amnon Yekutieli, \emph{Derived {P}icard groups of
  finite-dimensional hereditary algebras}, Compositio Math. \textbf{129}
  (2001), no.~3, 341--368.

\bibitem{Nakanishi11c}
Tomoki Nakanishi, \emph{Dilogarithm identities for conformal field theories and
  cluster algebras: simply laced case}, Nagoya Math. J. \textbf{202} (2011),
  23--43. \MR{2804544}

\bibitem{Palu08a}
Yann Palu, \emph{Cluster characters for 2-{C}alabi-{Y}au triangulated
  categories}, Ann. Inst. Fourier (Grenoble) \textbf{58} (2008), no.~6,
  2221--2248.

\bibitem{RavaniniTateoValleriani93}
F.~Ravanini, A.~Valleriani, and R.~Tateo, \emph{Dynkin {TBA}s}, Internat. J.
  Modern Phys. A \textbf{8} (1993), no.~10, 1707--1727.

\bibitem{Reiten10}
Idun Reiten, \emph{Tilting theory and cluster algebras}, arXiv:1012.6014
  [math.RT].

\bibitem{Ringel84}
Claus~Michael Ringel, \emph{Tame algebras and integral quadratic forms},
  Lecture Notes in Mathematics, vol. 1099, Springer Verlag, 1984.

\bibitem{Ringel07}
\bysame, \emph{Some remarks concerning tilting modules and tilted algebras.
  {Origin. Relevance. Future.}}, Handbook of Tilting Theory, LMS Lecture Note
  Series, vol. 332, Cambridge Univ. Press, Cambridge, 2007, pp.~413--472.

\bibitem{Szenes09}
Andr{\'a}s Szenes, \emph{Periodicity of {Y}-systems and flat connections},
  Lett. Math. Phys. \textbf{89} (2009), no.~3, 217--230.

\bibitem{VandenBergh08}
Michel Van~den Bergh, \emph{The signs of {S}erre duality}, Appendix A to
  R.~Bocklandt, Graded Calabi-Yau algebras of dimension 3, Journal of Pure and
  Applied Algebra 212 (2008), 14--32.

\bibitem{Verdier96}
Jean-Louis Verdier, \emph{Des cat{\'e}gories d{\'e}riv{\'e}es des
  cat{\'e}gories ab{\'e}liennes}, Ast{\'e}risque, vol. 239, Soci{\'e}t{\'e}
  Math{\'e}matique de France, 1996 (French).

\bibitem{Volkov07}
Alexandre~Yu. Volkov, \emph{On the periodicity conjecture for {$Y$}-systems},
  Comm. Math. Phys. \textbf{276} (2007), no.~2, 509--517.

\bibitem{YangZelevinsky08}
Shih-Wei Yang and Andrei Zelevinsky, \emph{Cluster algebras of finite type via
  {C}oxeter elements and principal minors}, Transform. Groups \textbf{13}
  (2008), no.~3-4, 855--895.

\bibitem{Zamolodchikov91}
Al.~B. Zamolodchikov, \emph{On the thermodynamic {B}ethe ansatz equations for
  reflectionless {$ADE$} scattering theories}, Phys. Lett. B \textbf{253}
  (1991), no.~3-4, 391--394.

\bibitem{Zelevinsky07a}
Andrei Zelevinsky, \emph{What is a cluster algebra?}, Notices of the A.M.S.
  \textbf{54} (2007), no.~11, 1494--1495.

\end{thebibliography}


\def\cprime{$'$} \def\cprime{$'$}
\providecommand{\bysame}{\leavevmode\hbox to3em{\hrulefill}\thinspace}
\providecommand{\MR}{\relax\ifhmode\unskip\space\fi MR }
\providecommand{\MRhref}[2]{%
  \href{http://www.ams.org/mathscinet-getitem?mr=#1}{#2}
}
\providecommand{\href}[2]{#2}

\end{document}